\documentclass[12pt]{article}
\usepackage{amsmath,amsfonts,amsthm,amssymb,mathtools}
\usepackage[square]{natbib}
\usepackage[utf8]{inputenc}
\usepackage[english]{babel}
\usepackage[nottoc,notlof,notlot]{tocbibind}
\usepackage{algorithm2e}
\usepackage{pgfplots}
\usepackage{tikz}
\usepackage{subfig}
\usepackage{enumitem}
\usepackage[inner=2.6cm,outer=1.8cm]{geometry}
\allowdisplaybreaks

\newcommand{\R}{\mathbb{R}}
\newcommand{\E}{\mathbb{E}}
\newcommand{\C}{\mathbb{C}}

\newcommand{\N}{\mathbb{N}}
\newcommand{\Z}{\mathbb{Z}}
\newcommand{\scal}[2]{\left\langle #1,#2 \right\rangle}

\newcommand{\Tr}{\mbox{\rm Tr}}
\newcommand{\Id}{\mbox{\rm Id}}
\renewcommand{\Im}{\mathrm{Im}\,}
\renewcommand{\Re}{\mathrm{Re}\,}

\DeclareMathOperator{\Card}{Card}

\DeclareMathOperator{\phase}{phase}
\DeclareMathOperator{\Range}{Range}
\DeclareMathOperator{\Var}{Var}

\newcounter{Exocount}
\newcounter{Exocount_1}
\newcounter{Questcount}
\newcounter{Ssquestcount}
\newtheorem{thm}{Theorem}[section]
\newtheorem{lem}[thm]{Lemma}
\newtheorem*{lem*}{Lemma}
\newtheorem{prop}[thm]{Proposition}
\newtheorem{cor}[thm]{Corollary}
\newtheorem{conj}[thm]{Conjecture}
\newtheorem*{conj*}{Conjecture}
\newtheorem*{thm*}{Theorem}
\newtheorem*{prop*}{Proposition}

\title{Phase retrieval with random Gaussian sensing vectors by alternating projections}
\author{Irène Waldspurger
  \thanks{MIT Institute for Data, Systems and Society; e-mail address: \texttt{waldspur@math.mit.edu}.}
  }
\date{}

\begin{document}

\maketitle

\begin{abstract}
We consider a phase retrieval problem, where we want to reconstruct a $n$-dimensional vector from its phaseless scalar products with $m$ sensing vectors. We assume the sensing vectors to be independently sampled from complex normal distributions. We propose to solve this problem with the classical non-convex method of alternating projections. We show that, when $m\geq Cn$ for $C$ large enough, alternating projections succeed with high probability, provided that they are carefully initialized. We also show that there is a regime in which the stagnation points of the alternating projections method disappear, and the initialization procedure becomes useless. However, in this regime, $m$ needs to be of the order of $n^2$. Finally, we conjecture from our numerical experiments that, in the regime $m=O(n)$, there are stagnation points, but the size of their attraction basin is small if $m/n$ is large enough, so alternating projections can succeed with probability close to $1$ even with no special initialization.
\end{abstract}

\section{Introduction}

The problem of reconstructing a low-rank matrix from linear observations appears under many forms in the fields of inverse problems and machine learning. An important amount of work has thus been devoted to the design of reconstruction algorithms coming with provable reconstruction guarantees. The first algorithms of this kind relied mostly on convexification techniques. They tended to have a high recovery rate, but a possibly prohibitive computational complexity. As a result, a need has emerged to prove similar guarantees for algorithms based on non-convex formulations, which are generally much faster.


In this article, we consider a subclass of low-rank recovery problems: \emph{phase retrieval problems}. In the finite-dimensional setting, phase retrieval consists in recovering an unknown vector $x_0\in\C^n$ from $m$ phaseless linear measurements, of the form
\begin{equation*}
b_k=|\scal{a_k}{x_0}|,\quad\quad k=1,\dots,m,
\end{equation*}
where the \textit{sensing vectors} $a_k\in\C^n$ are known. Phaseless measurements do not allow to distinguish $x_0$ from $ux_0$, for $u\in\C,|u|=1$, so the goal is only to recover $x_0$ \textit{up to a global phase}. Motivations for studying these problems come in particular from optical imaging; see \citep*{schechtman} for a recent review. Phase retrieval problems can be seen as low-rank matrix recovery problems, because knowing $|\scal{a_k}{x_0}|$ amounts to knowing
\begin{equation*}
|\scal{a_k}{x_0}|^2=\Tr(a_ka_k^*x_0x_0^*),
\end{equation*}
so reconstructing $x_0$ is equivalent to:
\begin{align}\label{eq:matricial_form}
\mbox{Reconstruct }X_0\in\mathcal{S}_n(\C)&\mbox{ from }\{\Tr(a_ka_k^*X_0)\}_{k=1,\dots,m}\\
&\mbox{ such that }\mathrm{rank}(X_0)=1\nonumber.
\end{align}


The vector $x_0$ is uniquely determined by the $m$ phaseless measurements as soon as $m\gtrsim 4n$ \citep*{balan}; however, reconstructing it is a priori NP-hard \citep*{fickus}. The oldest reconstruction algorithms \citep*{gerchberg, fienup} were iterative: they started from a random initial guess of $x_0$, and tried to iteratively refine it by various heuristics. Although these algorithms are empirically seen to succeed in a number of cases, they can also get stuck in stagnation points, whose existence is due to the non-convexity of the problem.

To overcome these convergence problems, convexification methods have been introduced \citep*{chai,candes2}. These methods consider the matricial formulation \eqref{eq:matricial_form}, but replace the non-convex rank constraint by a more favorable convex constraint. They provably reconstruct the unknown vector $x_0$ with high probability if the sensing vectors $a_k$ are ``random enough'' \citep*{candes_li,candes_li2,gross}. Numerical experiments show that they also perform well on more structured, non-random phase retrieval problems \citep*{maxcut,sun_smith}.


Unfortunately, this good precision comes at a high computational cost: optimizing the $n\times n$ matrix $X_0$ is much slower that directly reconstructing the $n$-dimensional vector $x_0$. Consequently, convexification techniques are impractical when the dimension of $x_0$ exceeds a few hundred. Authors have thus recently begun to design fast non-convex algorithms, for which it is possible to establish similar reconstruction guarantees as for convexified algorithms. The methods that have been developed rely on the following two-step scheme:
\begin{enumerate}[label=(\arabic*)]
\item an initialization step, that returns a point close to the solution;\label{item:step1}
\item a gradient descent (possibly with additional refinements) over a well-chosen non-convex cost function.\label{item:step2}
\end{enumerate}
The intuitive reason why this scheme works is that the cost function, although globally non-convex, enjoys some good geometrical property in a neighborhood of the solution (like convexity or a weak form of it \citep*{white}). So, if the point returned by the initialization step belongs to this neighborhood, gradient descent converges to the true solution.

A preliminary form of this scheme appears in \citep*{netrapalli}, with an alternating minimization in step \ref{item:step2} instead of a gradient descent. Then, considering the cost function
\begin{equation}\label{eq:cost_L1}
L_1(x)=\sum_{k=1}^m\left(b_k^2-|\scal{a_k}{x}|^2\right)^2,
\end{equation}
\citep*{candes_wirtinger} proved the correctness of the two-step scheme, with high probability, in the regime $m=O(n\log n)$, for random independent Gaussian sensing vectors. In \citep*{candes_wirtinger2,kolte}, the same result was shown in the regime $m=O(n)$ for a slightly different cost function, with additional truncation steps. In \citep*{zhang}, it was extended to the following non-smooth cost function:
\begin{equation*}
L_2(x)=\sum_{k=1}^m\left(b_k-|\scal{a_k}{x}|\right)^2.
\end{equation*}
Additionally, \citet*{sun_qu_wright} have shown that, in the regime $m=O(n\log^3 n)$, the cost function \eqref{eq:cost_L1} actually has no ``bad critical point'', and the initialization step is not necessary: the gradient descent in step \ref{item:step2} converges to the global minimum of $L_1$, almost whatever initial point it starts from. These authors have also numerically observed that, in the regime $m=O(n)$, despite the potential presence of bad critical points, the gradient descent succeeds, with at least constant probability, starting from a random initialization.



For other low-rank recovery problems than phase retrieval, we refer for example to \citep*{sun_luo,ge_lee_ma} for matrix completion, to \citep*{tu,bhojanapalli} for the case where the measurement scheme obeys a Restricted Isometry Property, and to \citep*{bandeira_low_rank} for $\Z_2$ synchronization problems.



In the case of phase retrieval, the most recently introduced non-convex algorithms are optimal in terms of both statistical and computational complexity, up to multiplicative constants. However, there is still a need to understand whether their theoretical reconstruction guarantees can be extended to more general classes of algorithms, that would not exactly follow the above two-step scheme, but would be closer to the algorithms that are actually used in applications. This in particular implies to answer the following two questions:
\begin{itemize}
\item In Step \ref{item:step2}, can we replace the explicit minimization of a cost function by a ``less local'' search, like alternating projections \citep*{gerchberg} or Douglas-Rachford \citep*{bauschke}?
\item Is the initialization step \ref{item:step1} necessary, or can Step \ref{item:step2} converge to the global optimum even starting from a random initialization, at least in certain cases?
\end{itemize}




In this article, we answer the first question: we show that, in the optimal regime of $m=O(n)$ random independent Gaussian sensing vectors, replacing gradient descent with alternating projections yields exact recovery with high probability, and convergence occurs at a linear rate.
\begin{thm*}[See Corollary \ref{cor:global_convergence}]
There exist absolute constants $C_1,C_2,M>0$, $\delta\in]0;1[$ such that, if $m\geq Mn$ and the sensing vectors are independently chosen according to complex normal distributions, the sequence of iterates $(z_t)_{t\in \N}$ produced by the alternating projections method satisfies
\begin{equation*}
\forall t\in\N^*,\quad\quad \inf_{\phi\in\R}||e^{i\phi}x_0-z_t||\leq \delta^t||x_0||,
\end{equation*}
with probability at least
\begin{equation*}
1-C_1\exp(-C_2m),
\end{equation*}
provided that alternating projections are correctly initialized, for example with the method described in \citep*{candes_wirtinger2}.
\end{thm*}
\noindent
Alternating projections, introduced by \citet*{gerchberg}, is the most ancient algorithm for phase retrieval. It is an intuitive method, whose implementation is extremely simple, and with no parameter to choose or tune; it is thus widely used. In terms of complexity, it is slower, for general measurements, than the best non-convex methods by only a logarithmic factor in the precision. For more ``structured'' measurements (as in all applications that we know of), it is as fast (see Paragraph \ref{ss:complexity}).


We believe that the second question, about the necessity of the initialization step, is also important. In addition to being a natural theoretical question, it has practical consequences: the initialization procedure depends on the probability distribution of the sensing vectors, and, for some families of sensing vectors appearing in applications, we do not (yet) have a valid initialization procedure. We partially answer it in the case where the sensing vectors are independent and Gaussian, and reconstruction is done with alternating projections. We propose a description of when this method globally converges to the true solution, depending on the number of measurements and the initialization procedure. This description is summarized in Figure \ref{fig:global_image}.
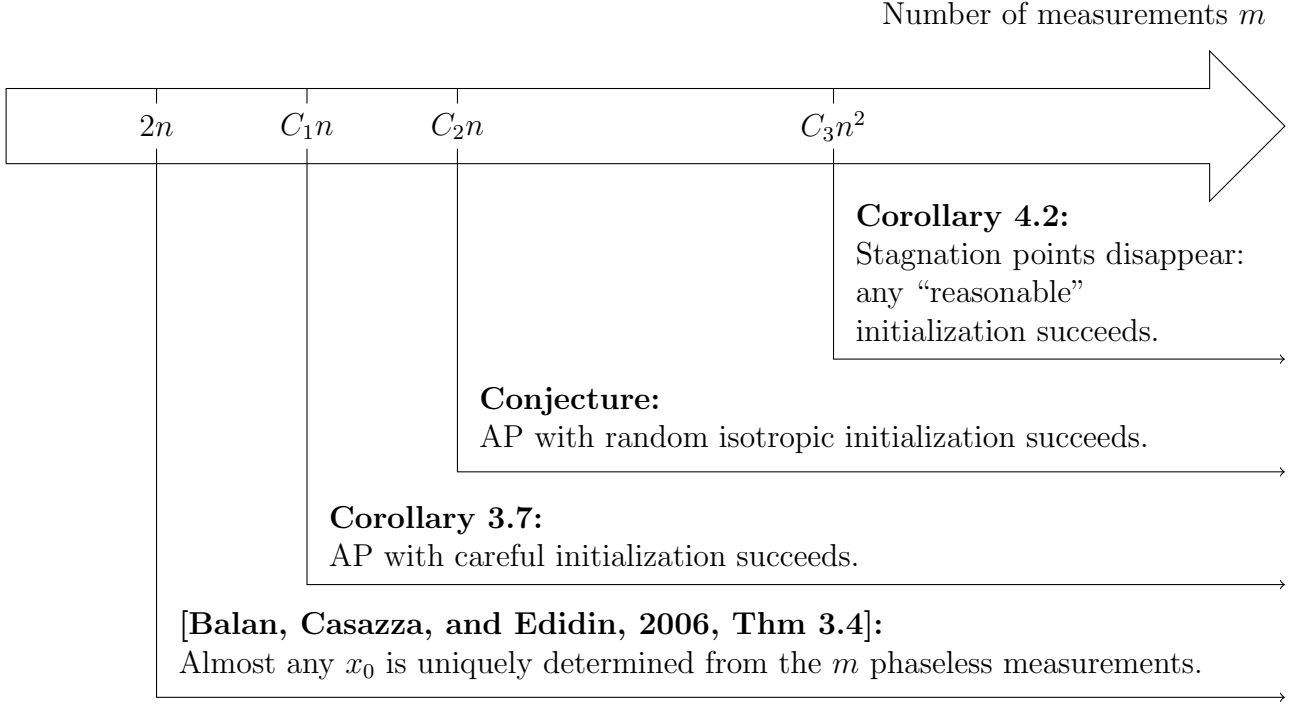
\begin{figure}[h]
\begin{tikzpicture}[scale=1]
  \centering
  \draw (11.5,2) node[right] {Number of measurements $m$};

  \draw (0,0) -- (16,0) -- (16,-0.5) -- (17,0.5) -- (16,1.5)
  -- (16,1) -- (0,1) -- (0,0);
  \draw (2,0.5) node {$2n$};
  \draw (2,1) -- (2,0.8);
  \draw (2.15,-7) node[above right, align=left] {\textbf{\citep*[Thm 3.4]{balan}:}\\
    Almost any $x_0$ is uniquely determined from the $m$ phaseless measurements.};
  \draw[->] (2,0.2) -- (2,-7.1) -- (17,-7.1);

  \draw (4,0.5) node {$C_1n$};
  \draw (4,1) -- (4,0.8);
  \draw (4.15,-5.5) node[above right, align=left] {\textbf{Corollary \ref{cor:global_convergence}:}\\
    AP with careful initialization succeeds.};
  \draw[->] (4,0.2) -- (4,-5.6) -- (17,-5.6);

  \draw (6,0.5) node {$C_2n$};
  \draw (6,1) -- (6,0.8);
  \draw (6.15,-4) node[above right, align=left] {\textbf{Conjecture:}\\
    AP with random isotropic initialization succeeds.};
  \draw[->] (6,0.2) -- (6,-4.1) -- (17,-4.1);

  \draw (11,0.5) node {$C_3n^2$};
  \draw (11,1) -- (11,0.8);
  \draw[->] (11,0.2) -- (11,-2.6) -- (17,-2.6);
  \draw (11.15,-2.5) node[above right, align=left] {\textbf{Corollary \ref{cor:global_convergence_without}:}\\
    Stagnation points disappear:\\ any ``reasonable''\\ initialization succeeds.};
\end{tikzpicture}
\caption{Schematic representation of the behavior of the alternating projections (AP) algorithm, as a function of the number of measurements $m$. All events happen only ``with high probability''.\label{fig:global_image}}
\end{figure}


As shown in the figure, there is a regime in which the stagnation points of the alternating projections routine disappear (except possibly on a ``small'' set that we define), and, with high probability, alternating projections converge starting from any initialization outside the small set. This regime is $m=O(n^2)$. Our numerical experiments clearly indicate that, below this regime, there are stagnation points. It is however possible that the attraction basin of the stagnation points is small: even in the regime $m=O(n)$, we numerically see that alternating projections, starting from a random isotropic initialization\footnote{By ``isotropic'', we mean that the law of the initial vector is invariant under linear unitary transformations.}, succeed with probability close to $1$ despite the presence of stagnation points. We leave this assertion as a conjecture.

\begin{thm*}[Informal, see Corollary \ref{cor:global_convergence_without}]
There exist $C_1,C_2,\gamma,M>0$, $\delta\in]0;1[$ such that, if $m\geq Mn^2$ and the sensing vectors are independently chosen according to complex normal distributions, with probability at least
\begin{equation*}
1-C_1\exp(-C_2 n),
\end{equation*}
the sequence of iterates $(z_t)_{t\in \N}$ produced by the alternating projections method satisfies
\begin{equation*}
\forall t\geq \gamma\log n,\quad\quad \inf_{\phi\in\R}||e^{i\phi}x_0-z_t||\leq \delta^{t-\gamma\log n}||x_0||,
\end{equation*}
starting from any initial point that does not belong to a small ``bad set''.
\end{thm*}

\begin{conj*}[See Conjecture \ref{conj:convergence_random}]
Let any $\epsilon>0$ be fixed. When $m\geq Cn$, for $C>0$ large enough, alternating projections, starting from a random isotropic initialization, converge to the true solution with probability at least $1-\epsilon$.
\end{conj*}

These theorem and conjecture are the parallels for alternating projections of the results and numerical observations obtained by \citet*{sun_qu_wright} for gradient descent over the cost function \eqref{eq:cost_L1}. The ``no stagnation point'' regime is much less favorable in the case of alternating projections than in the case of gradient descent: $m=O(n^2)\gg O(n\log^3 n)$. It could be due to the discontinuity of the alternating projections operator, but we have no evidence to support this fact.


\vskip 0.7cm

On the side of proof techniques, there has been a lot of work on the convergence of alternating projections in non-convex settings. Transversality arguments can be shown to prove, in certain cases, local convergence guarantees (``if the initial point is sufficiently close to the correct solution, alternating projections converge to this solution''). See for example \citep*{lewis,drusvyatskiy}. These arguments can be used in phase retrieval, and yield local convergence results for relatively general families of sensing vectors (not necessarily random) \citep*{noll,chen_fannjiang}. Unfortunately, they give no control on the convergence radius of the algorithm, so the obtained results have a mainly theoretical interest.


Bounding the convergence radius requires using the statistical properties of the sensing vectors. This was first attempted in \citep*{netrapalli}, where the authors proved the global convergence of a resampled version of the alternating projections algorithm. For a non resampled version, a preliminary result was given in \citep*{soltanolkotabi}. However, the bound on the convergence radius that underlies this result is small. As a consequence, global convergence is only proven for a suboptimal number of measurements ($m=O(n\log^2n)$), and with a complex initialization procedure.

A difficulty that we encounter is the fact that the alternating projections operator is not continuous. This difficulty also appears in the two recent articles \citep*{zhang,wang}, where the authors consider a gradient descent over a function whose gradient is not continuous. The proof that we give for our Theorem \ref{thm:global_convergence} follows a different path as theirs (it does not use a regularity condition); the statistical tools are however similar.


\vskip 0.7cm

The article is organized as follows. Section \ref{s:setup} precisely defines phase retrieval problems and the alternating projections algorithm. Section \ref{s:with_init} states and proves the first main result: the global convergence of alternating projections, with proper initialization, for $m=O(n)$ independent Gaussian measurements. Section \ref{s:without_init} proves the second main result: stagnation points disappear in the regime $m=O(n^2)$, making the initialization step useless. Finally, Section \ref{s:numerical} presents numerical results, and conjectures that the alternating projections algorithm can succeed without special initialization in the regime $m=O(n)$, despite the presence of stagnation points. All technical lemmas are deferred to the appendices.

\subsection{Notations}

For any $z\in\C$, $|z|$ is the modulus of $z$. We extend this notation to vectors: if $z\in\C^k$ for some $k\in\N^*$, then $|z|$ is the vector of $(\R^+)^k$ such that
\begin{equation*}
|z|_i=|z_i|,\quad\quad\forall i=1,\dots,k.
\end{equation*}
For any $z\in\C$, we set $E_{\phase}(z)$ to be the following subset of $\C$:
\begin{equation*}
\begin{array}{rll}
E_{\phase}(z)&=\left\{\frac{z}{|z|}\right\}&\mbox{ if }z\in\C-\{0\};\\
&=\{e^{i\phi},\phi\in\R\}&\mbox{ if }z=0.
\end{array}
\end{equation*}
We extend this definition to vectors $z\in\C^k$:
\begin{equation*}
E_{\phase}(z)=\prod_{i=1}^k E_{\phase}(z_i).
\end{equation*}
For any $z\in\C$, we define $\phase(z)$ by
\begin{equation*}
\begin{array}{rll}
\phase(z)&=\frac{z}{|z|}&\mbox{ if }z\in\C-\{0\};\\
&=1&\mbox{ if }z=0,
\end{array}
\end{equation*}
and extend this definition to vectors $z\in\C^k$, as for the modulus.

We denote by $\odot$ the pointwise product of vectors: for all $a,b\in\C^k$, $(a\odot b)$ is the vector of $\C^k$ such that
\begin{equation*}
(a\odot b)_i=a_ib_i,\quad\quad\forall i=1,\dots,k.
\end{equation*}
We define the operator norm of any matrix $A\in\C^{n_1\times n_2}$ by
\begin{equation*}
|||A||| = \sup_{v\in \C^{n_2}, ||v||=1}||Av||.
\end{equation*}
We denote by $A^\dag$ its Moore-Penrose pseudo-inverse. We note that $AA^\dag$ is the orthogonal projection onto $\Range(A)$. 

\section{Problem setup\label{s:setup}}

\subsection{Phase retrieval problem}

Les $n,m$ be positive integers. The goal of a phase retrieval problem is to reconstruct an unknown vector $x_0\in \C^n$ from $m$ measurements with a specific form.

We assume $a_1,\dots,a_m\in\C^n$ are given; they are called the \textit{sensing vectors}. We define a matrix $A\in\C^{m\times n}$ by
\begin{equation*}
A=\begin{pmatrix}a_1^*\\\vdots\\a_m^*\end{pmatrix}.
\end{equation*}
This matrix is called the \textit{measurement matrix}. The associated \textit{phase retrieval} problem is:
\begin{equation}\label{eq:problem_statement}
\mbox{reconstruct }x_0\mbox{ from }b\overset{def}{=}|Ax_0|.
\end{equation}
As the modulus is invariant to multiplication by unitary complex numbers, we can never hope to reconstruct $x_0$ better than \textit{up to multiplication by a global phase}. So, instead of exactly reconstructing $x_0$, we want to reconstruct $x_1$ such that
\begin{equation*}
x_1 = e^{i\phi}x_0,\quad\quad \mbox{for some }\phi\in\R.
\end{equation*}

In all this article, we assume the sensing vectors to be independent realizations of centered Gaussian variables with identity covariance:
\begin{equation}\label{eq:def_A}
(a_{i})_j\sim\mathcal{N}\left(0,\frac{1}{2}\right)
+\mathcal{N}\left(0,\frac{1}{2}\right)i,\quad\quad
\forall 1\leq i\leq m,1\leq j\leq n.
\end{equation}
The measurement matrix is in particular independent from $x_0$.

\citet*{balan} and \citet*{conca} have proved that, for \textit{generic} measurement matrices $A$, Problem \eqref{eq:problem_statement} always has a unique solution, up to a global phase, provided that $m\geq 4n-4$. In particular, with our measurement model \eqref{eq:def_A}, the reconstruction is guaranteed to be unique, with probability $1$, when $m\geq 4n-4$.

\subsection{Alternating projections}

The alternating projections method has been introduced for phase retrieval problems by \citet*{gerchberg}. It focuses on the reconstruction of $Ax_0$; if $A$ is injective, this then allows to recover $x_0$.

To reconstruct $Ax_0$, it is enough to find $z\in\C^m$ in the intersection of the following two sets.
\begin{enumerate}[label={(\arabic*)}]
\item $z\in \{z'\in\C^m,|z'|=b\}$;
\item $z\in\Range(A)$.
\end{enumerate}
Indeed, when the solution to Problem \eqref{eq:problem_statement} is unique, $Ax_0$ is the only element of $\C^m$ that simultaneously satisfies these two conditions (up to a global phase).

A natural heuristic to find such a $z$ is to pick any initial guess $z_0$, then to alternatively project it on the two constraint sets. In this context, we call \textit{projection} on a closed set $E\subset\C^m$ a function $P:\C^m\to E$ such that, for any $x\in\C^m$,
\begin{equation*}
||x-P(x)||=\inf_{e\in E}||x-e||.
\end{equation*}
The two sets defining constraints (1) and (2) admit projections with simple analytical expressions, which leads to the following formulas:
\begin{subequations}\label{eq:gs_image}
\begin{align}
y'_k&= b \odot \phase(y_k);& \mbox{(Projection onto set (1))}\\
y_{k+1}&= (AA^\dag) y'_k.& \mbox{(Projection onto set (2))}
\end{align}
\end{subequations}
If we define $z_k$ as the unique vector such that $y_k=Az_k$, an equivalent form of these equations is:
\begin{equation*}
z_{k+1} = A^\dag(b\odot\phase(Az_k)).
\end{equation*}

The hope is that the sequence $(y_k)_{k\in\N}$ converges towards $Ax_0$. Unfortunately, it can get stuck in \textit{stagnation points}. The following proposition (proven in Appendix \ref{s:stagnation_points}) characterizes these stagnation points.
\begin{prop}\label{prop:stagnation_points}
For any $y_0$, the sequence $(y_k)_{k\in\N}$ is bounded. Any accumulation point $y_\infty$ of $(y_k)_{k\in\N}$ satisfies the following property:
\begin{equation*}
\exists u\in E_{\phase}(y_\infty),\quad\quad
(AA^\dag)(b\odot u)=y_\infty.
\end{equation*}
In particular, if $y_\infty$ has no zero entry,
\begin{equation*}
(AA^\dag)(b\odot \phase(y_\infty))=y_\infty.
\end{equation*}
\end{prop}
Despite the relative simplicity of this characteristic property, it is extremely difficult to exactly compute the stagnation points, determine their attraction basin or avoid them when the algorithm happens to run into them.

The goal of this article is to show that, in certain settings, there are no stagnation points, or they can be avoided with a careful initialization procedure of the alternating projection routine.


\section{Alternating projections with good initialization\label{s:with_init}}

In this section, we prove the first of our two main results: in the regime $m=O(n)$, the method of alternating projections converges to the correct solution with high probability, if it is carefully initialized.

\subsection{Local convergence of alternating projections}

This paragraph proves the key result that we will need to establish our statement. This result is a local contraction property of the alternating projections operator $x\to A^\dag(b\odot\phase(Ax))$.

\begin{thm}\label{thm:local_convergence}
There exist $\epsilon,C_1,C_2,M>0$, and $\delta\in]0;1[$ such that, if $m\geq Mn$, then, with probability at least
\begin{equation*}
1-C_1\exp(-C_2m),
\end{equation*}
the following property holds: for any $x\in\C^n$ such that
\begin{equation*}\label{eq:hyp_x}
\inf_{\phi\in\R}||e^{i\phi}x_0-x||\leq \epsilon ||x_0||,
\end{equation*}
we have
\begin{equation}\label{eq:progres_lineaire}
\inf_{\phi\in\R}||e^{i\phi}x_0-A^\dag(b\odot\phase(Ax))||\leq \delta ||x_0-x||.
\end{equation}
\end{thm}

\begin{proof}
For any $x\in\C^n$, we can write $Ax$ as
\begin{equation}\label{eq:Ax_orth}
Ax = \lambda_x (Ax_0) + \mu_x v^x,
\end{equation}
where $\lambda_x\in\C,\mu_x\in\R^+$, and $v^x\in\Range(A)$ is a unitary vector orthogonal to $Ax_0$.

The following lemma is proven in Paragraph \ref{ss:diff_phase}.
\begin{lem}\label{lem:diff_phase}
For any $z_0,z\in\C$,
\begin{equation*}
|\phase(z_0+z)-\phase(z_0)| \leq 2. 1_{|z|\geq |z_0|/6} + \frac{6}{5}\left|\Im\left(\frac{z}{z_0}\right)\right|.
\end{equation*}
\end{lem}
So, for any $x\in\C^n$,
\begin{align*}
|\phase(\lambda_x)(Ax_0)_i&-(b\odot\phase(Ax))_i|\\
&=\left|\phase(\lambda_x)(Ax_0)_i-|Ax_0|_i\phase((Ax)_i)\right|\\
&=\left|\phase(\lambda_x)(Ax_0)_i-|Ax_0|_i\phase(\lambda_x(Ax_0)_i+\mu_x (v^x)_i)\right|\\
&= |Ax_0|_i\left|\phase(Ax_0)_i-\phase\left((Ax_0)_i+\frac{\mu_x}{\lambda_x} (v^x)_i\right) \right|\\
&\leq 2.|Ax_0|_i1_{|\mu_x/\lambda_x||v^x|_i\geq |Ax_0|_i/6} + \frac{6}{5}\left|\Im \left(\frac{\frac{\mu_x}{\lambda_x}v^x_i}{\phase((Ax_0)_i)}\right)\right|.
\end{align*}
As a consequence,
\begin{align}
||\phase(\lambda_x)(Ax_0)&-b\odot\phase(Ax)||\nonumber\\
&\leq \left|\left|
2.|Ax_0|\odot 1_{|\mu_x/\lambda_x||v^x|\geq |Ax_0|/6} + \frac{6}{5}\left|\Im \left(\left(\frac{\mu_x}{\lambda_x}v^x\right)\odot\overline{\phase(Ax_0)}\right)\right|\,
\right|\right|\nonumber\\
&\leq 2\left|\left|
|Ax_0|\odot 1_{6|\mu_x/\lambda_x||v^x|\geq |Ax_0|}\right|\right| + \frac{6}{5}\left|\left|\Im \left(\left(\frac{\mu_x}{\lambda_x}v^x\right)\odot\overline{\phase(Ax_0)}\right)\right|
\right|.\label{eq:error_sum}
\end{align}
Two technical lemmas allow us to upper bound the terms of this sum. The first one is proved in Paragraph \ref{ss:first_term}, the second one in Paragraph \ref{ss:second_term}.
\begin{lem}\label{lem:first_term}
For any $\eta>0$, there exists $C_1,C_2,M,\gamma>0$ such that the inequality
\begin{equation*}
||\,|Ax_0|\odot 1_{|v|\geq |Ax_0|}||\leq \eta ||v||
\end{equation*}
holds for any $v\in\Range(A)$ such that $||v||<\gamma ||Ax_0||$, with probability at least
\begin{equation*}
1-C_1\exp(-C_2m),
\end{equation*}
when $m\geq Mn$.
\end{lem}
\begin{lem}\label{lem:second_term}
For $M,C_1>0$ large enough, and $C_2>0$ small enough, when $m\geq M n$, the property
\begin{equation*}
||\Im(v\odot\overline{\phase(Ax_0)})||\leq \frac{4}{5}||v||
\end{equation*}
holds for any $v\in\Range(A)\cap \{Ax_0\}^\perp$, with probability at least
\begin{equation*}
1-C_1\exp(-C_2 m).
\end{equation*}
\end{lem}
Let us choose $\eta>0$ such that
\begin{equation*}
12\eta + \frac{24}{25}<1.
\end{equation*}
We define $\gamma>0$ as in Lemma \ref{lem:first_term}. The events described in Lemmas \ref{lem:first_term} and \ref{lem:second_term} hold with probability at least
\begin{equation*}
1-2C_1\exp(-C_2m).
\end{equation*}
When this happens, for all $x$ such that
\begin{equation*}
\left|\frac{\mu_x}{\lambda_x}\right|< \frac{\gamma}{6} \,||Ax_0||,
\end{equation*}
the terms in Equation \eqref{eq:error_sum} can be bounded as in the lemmas, because
\begin{equation*}
\left|\left| 6\frac{\mu_x}{\lambda_x}v^x
\right|\right|=6\left|\frac{\mu_x}{\lambda_x}\right|<\gamma||Ax_0||,
\end{equation*}
and $\frac{\mu_x}{\lambda_x}v^x\in \Range(A)\cap\{Ax_0\}^\perp$. So the following inequality holds:
\begin{align}\label{eq:consequence_lemmas}
||\phase(\lambda_x)(Ax_0)&-b\odot\phase(Ax)||
\leq \left(12 \eta 
+\frac{24}{25} \right)\left|\frac{\mu_x}{\lambda_x}\right|.
\end{align}
For any $x$ such that $\inf_{\phi\in\R}||e^{i\phi}x_0-x||\leq\epsilon||x_0||$, if we set $\epsilon^x=\inf_{\phi\in\R}\frac{||e^{i\phi}x_0-x||}{||x_0||}\leq\epsilon$,
\begin{equation*}
\inf_{\phi\in\R}||e^{i\phi}Ax_0-Ax||\leq \epsilon^x|||A|||\,||x_0||,
\end{equation*}
so, using Equation \eqref{eq:Ax_orth},
\begin{equation*}
\inf_{\phi\in\R}|e^{i\phi}-\lambda_x|^2||Ax_0||^2+|\mu_x|^2\leq (\epsilon^{x})^2|||A|||^2||x_0||^2,
\end{equation*}
which implies
\begin{gather*}
|\mu_x|\leq \epsilon^x|||A|||\,||x_0||;\\
|\lambda_x| \geq 1-\epsilon^x\frac{|||A|||\,||x_0||}{||Ax_0||}.
\end{gather*}
We can thus deduce from Equation \eqref{eq:consequence_lemmas} that, on an event of probability at least $1-2C_1\exp(-C_2m)$, as soon as $\inf_{\phi\in\R}||e^{i\phi}x_0-x||\leq \epsilon||x_0||$,
\begin{equation}\label{eq:maj_image}
||\phase(\lambda_x)(Ax_0)-b\odot\phase(Ax)||
\leq \left(12 \eta 
+\frac{24}{25} \right) \frac{\epsilon^x}{1-\epsilon^x\frac{|||A|||\,||x_0||}{||Ax_0||}}|||A|||\,||x_0||
\end{equation}
if
\begin{equation}\label{eq:condition_gamma}
\frac{\epsilon^x}{1-\epsilon^x\frac{|||A|||\,||x_0||}{||Ax_0||}}\frac{|||A|||\,||x_0||}{||Ax_0||}
< \frac{\gamma}{6}.
\end{equation}
Equation \eqref{eq:maj_image} implies in particular that, if Condition \eqref{eq:condition_gamma} holds,
\begin{equation}\label{eq:maj_signal}
||\phase(\lambda_x)x_0-A^\dag(b\odot\phase(Ax))||
\leq \left(12 \eta 
+\frac{24}{25} \right) \frac{\epsilon^x}{1-\epsilon^x\frac{|||A|||\,||x_0||}{||Ax_0||}}|||A^\dag|||\, |||A|||\,||x_0||.
\end{equation}

To conclude, it is enough to control the norms of $A$ and $A^\dag$ with the following classical result.
\begin{prop}[\citet*{davidson}, Thm II.13]\label{prop:davidson}
If $A$ is chosen according to Equation \eqref{eq:def_A}, then, for any $t$, with probability at least
\begin{equation*}
1-2\exp\left(-mt^2\right),
\end{equation*}
we have, for any $x\in\C^n$,
\begin{equation*}
\sqrt{m}\left(1-\sqrt{\frac{n}{m}}-t\right)||x||
\leq ||Ax||
\leq \sqrt{m}\left(1+\sqrt{\frac{n}{m}}+t\right)||x||.
\end{equation*}
\end{prop}
From this proposition, if we choose $\delta,M,t$ such that
\begin{gather*}
12\eta+\frac{24}{25} < \delta<1;\\
\epsilon<\min\left(\frac{1}{4},\frac{\gamma}{24},\frac{1}{2\delta}\left(\delta-12\eta-\frac{24}{25}\right)\right);\\
\frac{1+\sqrt{\frac{1}{M}}+t}{1-\sqrt{\frac{1}{M}}-t}\leq \min\left(2,\frac{(1-2\epsilon)\delta}{12\eta+\frac{24}{25}}\right).
\end{gather*}
we have, for $m\geq Mn$, with probability at least $1-2e^{-mt^2}$, as soon as $\epsilon^x\leq\epsilon$,
\begin{align*}
\frac{\epsilon^x}{1-\epsilon^x\frac{|||A|||\,||x_0||}{||Ax_0||}}
\frac{|||A|||\,||x_0||}{||Ax_0||}
&\leq \frac{\epsilon}{1-\epsilon\frac{|||A|||\,||x_0||}{||Ax_0||}}
\frac{1+\sqrt{\frac{1}{M}}+t}{1-\sqrt{\frac{1}{M}}-t}\\
&\leq \frac{2\epsilon}{1-\frac{1}{4}\frac{|||A|||\,||x_0||}{||Ax_0||}}
\\
&\leq \frac{2\epsilon}{1-\frac{1}{4}\frac{1+\sqrt{\frac{1}{M}}+t}{1-\sqrt{\frac{1}{M}}-t}}\\
&\leq 4\epsilon\\
&<\frac{\gamma}{6},
\end{align*}
and
\begin{align*}
\left(12 \eta 
+\frac{24}{25} \right) &\frac{\epsilon^x}{1-\epsilon^x\frac{|||A|||\,||x_0||}{||Ax_0||}}|||A^\dag|||\, |||A|||\,||x_0||\\
&\leq \left(12 \eta 
+\frac{24}{25} \right) \frac{\epsilon^x}{1-2\epsilon}|||A^\dag|||\, |||A|||\,||x_0||\\
&\leq \left(12 \eta 
+\frac{24}{25} \right) \frac{\epsilon^x}{1-2\epsilon}
\frac{1+\sqrt{\frac{1}{M}}+t}{1-\sqrt{\frac{1}{M}}-t}||x_0||\\
&\leq \delta \epsilon^x||x_0||.
\end{align*}

We now combine this with Equation \eqref{eq:maj_signal}: with probability at least
\begin{equation*}
1-2C_1\exp(-C_2m)-2\exp\left(-mt^2\right),
\end{equation*}
we have, for all $x$ such that $\inf_{\phi\in\R}||e^{i\phi}x_0-x||\leq\epsilon||x_0||$,
\begin{equation*}
||\phase(\lambda_x)x_0-A^\dag(b\odot \phase(Ax))||
\leq \delta\epsilon^x||x_0||=\delta \inf_{\phi\in\R}||e^{i\phi}x_0-x||.
\end{equation*}

\end{proof}

\subsection{Global convergence}

In the last paragraph, we have seen that the alternating projections operator is contractive, with high probability, in an $\epsilon ||x_0||$-neighborhood of the solution $x_0$. This implies that, if the starting point of alternating projections is at distance at most $\epsilon||x_0||$ from $x_0$, alternating projections converge to $x_0$. So if we have a way to find such an initial point, we obtain a globally convergent algorithm.

Several initialization methods have been proposed that achieve the precision we need with an optimal number of measurements, that is $m=O(n)$. Let us mention the truncated spectral initialization by \citet*{candes_wirtinger2} (improving upon the slightly suboptimal spectral initializations introduced by \citet*{netrapalli} and \citet*{candes_wirtinger}), the null initialization by \citet*{chen_fannjiang} and the method described by \citet*{gao_xu}. All these methods consist in computing the largest or smallest eigenvector of
\begin{equation*}
\sum_{i=1}^m\alpha_i a_ia_i^*,
\end{equation*}
where the $\alpha_1,\dots,\alpha_m$ are carefully chosen coefficients, that depend only on $b$.

The method of \citep*{candes_wirtinger2}, for example, has the following guarantees.
\begin{thm}[Proposition 3 of \citep*{candes_wirtinger2}]\label{thm:guarantee_init}
Let $\epsilon>0$ be fixed.

We define $z$ as the main eigenvector of
\begin{equation}
\frac{1}{m}\sum_{i=1}^m|a_i^*x_0|^2 a_ia_i^* 1_{|a_i^*x_0|^2\leq \frac{9}{m}\sum_{j=1}^m|a^*_ix_0|^2}.\label{eq:init_matrix}
\end{equation}
There exist $C_1,C_2,M>0$ such that, with probability at least
\begin{equation*}
1-C_1\exp(-C_2 m),
\end{equation*}
the vector $z$ obeys
\begin{equation*}
\inf_{\phi\in\R,\lambda\in\R^*_+}||e^{i\phi}x_0-\lambda z||\leq \epsilon ||x_0||,
\end{equation*}
provided that $m\geq M n$.
\end{thm}
Combining this initialization procedure with alternating projections, we get Algorithm \ref{alg:algo_complet}. As shown by the following corollary, it converges towards the correct solution, at a linear rate, with high probability, for $m=O(n)$.
\begin{algorithm}
\SetKwInOut{Input}{Input}
\SetKwInOut{Output}{Output}
\Input{$A\in\C^{m\times n},b=|Ax_0|\in\R^m,T\in\N^*$.}
\BlankLine
\textbf{Initialization:} set $z_0$ to be the main eigenvector of the matrix in Equation \eqref{eq:init_matrix}.

\For{$t=1$ \KwTo $T$}{
Set $z_{t}\leftarrow A^\dag(b\odot\phase(Az_{t-1}))$.}
\BlankLine
\Output{$z_T$.}
\BlankLine
\caption{Alternating projections with truncated spectral initialization\label{alg:algo_complet}}
\end{algorithm}
\begin{cor}\label{cor:global_convergence}
There exist $C_1,C_2,M>0,\delta\in]0;1[$ such that, with probability at least
\begin{equation*}
1-C_1\exp(-C_2m),
\end{equation*}
Algorithm \ref{alg:algo_complet} satisfies
\begin{equation}\label{eq:global_convergence}
\forall t\in\N^*,\quad\quad
\inf_{\phi\in\R}||e^{i\phi}x_0-z_t||\leq \delta^t ||x_0||,
\end{equation}
provided that $m\geq Mn$.
\end{cor}
\begin{proof}
Let us fix $\epsilon,\delta\in]0;1[$ as in Theorem \ref{thm:local_convergence}. Let us assume that the properties described in Theorems \ref{thm:local_convergence} and \ref{thm:guarantee_init} hold; it happens on an event of probability at least
\begin{equation*}
1-C_1\exp(-C_2m),
\end{equation*}
provided that $m\geq Mn$, for some constants $C_1,C_2,M>0$.

Let us prove that, on this event, Equation \eqref{eq:global_convergence} also holds.

We proceed by recursion. From Theorem \ref{thm:guarantee_init}, there exist $\phi\in\R,\lambda\in\R^*_+$ such that
\begin{equation*}
||e^{i\phi}x_0-\lambda z_0||\leq\epsilon||x_0||.
\end{equation*}
So, from Theorem \ref{thm:local_convergence}, applied to $x=\lambda z_0$,
\begin{align*}
\inf_{\phi\in\R}||e^{i\phi}x_0-z_1||&
=\inf_{\phi\in\R}||e^{i\phi}x_0-A^\dag(b\odot \phase(z_0))||\\
=\inf_{\phi\in\R}||e^{i\phi}x_0-A^\dag(b\odot \phase(\lambda z_0))||\\
&\leq \delta\inf_{\phi\in\R} ||e^{i\phi}x_0-\lambda z_0||\\
&\leq \epsilon \delta ||x_0||.
\end{align*}
This proves Equation \eqref{eq:global_convergence} for $t=1$.

The same reasoning can be reapplied to also prove the equation for $t=2,3,\dots$.
\end{proof}

\subsection{Complexity\label{ss:complexity}}

Let $\eta>0$ be the relative precision that we want to achieve:
\begin{equation*}
\inf_{\phi\in\R}||e^{i\phi}x_0-z_T||\leq \eta||x_0||.
\end{equation*}
Let us compute the number of operations that Algorithm \ref{alg:algo_complet} requires to reach this precision.

 The main eigenvector of the matrix defined in Equation \eqref{eq:init_matrix} can be computed - up to precision $\eta$ - in approximately $O(\log(1/\eta)+\log(n))$ power iterations. Each power iteration is essentially a matrix-vector multiplication, and thus requires $O(mn)$ operations.\footnote{These matrix-vector multiplications can be computed without forming the whole matrix (which would require $O(mn^2)$ operations), because this matrix factorizes as
\begin{equation*}
\frac{1}{m}A^* \mathrm{Diag}(|Ax_0|^2\odot I) A,
\end{equation*}
where $I\in\R^m$ is such that $\forall i\leq m,I_i=1_{|A_ix_0|^2\leq\frac{9}{m}\sum_{j=1}^m|A_ix_0|^2}$.}
As a consequence, the complexity of the initialization is
\begin{equation*}
O(mn\left(\log(1/\eta)+\log(n)\right)).
\end{equation*}
Then, at each step of the \texttt{for} loop, the most costly operation is the multiplication by $A^\dag$. When performed with the conjugate gradient method, it requires $O(mn\log(1/\eta))$ operations. To reach a precision equal to $\eta$, we need to perform $O(\log(1/\eta))$ iterations of the loop. So the total complexity of Algorithm \ref{alg:algo_complet} is
\begin{equation*}
O(mn\left(\log^2(1/\eta)+\log(n)\right)).
\end{equation*}

Let us mention that, when $A$ has a special structure, there may exist fast algorithms for the multiplication by $A$ and the orthogonal projection onto $\Range(A)$. In the case of masked Fourier measurements considered in \citep*{candes_li2}, for example, assuming that our convergence theorem still holds, despite the non-Gaussianity of the measurements, the complexity of each of these operations reduces to $O(m\log n)$, yielding a global complexity of
\begin{equation*}
O(m\log(n)(\log(1/\eta)+\log(n))).
\end{equation*}
The complexity is then almost linear in the number of measurements.

\begin{figure}
\centering
\begin{tabular}{|c|c|c|}
\hline
&Alternating projections&Truncated Wirtinger flow\\\hline
Unstructured case&$O\left(mn\left(\log^2(1/\eta)+\log(n)\right)\right)$
&$O\left(mn\left(\log(1/\eta)+\log(n)\right)\right)$\\\hline
Fourier masks&$O\left(m\log(n)\left(\log(1/\eta)+\log(n)\right)\right)$
&$O\left(m\log(n)\left(\log(1/\eta)+\log(n)\right)\right)$\\\hline
\end{tabular}
\caption{Complexity of alternating projections with initialization, and truncated Wirtinger flow.\label{fig:complexity}}
\end{figure}

As a comparison, Truncated Wirtinger flow, which is currently the most efficient known method for phase retrieval from Gaussian measurements, has an identical complexity, up to a $\log(1/\eta)$ factor in the unstructured case (see Figure \ref{fig:complexity}).

\section{Alternating projections without good initialization\label{s:without_init}}

\subsection{Main result}

In this section, we assume that the number of measurements is quadratic in $n$ instead of linear (that is $m\geq Mn^2$, for $M$ large enough). In this setting, we show that any initialization vector $x$, unless it is almost orthogonal to the ground truth $x_0$, yields perfect recovery when provided to the alternating projection routine. This in particular proves that, in this regime, there is no stagnation point (unless possibly among the vectors almost orthogonal to $x_0$).

The convergence rate is almost as good as in the case where a good initialization is provided: after $O(\log n)$ iterations, it becomes linear.

We say that a vector $x\in\C^n$ is \textit{not almost orthogonal} to $x_0$ if
\begin{equation*}
\mu \frac{||x_0||\,||x||}{\sqrt{n}}\leq |\scal{x_0}{x}|,
\end{equation*}
for some fixed constant $\mu>0$. In what follows, we assume $\mu=1$, but it is only to simplify the notations; the same result would hold for any value of $\mu$.

We remark that, in the unit sphere, the proportion (in terms of volume) of vectors that are almost orthogonal to $x_0$ goes to a constant depending on $\mu$ when $n$ goes to $+\infty$. This constant can be arbitrarily small if $\mu$ is small. As a consequence, if we choose $x\in\C^n$ according to an isotropic probability law, the probability that it is almost orthogonal to $x_0$ can be arbitrarily small.

To prove global convergence, we first need to understand what happens when we apply one iteration of the alternating projections routine to some vector $x$. We only consider vectors $x$ that are not almost orthogonal to $x_0$. We also do not consider vectors that are very close to $x_0$: these vectors are already taken care of by Theorem \ref{thm:local_convergence}.

\begin{thm}\label{thm:global_convergence}
For any $\epsilon>0$, there exist $C_1,C_2,M,\delta>0$ such that, if $m\geq Mn^2$, then, with probability at least
\begin{equation*}
1-C_1\exp(-C_2m^{1/8}),
\end{equation*}
the following property holds: for any $x\in\C^n$ such that
\begin{equation}\label{eq:global_cond}
\frac{||x_0||\,||x||}{\sqrt{n}}\leq
|\scal{x_0}{x}|\leq (1-\epsilon)||x_0||\,||x||,
\end{equation}
we have
\begin{equation}\label{eq:global_prop}
\frac{|\scal{x_0}{A^\dag(b\odot\phase(Ax))}|}{||x_0||\,||A^\dag(b\odot\phase(Ax))||} \geq
(1+\delta)\frac{|\scal{x_0}{x}|}{||x_0||\,||x||}.
\end{equation}
\end{thm}

Before proving this theorem, let us establish its main consequence : the global convergence of alternating projections starting from any initial point that is not almost orthogonal to $x_0$. The algorithm is summarized in Algorithm \ref{alg:algo_without_init} and global convergence is proven in Corollary \ref{cor:global_convergence_without}.
\begin{algorithm}
\SetKwInOut{Input}{Input}
\SetKwInOut{Output}{Output}
\Input{$A\in\C^{m\times n},b=|Ax_0|\in\R^m,T\in\N^*$, any $x\in\C^n$ not almost orthogonal to $x_0$.}
\BlankLine
\textbf{Initialization:} set $z_0=x$.

\For{$t=1$ \KwTo $T$}{
Set $z_{t}\leftarrow A^\dag(b\odot\phase(Az_{t-1}))$.}
\BlankLine
\Output{$z_T$.}
\BlankLine
\caption{Alternating projections without good initialization\label{alg:algo_without_init}}
\end{algorithm}

\begin{cor}\label{cor:global_convergence_without}
There exist $C_1,C_2,\gamma,M>0,\Delta\in]0;1[$ such that, with probability at least
\begin{equation*}
1-C_1\exp(-C_2n),
\end{equation*}
Algorithm \ref{alg:algo_without_init} satisfies:
\begin{equation}\label{eq:conv_rate}
\forall t\geq \gamma\log n,\quad\quad
\inf_{\phi\in\R}||e^{i\phi}x_0-z_t|| \leq \Delta^{t-\gamma\log n} ||x_0||,
\end{equation}
provided that $m\geq Mn^2$.
\end{cor}

\begin{proof}
From Theorem \ref{thm:local_convergence}, there exist $C_1^{(1)},C_2^{(1)},\epsilon^{(1)},M^{(1)}>0$ such that, if $m\geq M^{(1)}n$, then, with probability at least
\begin{equation*}
1-C_1^{(1)}\exp(-C_2^{(1)}m),
\end{equation*}
the following property holds: any $z\in\C^n$ such that $\inf_{\phi\in\R}||e^{i\phi}x_0-z||\leq\epsilon^{(1)}||x_0||$ satisfies
\begin{equation}
\inf_{\phi\in\R}||e^{i\phi}x_0-A^\dag(b\odot\phase(Az))||\leq\delta^{(1)}\inf_{\phi\in\R}||e^{i\phi}x_0-z||,\label{eq:local_interm}
\end{equation}
for some absolute constant $\delta^{(1)}\in]0;1[$. In the following, we assume that this event is realized.

We now use Theorem \ref{thm:global_convergence}, for $\epsilon={\epsilon^{(1)}}^2/2$. Let $C_1,C_2,M,\delta>0$ be defined as in this theorem. We assume that the event described in the theorem is realized, which happens with probability at least $1-C_1\exp(-C_2n)$.

We consider the sequence $(z_t)_{t\geq 0}$ defined in Algorithm \ref{alg:algo_without_init}, and distinguish two cases.

First, if the initial point $z_0=x$ is such that
\begin{equation*}
|\scal{x_0}{x}|>(1-\epsilon)||x_0||\,||x||,
\end{equation*}
then, setting $x'=\frac{||x_0||}{||x||}x$,
\begin{align*}
\inf_{\phi\in\R}||e^{i\phi}x_0-x'||&=\sqrt{||x_0||^2+||x'||^2-2|\scal{x_0}{x'}|}\\
&<||x_0||\sqrt{2\epsilon}\\
&=\epsilon^{(1)}||x_0||.
\end{align*}
We can thus proceed by recursion, as in the proof of Corollary \ref{cor:global_convergence}, to show that:
\begin{equation}\label{eq:local_combined_interm}
\forall t\in\N^*,\quad\quad
\inf_{\phi\in\R}||e^{i\phi}x_0-z_t||\leq (\delta^{(1)})^t\epsilon^{(1)}||x_0||.
\end{equation}
So Equation \eqref{eq:conv_rate} is satisfied, provided that we have chosen $\Delta\geq \delta^{(1)}$.

Second, we consider the case where the initial point $z_0=x$ is such that
\begin{equation*}
|\scal{x_0}{x}|\leq (1-\epsilon)||x_0||\,||x||.
\end{equation*}
Let then $\mathcal{T}$ be the smallest index $t$ such that the following inequality is not satisfied:
\begin{equation}
\frac{||x_0||\,||z_t||}{\sqrt{n}}\leq |\scal{x_0}{z_t}|\leq
(1-\epsilon) ||x_0||\,||z_t||.\label{eq:def_mathcal_T}
\end{equation}
As $z_0=x$ is not almost orthogonal to $x_0$, we must have $\mathcal{T}\geq 1$. For any $t=0,\dots,\mathcal{T}-1$, Equation \eqref{eq:global_prop} of Theorem \ref{thm:global_convergence} ensures that
\begin{equation}\label{eq:scal_growth}
\frac{|\scal{x_0}{z_{t+1}}|}{||x_0||\,||z_{t+1}||}
\geq (1+\delta) \frac{|\scal{x_0}{z_{t}}|}{||x_0||\,||z_{t}||}.
\end{equation}
In particular,
\begin{equation*}
\frac{|\scal{x_0}{z_{\mathcal{T}}}|}{||x_0||\,||z_{\mathcal{T}}||}
\geq \frac{|\scal{x_0}{z_{0}}|}{||x_0||\,||z_{0}||}\geq \frac{1}{\sqrt{n}}.
\end{equation*}
As Equation \eqref{eq:def_mathcal_T} is not satisfied, it means that
\begin{equation*}
|\scal{x_0}{z_{\mathcal{T}}}|>(1-\epsilon)||x_0||\,||z_{\mathcal{T}}||.
\end{equation*}
We can now apply the same reasoning as the one that led to Equation \eqref{eq:local_combined_interm}, and get
\begin{equation*}
\forall t\geq \mathcal{T}+1,\quad\quad
\inf_{\phi\in\R}||e^{i\phi}x_0-z_t||\leq (\delta^{(1)})^{t-\mathcal{T}}\epsilon^{(1)}||x_0||.
\end{equation*}
This implies Equation \eqref{eq:conv_rate}, provided that $\mathcal{T}\leq \gamma\log n$ for some absolute constant $\gamma$. From Equation \eqref{eq:scal_growth} and the fact that $z_0$ is not almost orthogonal to $x_0$,
\begin{equation*}
\frac{|\scal{x_0}{z_{\mathcal{T}-1}}|}{||x_0||\,||z_{\mathcal{T}-1}||}
\geq \frac{(1+\delta)^{\mathcal{T}-1}}{\sqrt{n}}.
\end{equation*}
As $\mathcal{T}-1$ satisfies Equation \eqref{eq:def_mathcal_T}, we must have
\begin{gather*}
\frac{(1+\delta)^{\mathcal{T}-1}}{\sqrt{n}}\leq 1-\epsilon\leq 1;\\
\Rightarrow\quad\quad
\mathcal{T}\leq 1+\frac{\log n }{2 \log(1+\delta)}.
\end{gather*}
And this expression can be bounded by $\gamma \log n$, for some $\gamma>0$ independent from $n$.

So we have shown that Equation \eqref{eq:conv_rate} holds when the events described in Theorems \ref{thm:local_convergence} and \ref{thm:global_convergence} happen. When $m\geq \max(M,M^{(1)})n^2$, this occurs with probability at least
\begin{equation*}
1-C_1^{(1)}\exp(-C_2^{(1)}m)-C_1\exp(-C_2n)
\geq 1 - (C_1+C_1^{(1)})\exp(-\min(C_2^{(1)},C_2)n).
\end{equation*}

\end{proof}

\subsection{Proof of Theorem \ref{thm:global_convergence}}

\begin{proof}[Proof of Theorem \ref{thm:global_convergence}]
We will actually consider a variant of Equation \eqref{eq:global_prop}, in the ``image domain'', that is in $\C^m$ instead of $\C^n$. This variant is easier to analyze and, according to the following lemma (proven in Paragraph \ref{ss:global_intro}), it implies Equation \eqref{eq:global_prop}.
\begin{lem}\label{lem:global_intro}
To prove Theorem \ref{thm:global_convergence}, it is enough to prove that there exist $C_1,C_2,M,\delta>0$ such that, if $m\geq Mn^2$, then, with probability at least $1-C_1\exp(-C_2m^{1/8})$, the property
\begin{equation}
|\scal{Ax_0}{b\odot\phase(Ax)}| \geq (1+\delta) m \frac{||x_0||}{||x||}|\scal{x_0}{x}|
\label{eq:scal_augmente}
\end{equation}
holds for any $x\in\C^n$ verifying Condition \eqref{eq:global_cond}.
\end{lem}

The proof of Equation \eqref{eq:scal_augmente} is in two parts. We first prove (Lemma \ref{lem:net}) that this equation holds (with high probability) for all $x$ belonging to a net with very small spacing. This part is the most technical: a direct union bound, that does not take advantage of the correlation between the vectors of the net, is not sufficient. We use a chaining argument instead. The detailed proof is in Paragraph \ref{ss:net}.

In a second part (Lemma \ref{lem:inside_net}), we prove that, with high probability, for any $x$ and $y$ very close, $|\scal{Ax_0}{b\odot\phase(Ax)}-\scal{Ax_0}{b\odot\phase(y)}|$ is small. This allows us to extend the inequality proven for vectors of the net to all vectors. This result is a consequence of two facts: first, the phase is a Lipschitz function outside any neighborhood of zero. Second, with high probability, for any $x$ and $y$, the vectors $Ax$ and $Ay$ have few entries that are close to zero. The detailed proof is in Paragraph \ref{ss:inside_net}.

\begin{lem}\label{lem:net}
For any $n\in\N^*$, we set
\begin{equation*}
\mathcal{E}_n=\left\{x\in\C^n, ||x||=1\mbox{ and }\frac{||x_0||\,||x||}{\sqrt{n}}\leq
|\scal{x_0}{x}|\leq (1-\epsilon)||x_0||\,||x|| \right\}.
\end{equation*}
Let $\alpha$ be any positive number.

There exist $c,C_1,C_2,M,\delta>0$ and, for any $n\in\N^*$, a $c m^{-\alpha}$-net $\mathcal{N}_n$ of $\mathcal{E}_n$ such that, when $m\geq Mn^2$, with probability at least
\begin{equation*}
1-C_1\exp(-C_2m^{1/2}),
\end{equation*}
the following property holds: for any $x\in\mathcal{N}_n$,
\begin{equation*}
|\scal{Ax_0}{b\odot\phase(Ax)}| \geq (1+\delta) m\frac{||x_0||}{||x||}|\scal{x_0}{x}|.
\end{equation*}
\end{lem}

\begin{lem}\label{lem:inside_net}
For any $c>0$, there exist $C_1,C_2,C_3>0$ such that, with probability at least
\begin{equation*}
1-C_1\exp(-C_2 m^{1/8}),
\end{equation*}
the following property holds for any unit-normed $x,y\in\C^n$, when $m\geq 2n^2$:
\begin{equation*}
|\scal{Ax_0}{b\odot\phase(Ax)}-\scal{Ax_0}{b\odot\phase(Ay)}|
\leq C_3||x_0||^2 n m^{1/4}
\quad\mbox{if }||x-y||\leq cm^{-7/2}.
\end{equation*}
\end{lem}

To conclude, we apply Lemma \ref{lem:net} with $\alpha=7/2$. We define $c,C_1,C_2,M,\delta>0$, the set $\mathcal{E}_n$ and the $cm^{-7/2}$-net $\mathcal{N}_n$ as in the statement of this lemma. With probability at least
\begin{equation*}
1-C_1\exp(-C_2m^{1/2})-C_1\exp(-C_2m^{1/8}),
\end{equation*}
the events described in both Lemmas \ref{lem:net} and \ref{lem:inside_net} happen. In this case, for any $x\in\C^n$ verifying Condition \eqref{eq:global_cond}, the normalized vector $x'=x/||x||$ belongs to $\mathcal{E}_n$. As $\mathcal{N}_n$ is a $cm^{-7/2}$-net of $\mathcal{E}_n$, there exists $y\in\mathcal{N}_n$ such that
\begin{equation*}
||x'-y|| \leq cm^{-7/2}.
\end{equation*}
By triangular inequality, and using Lemmas \ref{lem:net} and \ref{lem:inside_net},
\begin{align*}
\left|\scal{Ax_0}{b\odot\phase(Ax')}\right|
&\geq \left|\scal{Ax_0}{b\odot\phase(Ay)}\right|\\
&\hskip 1cm
- \left|\scal{Ax_0}{b\odot\phase(Ay)}-\scal{Ax_0}{b\odot\phase(Ax')}\right|\\
&\geq (1+\delta)m \frac{||x_0||}{||y||}|\scal{x_0}{y}|
- C_3 ||x_0||^2 nm^{1/4}\\
&= (1+\delta)m ||x_0|| \,|\scal{x_0}{y}|
- C_3 ||x_0||^2 n m^{1/4}\\
&\geq (1+\delta)m ||x_0|| \,|\scal{x_0}{x'}|
- (1+\delta)m ||x_0||^2 ||x'-y||
- C_3 ||x_0||^2 n m^{1/4}\\
&\geq (1+\delta)m ||x_0|| \,|\scal{x_0}{x'}|
- ||x_0||^2\left((1+\delta)c m^{-5/2} 
+ C_3  n m^{1/4}\right).
\end{align*}
As $x'$ belongs to $\mathcal{E}_n$, if $m\geq Mn^2$ and $m$ is large enough,
\begin{align*}
||x_0||\left((1+\delta)cm^{-5/2}+C_3n m^{1/4}\right)
&\leq ||x_0||\frac{\delta m n^{-1/2}}{2}\\
&\leq \frac{\delta m}{2}|\scal{x_0}{x'}|.
\end{align*}
So we deduce from this and the inequality immediately before:
\begin{align*}
\left|\scal{Ax_0}{b\odot\phase(Ax')}\right|
&\geq \left(1+\frac{\delta}{2}\right)m ||x_0||\,|\scal{x_0}{x'}|;\\
\Rightarrow\quad\quad
\left|\scal{Ax_0}{b\odot\phase(Ax)}\right|
&\geq \left(1+\frac{\delta}{2}\right)m \frac{||x_0||}{||x||}\,|\scal{x_0}{x}|.
\end{align*}
By Lemma \ref{lem:global_intro}, this is what we had to prove.

\end{proof}

\section{Numerical experiments\label{s:numerical}}

In this section, we numerically validate the results obtained in Corollaries \ref{cor:global_convergence} and \ref{cor:global_convergence_without}. We formulate a conjecture about the convergence of alternating projections with random initialization, in the regime $m=O(n)$.

The code used to generate Figures \ref{fig:with_init}, \ref{fig:tt_conv_tab} and \ref{fig:without_init} is available at \url{http://www-math.mit.edu/~waldspur/code/alternating_projections_code.zip}.

\subsection{Alternating projections with initialization}

Our first experiment consists in a numerical validation of Corollary \ref{cor:global_convergence}: alternating projections succeed with high probability, when they start from a good initial point, in the regime where the number of measurements is linear in the problem dimension ($m=O(n)$).

We use the initialization method described in \citep*{candes_wirtinger2}, as presented in Algorithm \ref{alg:algo_complet}. We run the algorithm for various choices of $n$ and $m$, $3000$ times for each choice. This allows us to compute an empirical probability of success, for each value of $(n,m)$.

The results are presented in Figure \ref{fig:with_init}. They confirm that, when $m=Cn$, for a sufficiently large constant $C>0$, the success probability can be arbitrarily close to $1$.

\begin{figure}
  \centering
  \newlength\figureheight 
  \newlength\figurewidth 
  \setlength\figureheight{6cm} 
  \setlength\figurewidth{8cm}
%
%
\begin{tikzpicture}

\begin{axis}[%
width=0.662\figurewidth,
height=\figureheight,
at={(0\figurewidth,0\figureheight)},
scale only axis,
point meta min=0,
point meta max=1,
axis on top,
separate axis lines,
every outer x axis line/.append style={black},
every x tick label/.append style={font=\color{black}},
xmin=0.5,
xmax=15.5,
xtick={1,3,5,7,9,11,13,15},
xticklabels={{2},{6},{10},{14},{18},{22},{26},{30}},
xlabel={Signal size $n$},
every outer y axis line/.append style={black},
every y tick label/.append style={font=\color{black}},
y dir=reverse,
ymin=0.5,
ymax=17.5,
ytick={2,4,6,8,10,12,14,16},
yticklabels={{92},{80},{68},{56},{44},{32},{20},{8}},
ylabel={Number of measurements $m$},
axis background/.style={fill=white}
]
\addplot [forget plot] graphics [xmin=0.5,xmax=15.5,ymin=0.5,ymax=17.5] {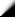};
\addplot [color=red,solid,line width=2.0pt,forget plot]
  table[row sep=crcr]{%
0.5	16.8333333333333\\
15.5	1.83333333333333\\
};
\end{axis}
\end{tikzpicture}%
  \caption{Probability of success for Algorithm \ref{alg:algo_complet}, as a function of $n$ and $m$. Black points indicate a probability equal to $0$, and white points a probability equal to $1$. The red line serves as a reference: it is the line $m=3n$.
    \label{fig:with_init}}
\end{figure}

\subsection{Alternating projections without good initialization}

\subsubsection{Disappearing of stagnation points\label{sss:disappearing}}

Next, we investigate Corollary \ref{cor:global_convergence_without}: if $m\geq Cn^2$, for $C>0$ large enough, the method of alternating projections succeeds, with high probability, starting from any initialization (that is not almost orthogonal to the true solution). In particular, there is no stagnation point, unless possibly among vectors that are \textit{almost orthogonal} to the true solution.

To numerically validate this result, we have generated vectors $x_0$ of size $n$ and measurements matrices $A$ of size $m\times n$ for various choices of $n$ and $m$. For each $(x_0,A)$, we have randomly chosen $10000$ initializations that were not almost orthogonal to $x_0$, and we have recorded whether alternating projections, starting from these initializations, always succeeded in reconstructing $x_0$ from $|Ax_0|$. When at least one of these initializations failed, it proved that there was at least one stagnation point. Otherwise, we have considered it as a sign of absence of stagnation points.

We could thus compute, for each choice of $(n,m)$, the probability of absence of stagnation point. The result is displayed on Figure \ref{fig:tt_conv_tab}. As foreseen by Corollary \ref{cor:global_convergence_without}, the probability becomes arbitrarily close to $1$ when $m\geq Cn^2$ for $C>0$ large enough.

\begin{figure}
	\centering
	\setlength\figureheight{6cm} 
	\setlength\figurewidth{8cm}
%
%
\begin{tikzpicture}

\begin{axis}[%
width=0.577\figurewidth,
height=\figureheight,
at={(0\figurewidth,0\figureheight)},
scale only axis,
point meta min=0,
point meta max=1,
axis on top,
separate axis lines,
every outer x axis line/.append style={black},
every x tick label/.append style={font=\color{black}},
xmin=0.5,
xmax=10.5,
xtick={1,3,5,7,9},
xticklabels={{2},{4},{6},{8},{10}},
xlabel={Signal size $n$},
every outer y axis line/.append style={black},
every y tick label/.append style={font=\color{black}},
y dir=reverse,
ymin=0.5,
ymax=13.5,
ytick={2,4,6,8,10,12},
yticklabels={{112},{92},{72},{52},{32},{12}},
ylabel={Number of measurements $m$},
axis background/.style={fill=white}
]
\addplot [forget plot] graphics [xmin=0.5,xmax=10.5,ymin=0.5,ymax=13.5] {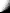};
\addplot [color=red,solid,line width=2.0pt,forget plot]
  table[row sep=crcr]{%
0.5	13.0875\\
1	13\\
1.5	12.8875\\
2	12.75\\
2.5	12.5875\\
3	12.4\\
3.5	12.1875\\
4	11.95\\
4.5	11.6875\\
5	11.4\\
5.5	11.0875\\
6	10.75\\
6.5	10.3875\\
7	10\\
7.5	9.5875\\
8	9.15\\
8.5	8.6875\\
9	8.2\\
9.5	7.6875\\
10	7.15\\
10.5	6.5875\\
};
\end{axis}
\end{tikzpicture}%
	\caption{Probability that there is no stagnation point, as a function of $n$ and $m$. Black points indicate a probability equal to $0$, and white points a probability equal to $1$. The red line serves as a reference: it is the line $m=\frac{1}{2}n^2$.
\label{fig:tt_conv_tab}}
\end{figure}

The same results are presented in Figure \ref{fig:no_stagnation} under a different form. The graph on the left hand side shows, for each $n$, the number $M_n$ of measurements above which the probability that there is at least one stagnation point drops under $0.5$. The curve has a clear quadratic shape.

The plot on the right hand side represents $M_n/n^2$ as a function of $n$. It is clearly upper bounded by a constant. It also seems to be lower bounded by a positive constant (or possibly by a very slowly decaying function, like $(\log\log)^{-1}$), which indicates that the number of measurements $m=O(n^2)$ that appears in Corollary \ref{cor:global_convergence_without} is probably optimal: when $m\ll n^2$, the probability that there are no stagnation points is small.

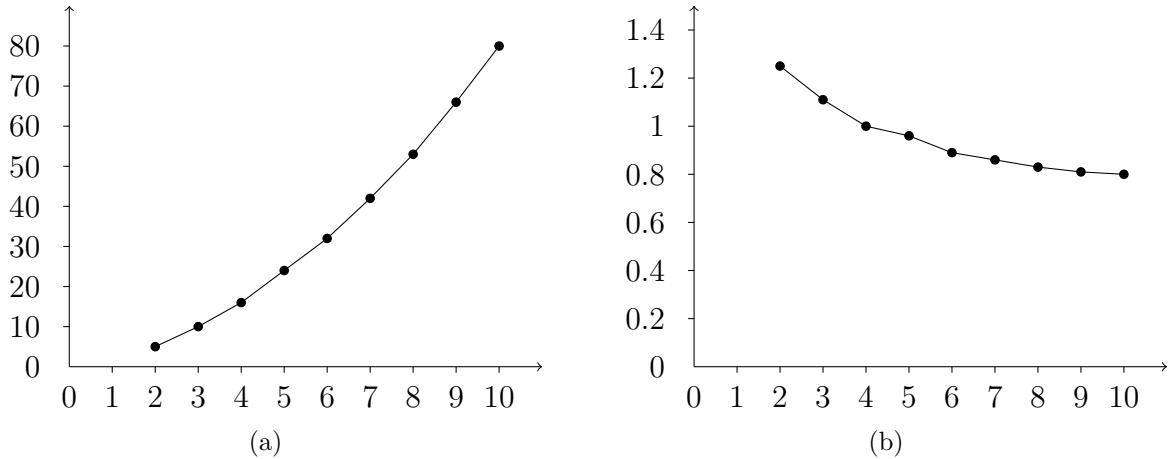
\begin{figure}
\subfloat[][]{
\begin{tikzpicture}[scale=0.8]
  \centering
  \draw[->] (0,0) -- (0,6);
  \draw[->] (0,0) -- (11/1.4,0);
  \draw plot[mark=*] coordinates {(2/1.4,5/15) (3/1.4,10/15) (4/1.4,16/15)
    (5/1.4,24/15) (6/1.4,32/15) (7/1.4,42/15) (8/1.4,53/15)
    (9/1.4,66/15) (10/1.4,80/15)};
  \foreach \y in {0,10,20,...,80} \draw(-0.3,\y/15)node[left]{\y};
  \foreach \y in {10,20,...,80} \draw(0,\y/15) -- (-0.1,\y/15);
  \foreach \x in {0,1,...,10} \draw(\x/1.4,-0.15)node[below]{\x};
  \foreach \x in {1,2,...,10} \draw(\x/1.4,0) -- (\x/1.4,-0.1);
\end{tikzpicture}}
\quad\quad
\subfloat[][]{
\begin{tikzpicture}[scale=0.8]
  \centering
  \draw[->] (0,0) -- (0,6);
  \draw[->] (0,0) -- (11/1.4,0);
  \draw plot[mark=*] coordinates {(2/1.4,1.25*4) (3/1.4,1.11*4) (4/1.4,1*4)
    (5/1.4,0.96*4) (6/1.4,0.89*4) (7/1.4,0.86*4) (8/1.4,0.83*4)
    (9/1.4,0.81*4) (10/1.4,0.8*4)};
  \foreach \y in {0,0.2,0.4,0.6,0.8,1,1.2,1.4} \draw(-0.3,\y*4)node[left]{\y};
  \foreach \y in {0.2,0.4,...,1.4} \draw(0,\y*4) -- (-0.1,\y*4);
  \foreach \x in {0,1,...,10} \draw(\x/1.4,-0.15)node[below]{\x};
  \foreach \x in {1,2,...,10} \draw(\x/1.4,0) -- (\x/1.4,-0.1);
\end{tikzpicture}}
  \caption{(a) For each signal size $n$, the smallest number of measurements $m$ for which the probability that there exist stagnation points is under $0.5$. (b) The same curve, renormalized by division by $n^2$.
    \label{fig:no_stagnation}}
\end{figure}


\subsubsection{Random initialization}

Our last experiment consists in measuring the probability that alternating projections succeed, when started from a random initial point (sampled from the unit sphere with uniform probability).

The results are presented in Figure \ref{fig:without_init}. They lead to the following conjecture.
\begin{conj}\label{conj:convergence_random}
Let any $\epsilon>0$ be fixed. When $m\geq Cn$, for $C>0$ large enough, alternating projections with a random isotropic initialization succeed with probability at least $1-\epsilon$.
\end{conj}

As we have seen in Paragraph \ref{sss:disappearing}, in the regime $m=O(n)$, there are (attractive) stagnation points, so there are initializations for which alternating projections fail. However, it seems that these bad initializations occupy a very small volume in the space of all possible initial points. Therefore, a random initialization leads to success with high probability.

Unfortunately, proving this conjecture a priori requires to evaluate in some way the size of the attraction basin of stagnation points, which seems difficult.

\begin{figure}
	\centering
	\setlength\figureheight{6cm} 
	\setlength\figurewidth{8cm}
%
%
\begin{tikzpicture}

\begin{axis}[%
width=0.662\figurewidth,
height=\figureheight,
at={(0\figurewidth,0\figureheight)},
scale only axis,
point meta min=0,
point meta max=1,
axis on top,
separate axis lines,
every outer x axis line/.append style={black},
every x tick label/.append style={font=\color{black}},
xmin=0.5,
xmax=15.5,
xtick={1,3,5,7,9,11,13,15},
xticklabels={{2},{6},{10},{14},{18},{22},{26},{30}},
xlabel={Signal size $n$},
every outer y axis line/.append style={black},
every y tick label/.append style={font=\color{black}},
y dir=reverse,
ymin=0.5,
ymax=17.5,
ytick={2,4,6,8,10,12,14,16},
yticklabels={{152},{132},{112},{92},{72},{52},{32},{12}},
ylabel={Number of measurements $m$},
axis background/.style={fill=white}
]
\addplot [forget plot] graphics [xmin=0.5,xmax=15.5,ymin=0.5,ymax=17.5] {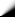};
\addplot [color=red,solid,line width=2.0pt,forget plot]
  table[row sep=crcr]{%
0.5	16.9\\
15.5	7.9\\
};
\end{axis}
\end{tikzpicture}%
	\caption{Probability of success for alternating projections with a random Gaussian initialization, as a function of $n$ and $m$. Black points indicate a probability equal to $0$, and white points a probability equal to $1$. The red line serves as a reference: it is the line $m=3n$.\label{fig:without_init}}
\end{figure}

\appendix

\section{Proposition \ref{prop:stagnation_points}\label{s:stagnation_points}}

\begin{prop*}[Proposition \ref{prop:stagnation_points}]
For any $y_0$, the sequence $(y_k)_{k\in\N}$ is bounded. Any accumulation point $y_\infty$ of $(y_k)_{k\in\N}$ satisfies the following property:
\begin{equation*}
\exists u\in E_{\phase}(y_\infty),\quad\quad
(AA^\dag)(b\odot u)=y_\infty.
\end{equation*}
In particular, if $y_\infty$ has no zero entry,
\begin{equation*}
(AA^\dag)(b\odot \phase(y_\infty))=y_\infty.
\end{equation*}
\end{prop*}

\begin{proof}[Proof of Proposition \ref{prop:stagnation_points}]
The boundedness of $(y_k)_{k\in\N}$ is a consequence of the fact that $||y'_k||=||b||$ for all $k$, so $||y_{k+1}||\leq |||AA^\dag|||\,||b||$.

Let us show the second part of the statement. Let $y_\infty$ be an accumulation point of $(y_k)_{k\in\N}$, and let $\phi:\N\to\N$ be an extraction such that
\begin{equation*}
y_{\phi(n)}\to y_\infty\quad\mbox{when}\quad n\to+\infty.
\end{equation*}
By compacity, as $(y'_{\phi(n)})_{n\in\N}$ and $(y_{\phi(n)+1})_{n\in\N}$ are bounded sequences, we can assume, even if we have to consider replace $\phi$ by a subextraction, that they also converge. We denote by $y'_\infty$ and $y_\infty^{+1}$ their limits:
\begin{equation*}
y'_{\phi(n)}\to y'_{\infty}\quad\mbox{and}\quad y_{\phi(n)+1}\to y_\infty^{+1}\quad\mbox{when }n\to+\infty.
\end{equation*}

Let us define
\begin{equation*}
E_b=\{y'\in\C^m,|y'|=b\}.
\end{equation*}
We observe that, for any $k$,
\begin{equation*}
d(y'_{k-1},\Range(A))\geq d(y_{k},E_b)\geq d(y'_k,\Range(A)).
\end{equation*}
Indeed, because the operators $y\to b\odot\phase(y)$ and $y\to(AA^\dag)y$ are projections,
\begin{equation*}
\begin{array}{rcl}
d(y'_{k-1},\Range(A))=&d(y'_{k-1},y_k)&\geq d(y_k,E_b);\\
d(y_k,E_b)=&d(y_k,y'_k)&\geq d(y'_k,\Range(A)).
\end{array}
\end{equation*}
So the sequences $(d(y_k,E_b))_{k\in\N}$ and $(d(y'_k,\Range(A)))_{k\in\N}$ converge to the same non-negative limit, that we denote by $\delta$. In particular,
\begin{align*}
d(y_\infty,E_b)=\delta=d(y'_\infty,\Range(A)).
\end{align*}
If we pass to the limit the equalities $d(y_{\phi(n)},E_b)=||y_{\phi(n)}-y'_{\phi(n)}||$ and $d(y'_{\phi(n)},\Range(A))=||y'_{\phi(n)}-y_{\phi(n)+1}||$, we get
\begin{equation*}
||y_\infty-y'_\infty||=||y'_\infty-y_\infty^{+1}||=\delta=d(y'_\infty,\Range(A)).
\end{equation*}
As $\Range(A)$ is convex, the projection of $y'_\infty$ onto it is uniquely defined. This implies
\begin{equation*}
y_\infty=y_\infty^{+1},
\end{equation*}
and, because $\forall n,y_{\phi(n)+1}=(AA^\dag) y'_{\phi(n)}$,
\begin{equation*}
y_\infty=y_\infty^{+1}=(AA^\dag)y'_\infty.
\end{equation*}
To conclude, we now have to show that $y'_\infty=b\odot u$ for some $u\in E_{\phase}(y_\infty)$. We use the fact that, for all $n$, $y'_{\phi(n)}=b\odot \phase(y_{\phi(n)})$.

For any $i\in\{1,\dots,m\}$, if $(y_\infty)_i\ne 0$, $\phase$ is continuous around $(y_\infty)_i$, so $(y'_{\infty})_i= b_i\phase((y_\infty)_i)$. We then set $u_i=\phase((y_\infty)_i)$, and we have $(y'_\infty)_i=b_iu_i$.

If $(y_\infty)_i=0$, we set $u_i=\phase((y'_\infty)_i)\in E_{\phase}(0)=E_{\phase}((y_\infty)_i)$. We then have $y'_\infty=|y'_\infty|u_i=b_iu_i$.

With this definition of $u$, we have, as claimed, $y'_\infty=b\odot u$ and $u\in E_{\phase}(y_\infty)$.

\end{proof}

\section{Technical lemmas for Section \ref{s:with_init}}

\subsection{Proof of Lemma \ref{lem:diff_phase}\label{ss:diff_phase}}

\begin{lem*}[Lemma \ref{lem:diff_phase}]
For any $z_0,z\in\C$,
\begin{equation*}
|\phase(z_0+z)-\phase(z_0)| \leq 2. 1_{|z|\geq |z_0|/6} + \frac{6}{5}\left|\Im\left(\frac{z}{z_0}\right)\right|.
\end{equation*}
\end{lem*}
\begin{proof}
The inequality holds if $z_0=0$, so we can assume $z_0\ne 0$. We remark that, in this case,
\begin{equation*}
|\phase(z_0+z)-\phase(z_0)| = |\phase(1+z/z_0)-1|.
\end{equation*}
It is thus enough to prove the lemma for $z_0=1$, so we make this assumption.

When $|z|\geq 1/6$, the inequality is valid. Let us now assume that $|z|<1/6$. Let $\theta\in\left]-\frac{\pi}{2};\frac{\pi}{2}\right[$ be such that
\begin{equation*}
e^{i\theta}=\phase(1+z).
\end{equation*}
Then
\begin{align*}
|\phase(1+z)-1| &=|e^{i\theta}-1|\\
&=2|\sin(\theta/2)|\\
&\leq |\tan\theta|\\
&=\frac{|\Im(1+z)|}{|\Re(1+z)|}\\
&\leq \frac{|\Im(z)|}{1-|z|}\\
&\leq \frac{6}{5}|\Im(z)|.
\end{align*}
So the inequality is also valid.
\end{proof}

\subsection{Proof of Lemma \ref{lem:first_term}\label{ss:first_term}}

\begin{lem*}[Lemma \ref{lem:first_term}]
For any $\eta>0$, there exists $C_1,C_2,M,\gamma>0$ such that the inequality
\begin{equation*}
||\,|Ax_0|\odot 1_{|v|\geq |Ax_0|}||\leq \eta ||v||
\end{equation*}
holds for any $v\in\Range(A)$ such that $||v||<\gamma ||Ax_0||$, with probability at least
\begin{equation*}
1-C_1\exp(-C_2m),
\end{equation*}
when $m\geq Mn$.
\end{lem*}

\begin{proof}
For any $S\subset\{1,\dots,m\}$, we denote by $1_S$ the vector of $\C^m$ such that
\begin{align*}
(1_S)_j&=1\mbox{ if }j\in S\\
&=0\mbox{ if }j\notin S.
\end{align*}
We use the following two lemmas, proven in Paragraphs \ref{sss:S_geq_bm} and \ref{sss:S_leq_bm}.

\begin{lem}\label{lem:S_geq_bm}
Let $\beta\in]0;1/2[$ be fixed. There exist $C_1>0$ such that, with probability at least
\begin{equation*}
1-C_1\exp(-\beta^3m/e),
\end{equation*}
the following property holds: for any $S\subset\{1,\dots,m\}$ such that $\Card(S)\geq\beta m$,
\begin{equation}\label{eq:Ax01S}
||\,|Ax_0|\odot 1_S|| \geq \beta^{3/2}e^{-1/2}||Ax_0||.
\end{equation}
\end{lem}

\begin{lem}\label{lem:S_leq_bm}
Let $\beta\in\left]0;\frac{1}{100}\right]$ be fixed. There exist $M,C_1,C_2>0$ such that, if $m\geq M n$, then, with probability at least
\begin{equation*}
1-C_1\exp(-C_2 m),
\end{equation*}
the following property holds: for any $S\subset\{1,\dots,m\}$ such that $\Card (S)<\beta m$ and for any $y\in\Range(A)$,
\begin{equation}\label{eq:y1S}
||y\odot 1_S||\leq 10\sqrt{\beta\log(1/\beta)}||y||.
\end{equation}
\end{lem}

Let $\beta>0$ be such that $10\sqrt{\beta\log(1/\beta)}\leq \eta$. Let $M$ be as in Lemma \ref{lem:S_leq_bm}. We set
\begin{equation*}
\gamma = \beta^{3/2}e^{-1/2}.
\end{equation*}
We assume that Equations \eqref{eq:Ax01S} and \eqref{eq:y1S} hold; from the lemmas, this occurs with probability at least
\begin{equation*}
1-C_1'\exp(-C_2'm),
\end{equation*}
for some constants $C_1',C_2'>0$, provided that $m\geq Mn$.

On this event, for any $v\in\Range(A)$ such that $||v||<\gamma ||Ax_0||$, if we set $S_v=\{i\mbox{ s.t. }|v_i|\geq|Ax_0|_i\}$, we have that
\begin{equation*}
\Card S_v < \beta m.
\end{equation*}
Indeed, if it was not the case, we would have, by Equation \eqref{eq:Ax01S},
\begin{align*}
||v||&\geq ||v\odot 1_{S_v}||\\
&\geq ||\,|Ax_0|\odot 1_{S_v}||\\
&\geq \beta^{3/2}e^{-1/2}||Ax_0||\\
&=\gamma ||Ax_0||,
\end{align*}
which is in contradiction with the way we have chosen $v$.

So we can apply Equation \eqref{eq:y1S}, and we get
\begin{align*}
||\,|Ax_0|\odot 1_{|v|\geq |Ax_0|}||
&\leq ||v\odot 1_S||\\
&\leq 10\sqrt{\beta\log(1/\beta)}||v||\\
&\leq \eta ||v||.
\end{align*}

\end{proof}

\subsubsection{Proof of Lemma \ref{lem:S_geq_bm}\label{sss:S_geq_bm}}

\begin{proof}[Proof of Lemma \ref{lem:S_geq_bm}]
If we choose $C_1$ large enough, it is enough to show the property for $m$ larger than some fixed constant.

We first assume $S$ fixed, with cardinality $\Card S\geq\beta m$. We use the following lemma.
\begin{lem}[\citet*{dasgupta}, Lemma 2.2]\label{lem:dasgupta}
Let $k_1<k_2$ be natural numbers. Let $X\in\C^{k_2}$ be a random vector whose coordinates are independent, Gaussian, of variance $1$. Let $Y$ be the projection of $X$ onto its $k_1$ first coordinates. Then, for any $t>0$,
\begin{align*}
\mbox{\rm Proba}\left(\frac{||Y||}{||X||}\leq \sqrt{\frac{t k_1}{k_2}}\right)
&\leq \exp\left(k_1(1-t+\log t)\right)&\mbox{if }t<1;\\
\mbox{\rm Proba}\left(\frac{||Y||}{||X||}\geq \sqrt{\frac{t k_1}{k_2}}\right)
&\leq \exp\left(k_1(1-t+\log t)\right)&\mbox{if }t>1.
\end{align*}
\end{lem}
From this lemma, for any $t\in]0;1[$, because $Ax_0$ has independent Gaussian coordinates,
\begin{align*}
P\left(\frac{||\,|Ax_0|\odot 1_S||}{||Ax_0||}\leq\sqrt{t\beta} \right)
\leq \exp\left(-\beta m(t-1-\ln t)\right).
\end{align*}
In particular, for $t=\frac{\beta^2}{e}$,
\begin{equation}
P\left(\frac{||\,|Ax_0|\odot 1_S||}{||Ax_0||}\leq \beta^{3/2}e^{-1/2} \right)
\leq \exp\left(-\beta m\left(\frac{\beta^2}{e}-2\ln \beta\right)\right).
\label{eq:P_Ax01S}
\end{equation}
As soon as $m$ is large enough, the number of subsets $S$ of $\{1,\dots,m\}$ with cardinality $\lceil \beta m\rceil$ satisfies
\begin{align}
\binom{m}{\lceil\beta m\rceil}
&\leq\left(\frac{em}{\lceil \beta m\rceil}\right)^{\lceil \beta m\rceil} \nonumber\\
&\leq\exp\left(2m\beta\log\frac{1}{\beta}\right).\label{eq:maj_binom}
\end{align}
(The first inequality is a classical result regarding binomial coefficients.)

We combine Equations \eqref{eq:P_Ax01S} and \eqref{eq:maj_binom}: Property \eqref{eq:Ax01S} is satisfied for any $S$ of cardinality $\lceil\beta m\rceil$ with probability at least
\begin{equation*}
1-\exp\left(-\frac{\beta^3}{e}m\right),
\end{equation*}
provided that $m$ is larger that some constant which depends on $\beta$.

If it is satisfied for any $S$ of cardinality $\lceil\beta m\rceil$, then it is satisfied for any $S$ of cardinality larger than $\beta m$, which implies the result.
\end{proof}

\subsubsection{Proof of Lemma \ref{lem:S_leq_bm}\label{sss:S_leq_bm}}

\begin{proof}[Proof of Lemma \ref{lem:S_leq_bm}]
We first assume $S$ to be fixed, of cardinality exactly $\lceil\beta m\rceil$.

Any vector $y\in\Range(A)$ is of the form $y=Av$, for some $v\in\C^n$. Inequality \eqref{eq:y1S} can then be rewritten as:
\begin{equation}\label{eq:S_leq_bm_rewritten}
||A_S v||=||\mathrm{Diag}(1_S)Av||\leq 10\sqrt{\beta\log(1/\beta)}||Av||,
\end{equation}
where $A_S$, by definition, is the submatrix obtained from $A$ by extracting the rows whose indexes are in $S$.

We apply Proposition \ref{prop:davidson} to $A$ and $A_S$, respectively for $t=\frac{1}{2}$ and $t=3\sqrt{\log(1/\beta)}$. It guarantees that the following properties hold:
\begin{gather*}
\inf_{v\in\C^n}\frac{||Av||}{||v||}\geq \sqrt{m}\left(\frac{1}{2}-\sqrt{\frac{n}{m}}\right);\\
\sup_{v\in\C^n}\frac{||A_Sv||}{||v||}\leq \sqrt{\Card S}\left(1+\sqrt\frac{n}{\Card S}+3\sqrt{\log(1/\beta)}\right),
\end{gather*}
with respective probabilities at least
\begin{gather*}
1-2\exp\left(-\frac{m}{4}\right);\\
\mbox{and }1-2\exp\left(-9(\Card S)\log(1/\beta)\right)\geq 1-2\exp\left(-9\beta\log(1/\beta)m\right).
\end{gather*}
Assuming $m\geq Mn$ for some $M>0$, we deduce from these inequalities that
\begin{align}
\forall v\in\C^n,\quad\quad
||A_Sv||&\leq \sqrt\frac{\Card S}{m}\left(\frac{1+\sqrt\frac{n}{\Card S}+3\sqrt{\log (1/\beta)}}{\frac{1}{2}-\sqrt{\frac{n}{m}}}\right)||Av||\nonumber\\
&\leq \sqrt{\beta+\frac{1}{m}}\left(\frac{1+\sqrt{\frac{1}{\beta M}}+3\sqrt{\log(1/\beta)}}{\frac{1}{2}-\sqrt{\frac{1}{M}}}\right)||Av||,\label{eq:tmp}
\end{align}
with probability at least
\begin{equation*}
1-2\exp\left(-9\beta\log(1/\beta)m\right)-2\exp\left(-\frac{m}{4}\right).
\end{equation*}
If we choose $M$ large enough, we can upper bound Equation \eqref{eq:tmp} by $(\epsilon+2\sqrt{\beta}(1+3\sqrt{\log(1/\beta)}))||Av||\leq (\epsilon+8\sqrt{\beta}\sqrt{\log(1/\beta)})$ for any fixed $\epsilon>0$. So this inequality implies Equation \eqref{eq:S_leq_bm_rewritten}.

As in the proof of Lemma \ref{lem:S_geq_bm}, there are at most
\begin{equation*}
\exp\left(2m\beta\log\frac{1}{\beta}\right)
\end{equation*}
subsets of $\{1,\dots,m\}$ with cardinality $\lceil\beta m\rceil$, as soon as $m$ is large enough. As a consequence, Equation \eqref{eq:S_leq_bm_rewritten} holds for any $v\in\C^n$ and $S$ of cardinality $\lceil\beta m\rceil$ with probability at least
\begin{equation*}
1-2\exp\left(-7\beta\log(1/\beta)m\right)-2\exp\left(-\left(\frac{1}{4}-2\beta\log\frac{1}{\beta}\right)m\right).
\end{equation*}
When $\beta\leq \frac{1}{100}$, we have
\begin{equation*}
\frac{1}{4}-2\beta\log\frac{1}{\beta}> 0,
\end{equation*}
so the resulting probability is larger than
\begin{equation*}
1-C_1\exp(-C_2 m),
\end{equation*}
for some well-chosen constants $C_1,C_2>0$.

This ends the proof. Indeed, if Equation \eqref{eq:S_leq_bm_rewritten} holds for any set of cardinality $\lceil\beta m\rceil$, it also holds for any set of cardinality $\Card S<\beta m$, because $||A_{S'} v||\leq ||A_S v||$ whenever $S'\subset S$. This implies Equation \eqref{eq:y1S}.
\end{proof}

\subsection{Proof of Lemma \ref{lem:second_term}\label{ss:second_term}}

\begin{lem*}[Lemma \ref{lem:second_term}]
For $M,C_1>0$ large enough, and $C_2>0$ small enough, when $m\geq M n$, the property
\begin{equation}\label{eq:second_term}
||\Im(v\odot\overline{\phase(Ax_0)})||\leq \frac{4}{5}||v||
\end{equation}
holds for any $v\in\Range(A)\cap \{Ax_0\}^\perp$, with probability at least
\begin{equation*}
1-C_1\exp(-C_2 m).
\end{equation*}
\end{lem*}

\begin{proof}
If we multiply $x_0$ by a positive real number, we can assume $||x_0||=1$. Moreover, as the law of $A$ is invariant under right multiplication by a unitary matrix, we can assume that
\begin{equation*}
x_0=\left(\begin{smallmatrix}1\\0\\\vdots\\0\end{smallmatrix}\right).
\end{equation*}
Then, if we write $A_1$ the first column of $A$, and $A_{2:n}$ the submatrix of $A$ obtained by removing this first column,
\begin{equation}\label{eq:range_perp}
\Range(A)\cap\{Ax_0\}^\perp
=\left\{w-\frac{\scal{w}{A_1}}{||A_1||^2}A_1,
w\in\Range(A_{2:n})
\right\}.
\end{equation}
We first observe that
\begin{equation*}
\sup_{w\in\Range(A_{2:n})-\{0\}}\frac{|\scal{w}{A_1}|}{||w||}
\end{equation*}
is the norm of the orthogonal projection of $A_1$ onto $\Range(A_{2:n})$. The $(n-1)$-dimensional subspace $\Range(A_{2:n})$ has an isotropic distribution in $\C^m$, and is independent of $A_1$. Thus, from Lemma \ref{lem:dasgupta} coming from \citep*{dasgupta}, for any $t>1$,
\begin{equation*}
\sup_{w\in\Range(A_{2:n})-\{0\}}\frac{|\scal{w}{A_1}|}{||w||\,||A_1||}< \sqrt{\frac{t(n-1)}{m}},
\end{equation*}
with probability at least
\begin{equation*}
1-\exp\left(-(n-1)(t-1-\ln t)\right).
\end{equation*}
We take $t=\frac{m}{n-1}(0.04)^2$ (which is larger than $1$ when $m\geq Mn$ with $M>0$ large enough), and it implies that
\begin{equation}\label{eq:second_005}
\sup_{w\in\Range(A_{2:n})-\{0\}}\frac{|\scal{w}{A_1}|}{||w||\,||A_1||}< 0.04
\end{equation}
with probability at least
\begin{equation*}
1-\exp(-c_2m)
\end{equation*}
for some constant $c_2>0$, provided that $m\geq Mn$ with $M$ large enough.

Second, as $A_{2:n}$ is a random matrix of size $m\times(n-1)$, whose entries are independent and distributed according to the law $\mathcal{N}(0,1/2)+\mathcal{N}(0,1/2)i$, we deduce from Proposition \ref{prop:davidson} applied with $t=0.01$ that, with probability at least
\begin{equation*}
1-2\exp\left(-10^{-4}m\right),
\end{equation*}
we have, for any $x\in\C^{n-1}$,
\begin{equation}\label{eq:norm_C}
||A_{2:n}x|| \geq \sqrt{m}\left(1-\sqrt{\frac{(n-1)}{m}}-0.01\right)||x||\geq 0.98\sqrt{m}||x||,
\end{equation}
provided that $m\geq 10000n$.

We now set
\begin{equation*}
C = \mathrm{Diag}(\overline{\phase(A_1)})A_{2:n}.
\end{equation*}
The matrix $\left(\begin{matrix}\Im C&\Re C\end{matrix}\right)$ has size $m\times(2(n-1))$; its entries are independent and distributed according to the law $\mathcal{N}(0,1/2)$. So by \citep*[Thm II.13]{davidson} (applied with $t=0.01$), with probability at least
\begin{equation*}
1-\exp(-5.10^{-5}m),
\end{equation*}
we have, for any $x\in\R^{2(n-1)}$,
\begin{equation}\label{eq:norm_ImReC}
\left|\left|\left(\begin{matrix}\Im C&\Re C\end{matrix}\right)x
\right|\right|\leq \sqrt{\frac{m}{2}}\left(1+\sqrt\frac{2(n-1)}{m}+0.01\right)||x||
\leq 1.02\sqrt\frac{m}{2}||x||,
\end{equation}
provided that $m\geq 20000n$.

When Equations \eqref{eq:norm_C} and \eqref{eq:norm_ImReC} are simultaneously valid, any $w=A_{2:n}w'$ belonging to $\Range(A_{2:n})$ satisfies:
\begin{align}
\left|\left|\Im(w\odot\overline{\phase(Ax_0)})\right|\right|
&=\left|\left|\Im(Cw')\right|\right|\nonumber\\
&=\left|\left|\begin{pmatrix}\Im C&\Re C\end{pmatrix}\begin{pmatrix}
\Re w'\\\Im w'\end{pmatrix}
 \right|\right|\nonumber\\
&\leq 1.02\sqrt\frac{m}{2}\left|\left|\begin{pmatrix}
\Re w'\\\Im w'\end{pmatrix}
 \right|\right|\nonumber\\
&=1.02\sqrt\frac{m}{2}||w'||\nonumber\\
&\leq \frac{1.02}{0.98\sqrt{2}}||A_{2:n}w'||\nonumber\\
&= \frac{1.02}{0.98\sqrt{2}}||w||\nonumber\\
&\leq 0.75 ||w||.\label{eq:norms_combined}
\end{align}

We now conclude. Equations \eqref{eq:second_005}, \eqref{eq:norm_C} and \eqref{eq:norm_ImReC} hold simultaneously with probability at least
\begin{equation*}
1-C_1\exp(-C_2 m)
\end{equation*}
for any $C_1$ large enough and $C_2$ small enough, provided that $m\geq Mn$ with $M$ large enough. Let us show that, on this event, Equation \eqref{eq:second_term} also holds. Any $v\in\Range(A)\cap\{Ax_0\}^\perp$, from Equality \eqref{eq:range_perp}, can be written as
\begin{equation*}
v=w-\frac{\scal{w}{A_1}}{||A_1||^2}A_1,
\end{equation*}
for some $w\in\Range(A_{2:n})$. Using Equation \eqref{eq:second_005}, then Equation \eqref{eq:norms_combined}, we get:
\begin{align*}
\left|\left|\Im(v\odot\overline{\phase(Ax_0)})\right|\right|
&\leq\left|\left|\Im(w\odot\overline{\phase(Ax_0)})\right|\right|
+\left|\left|\frac{\scal{w}{A_1}}{||A_1||^2}A_1\right|\right|\\
&\leq\left|\left|\Im(w\odot\overline{\phase(Ax_0)})\right|\right|
+0.04 ||w||\\
&\leq 0.79 ||w||.
\end{align*}
But then, by Equation \eqref{eq:second_005} again,
\begin{equation*}
||v||^2=||w||^2-\frac{\scal{w}{A_1}^2}{||A_1||^2}\geq (1-(0.04)^2)||w||^2.
\end{equation*}
So
\begin{align*}
\left|\left|\Im(v\odot\overline{\phase(Ax_0)})\right|\right|
&\leq 0.79 ||w||\\
&\leq \frac{0.79}{\sqrt{1-(0.04)^2}}||v||\\
&\leq \frac{4}{5}||v||.
\end{align*}

\end{proof}

\section{Technical lemmas for Section \ref{s:without_init}}

\subsection{Proof of Lemma \ref{lem:global_intro}\label{ss:global_intro}}

\begin{lem*}[Lemma \ref{lem:global_intro}]
To prove Theorem \ref{thm:global_convergence}, it is enough to prove that there exist $C_1,C_2,M,\delta>0$ such that, if $m\geq Mn^2$, then, with probability at least $1-C_1\exp(-C_2m^{1/8})$, the property
\begin{equation*}
|\scal{Ax_0}{b\odot\phase(Ax)}| \geq (1+\delta) m \frac{||x_0||}{||x||}|\scal{x_0}{x}|
\tag{\ref{eq:scal_augmente}}
\end{equation*}
holds for any $x\in\C^n$ verifying Condition \eqref{eq:global_cond}.
\end{lem*}

\begin{proof}[Proof of Lemma \ref{lem:global_intro}]
Let us define $\lambda_1(A)\geq\dots\geq\lambda_n(A)$ to be the $n$ singular values of $A$. From Proposition \ref{prop:davidson}, setting $t=\delta'/\sqrt{n}$ for $\delta'$ small enough, if $M$ is high enough, we have with probability larger than $1-C_1\exp(-C_2m/n)\geq 1-C_1\exp(-C_2m^{1/2})$,
\begin{gather*}
\frac{\lambda_1^2(A)}{\lambda_n^2(A)}-1\leq \frac{\delta}{3}\frac{1}{\sqrt{n}},\\
\mbox{and }\lambda_1(A)\lambda_n(A)\leq \left(\frac{1+\delta}{1+2\delta/3}\right)m,
\end{gather*}
when $m\geq Mn^2$.

In this case, we have in particular, for any $x$ satisfying Equation \eqref{eq:global_cond},
\begin{equation*}
\frac{\lambda_1^2(A)}{\lambda_n^2(A)}-1
\leq \frac{\delta}{3}\frac{|\scal{x_0}{x}|}{||x_0||\,||x||}.
\end{equation*}
For any $x$,
\begin{align*}
|\scal{Ax_0}{b\odot\phase(Ax)}|
&=|\scal{Ax_0}{(AA^\dag)(b\odot\phase(Ax))}|\\
&=|\scal{(A^*A)x_0}{A^\dag(b\odot\phase(Ax))}|\\
&\leq \lambda_n^2(A)|\scal{x_0}{A^\dag(b\odot\phase(Ax))}|\\
&\hskip 2cm+ |\scal{(A^*A-\lambda_n^2(A)\Id)x_0}{A^\dag(b\odot\phase(Ax))}|\\
&\leq \lambda_n^2(A)|\scal{x_0}{A^\dag(b\odot\phase(Ax))}|\\
&\quad\quad+(\lambda_1^2(A)-\lambda_n^2(A))||x_0||\,||A^\dag(b\odot\phase(Ax))||.
\end{align*}
So when $x$ satisfies Equations \eqref{eq:global_cond} and \eqref{eq:scal_augmente},
\begin{align*}
\frac{|\scal{x_0}{A^\dag(b\odot\phase(Ax))}|}{||x_0||\,||A^\dag(b\odot\phase(Ax))||}
&\geq\frac{1}{\lambda_n^2(A)}\frac{|\scal{Ax_0}{b\odot\phase(Ax)}|}{||x_0||\,||A^\dag(b\odot\phase(Ax))||}
- \left(\frac{\lambda_1^2(A)}{\lambda_n^2(A)}-1\right)\\
&\geq\frac{1}{\lambda_n^2(A)}\frac{|\scal{Ax_0}{b\odot\phase(Ax)}|}{||x_0||\,||A^\dag(b\odot\phase(Ax))||}
- \frac{\delta}{3}\frac{|\scal{x_0}{x}|}{||x_0||\,||x||}\\
&\geq(1+\delta)\frac{m}{\lambda_n^2(A)}\frac{|\scal{x_0}{x}|}{||x||\,||A^\dag(b\odot\phase(Ax))||}
- \frac{\delta}{3}\frac{|\scal{x_0}{x}|}{||x_0||\,||x||}\\
&\geq(1+\delta)\frac{m}{\lambda_n(A)}\frac{|\scal{x_0}{x}|}{||x||\,||b||}
- \frac{\delta}{3}\frac{|\scal{x_0}{x}|}{||x_0||\,||x||}\\
&=(1+\delta)\frac{m}{\lambda_n(A)}\frac{|\scal{x_0}{x}|}{||x||\,||Ax_0||}
- \frac{\delta}{3}\frac{|\scal{x_0}{x}|}{||x_0||\,||x||}\\
&\geq (1+\delta)\frac{m}{\lambda_1(A)\lambda_n(A)}\frac{|\scal{x_0}{x}|}{||x||\,||x_0||}
- \frac{\delta}{3}\frac{|\scal{x_0}{x}|}{||x_0||\,||x||}\\
&\geq \left(1+\frac{\delta}{3}\right)\frac{|\scal{x_0}{x}|}{||x_0||\,||x||}.
\end{align*}
So Equation \eqref{eq:global_prop} is also satisfied (although for a smaller value of $\delta$).
\end{proof}

\subsection{Proof of Lemma \ref{lem:net}\label{ss:net}}

\begin{lem*}[Lemma \ref{lem:net}]
For any $n\in\N^*$, we set
\begin{equation*}
\mathcal{E}_n=\left\{x\in\C^n, ||x||=1\mbox{ and }\frac{||x_0||\,||x||}{\sqrt{n}}\leq
|\scal{x_0}{x}|\leq (1-\epsilon)||x_0||\,||x|| \right\}.
\end{equation*}
Let $\alpha$ be any positive number.

There exist $c,C_1,C_2,M,\delta>0$ and, for any $n\in\N^*$, a $c m^{-\alpha}$-net $\mathcal{N}_n$ of $\mathcal{E}_n$ such that, when $m\geq Mn^2$, with probability at least
\begin{equation*}
1-C_1\exp(-C_2m^{1/2}),
\end{equation*}
the following property holds: for any $x\in\mathcal{N}_n$,
\begin{equation*}
|\scal{Ax_0}{b\odot\phase(Ax)}| \geq (1+\delta) m\frac{||x_0||}{||x||}|\scal{x_0}{x}|.
\end{equation*}
\end{lem*}

\begin{proof}
For any $n\in\N^*$, $k\in\N$, let $\mathcal{M}_n^k$ be a $2^{-k}$-net of $\mathcal{E}_n$. As $\mathcal{E}_n$ is a closed subset of the complex unit sphere of dimension $n$, we can construct $\mathcal{M}_n^k$ as
\begin{equation*}
\mathcal{M}_n^k=\{P_{\mathcal{E}_n}(y),y\in\mathcal{V}_n^k\},
\end{equation*}
where $\mathcal{V}_n^k$ is a $2^{-(k+1)}$-net of the unit sphere, and, for any $y$, $P_{\mathcal{E}_n}(y)$ is a point in $\mathcal{E}_n$ whose distance to $y$ is minimal. From \citep[Lemma 5.2]{vershynin}, this implies that we can choose $\mathcal{M}_n^k$ such that
\begin{equation}
\Card\mathcal{M}_n^k \leq  \left(1+\frac{2}{2^{-(k+1)}}\right)^{2n}
\leq 2^{2n(k+3)}.\label{eq:Card_Mnk}
\end{equation}

For any $x\in\C^n$, we set
\begin{equation*}
F(x)=\E\left(\scal{Ax_0}{b\odot\phase(Ax)}\right)
\end{equation*}
(where the expectation denotes the expectation over $A$ with $x_0$ and $x$ fixed).

The main difficulty consists in showing that $\scal{Ax_0}{b\odot\phase(Ax)}$ is  close to its expectation for all $x\in\mathcal{M}_n^K$, with $K\in\N^*$ relatively large. This is what the following lemma does; it is proved in Paragraph \ref{sss:ecart_net}.
\begin{lem}\label{lem:ecart_net}
For any $\eta,\mathcal{A}>0$, there exist $c,C_1,C_2,M>0$ such that, when $m\geq Mn^2$, for any $k\in\N$ such that $k\leq \mathcal{A} \log m-c$, with probability at least
\begin{equation*}
1-C_1\exp(-C_2 m^{1/2}),
\end{equation*}
the following property holds: for any $x\in\mathcal{M}_n^k,y\in\mathcal{M}_n^{k+1}$ such that $||x-y||\leq 2^{-(k-1)}$,
\begin{equation*}
|\left(\scal{Ax_0}{b\odot\phase(Ax)}-F(x)\right)-
\left(\scal{Ax_0}{b\odot\phase(Ay)}-F(y)\right)| \leq \frac{\eta}{(k+1)^2}\frac{m}{\sqrt{n}} ||x_0||^2 .
\end{equation*}
In the case $k=0$, we additionally have, with the same probability: for all $x\in\mathcal{M}_n^0$,
\begin{equation*}
|\left(\scal{Ax_0}{b\odot\phase(Ax)}-F(x)\right)| \leq \eta\frac{m}{\sqrt{n}}
||x_0||^2.
\end{equation*}
\end{lem}

Let $\eta,\mathcal{A}>0$ be temporarily fixed. We set $K=\lceil \mathcal{A}\log m-c\rceil$. The event described in the previous lemma holds for all $k\leq K-1$ with probability at least $1-K C_1\exp(-C_2 m^{1/2})$.

For any $x\in\mathcal{M}_n^K$, there exists a sequence $(y_0,y_1,\dots,y_{K-1},y_K)$ such that
\begin{gather*}
y_K=x;\\
\forall k\leq K, y_k\in \mathcal{M}_n^k;\\
\forall k\leq K-1,||y_k-y_{k+1}||\leq 2^{-k}.
\end{gather*}
So when the event of Lemma \ref{lem:ecart_net} holds, we have, for any $x\in\mathcal{M}_n^K$,
\begin{align*}
&\left|\scal{Ax_0}{b\odot\phase(Ax)}-F(x)\right|\\
&\quad\leq \left|\scal{Ax_0}{b\odot\phase(Ay_0)}-F(y_0)\right|\\
&\quad\quad + \sum_{k=0}^{K-1}\left|\left(\scal{Ax_0}{b\odot\phase(Ay_k)}-F(y_k)\right)
-\left(\scal{Ax_0}{b\odot\phase(Ay_{k+1})}-F(y_{k+1})\right)
\right|\\
&\quad\leq \eta \frac{m}{\sqrt{n}}\left(1+\sum_{k=0}^{K-1}\frac{1}{(k+1)^2}\right)||x_0||^2\\
&\quad \leq\eta\left(1+\frac{\pi^2}{6}\right)\frac{m}{\sqrt{n}}||x_0||^2.
\end{align*}

To conclude, we only have to evaluate $F$. This is done by the following lemma, proven in Paragraph \ref{sss:F}.
\begin{lem}\label{lem:F}
There exist $\delta>0$ such that, for any $x\in\mathcal{E}_n$,
\begin{equation*}
|F(x)|\geq (1+\delta)m\frac{||x_0||}{||x||}|\scal{x_0}{x}|.
\end{equation*}
\end{lem}
We combine this lemma and the equation before the lemma: with probability at least $1-KC_1\exp(-C_2m^{1/2})$, for any $x\in\mathcal{M}_n^K$,
\begin{align*}
|\scal{Ax_0}{b\odot\phase(Ax)}|
&\geq |F(x)| - \eta\left(1+\frac{\pi^2}{6}\right)\frac{m}{\sqrt{n}}||x_0||^2\\
&\geq (1+\delta)m\frac{||x_0||}{||x||}|\scal{x_0}{x}| - \eta\left(1+\frac{\pi^2}{6}\right)\frac{m}{\sqrt{n}}||x_0||^2\\
&\geq \left(1+\delta-\eta\left(1+\frac{\pi^2}{6}\right)\right)m\frac{||x_0||}{||x||}|\scal{x_0}{x}|.
\end{align*}
For the last inequality, we have used the fact that $x\in\mathcal{E}_n$, so $|\scal{x_0}{x}|\geq ||x_0||\,||x||/\sqrt{n}$.

We can choose $\eta>0$ sufficiently small so that $1+\delta-\eta\left(1+\frac{\pi^2}{6}\right)>1+\frac{\delta}{2}$. We fix $\mathcal{A}$ to be any real number larger than $\alpha/\log 2$. Then, from the definition of $K$,
\begin{equation*}
2^{-K}\leq 2^{-\mathcal{A}\log m+c}= 2^c m^{-\mathcal{A}\log 2}\leq 2^cm^{-\alpha}.
\end{equation*}
As $K\leq \mathcal{A}\log m-c+1$, we can upper bound $1-KC_1\exp(-C_2m^{1/2})$ by $1-C_1'\exp(-C_2'm^{1/2})$, for $C'_1,C'_2>0$ well-chosen. If we summarize, we get that, with probability at least $1-C_1'\exp(-C_2'm^{1/2})$,
\begin{align*}
\forall x\in\mathcal{M}_n^K,\quad\quad
|\scal{Ax_0}{b\odot\phase(Ax)}|
\geq \left(1+\frac{\delta}{2}\right)m\frac{||x_0||}{||x||}|\scal{x_0}{x}|,
\end{align*}
and $\mathcal{M}_n^K$ is a $2^cm^{-\alpha}$-net of $\mathcal{E}_n$. The lemma is proved.
\end{proof}

\subsubsection{Proof of Lemma \ref{lem:ecart_net}\label{sss:ecart_net}}

\begin{lem*}[Lemma \ref{lem:ecart_net}]
For any $\eta,\mathcal{A}>0$, there exist $c,C_1,C_2,M>0$ such that, when $m\geq Mn^2$, for any $k\in\N$ such that $k\leq \mathcal{A} \log m-c$, with probability at least
\begin{equation*}
1-C_1\exp(-C_2 m^{1/2}),
\end{equation*}
the following property holds: for any $x\in\mathcal{M}_n^k,y\in\mathcal{M}_n^{k+1}$ such that $||x-y||\leq 2^{-(k-1)}$,
\begin{equation*}
|\left(\scal{Ax_0}{b\odot\phase(Ax)}-F(x)\right)-
\left(\scal{Ax_0}{b\odot\phase(Ay)}-F(y)\right)| \leq \frac{\eta}{(k+1)^2}\frac{m}{\sqrt{n}}||x_0||^2 .
\end{equation*}
In the case $k=0$, we additionally have, with the same probability: for all $x\in\mathcal{M}_n^0$,
\begin{equation*}
|\left(\scal{Ax_0}{b\odot\phase(Ax)}-F(x)\right)| \leq \eta\frac{m}{\sqrt{n}}||x_0||^2.
\end{equation*}
\end{lem*}

\begin{proof}[Proof of Lemma \ref{lem:ecart_net}]
We only prove the first part of the lemma. The proof of the second one follows the same principle.

As our expressions are all homogeneous in $x_0$, we can assume that $||x_0||=1$.

For any $j=1,\dots,m$, let us denote by $a_j^*$ the $j$-th line of $A$. We have
\begin{equation*}
\scal{Ax_0}{b\odot\phase(Ax)}=
\sum_{j=1}^m|a_j^*x_0|^2\phase(a_j^*x)\phase(\overline{a_j^*x_0}).
\end{equation*}
As all the $a_j^*$ are identically distributed,
\begin{equation*}
\forall j,\quad\quad
\E\left(|a_j^*x_0|^2\phase(a_j^*x)\phase(\overline{a_j^*x_0})\right)
=\frac{1}{m}\E\scal{Ax_0}{b\odot\phase(Ax)}=\frac{1}{m}F(x).
\end{equation*}
So for any fixed $x,y$, we have
\begin{align}
\left(\scal{Ax_0}{b\odot\phase(Ax)}-F(x)\right)&-
\left(\scal{Ax_0}{b\odot\phase(Ay)}-F(y)\right)\nonumber\\
&=\sum_{j=1}^m \left(|a_j^*x_0|^2Z_j-\E(|a_j^*x_0|^2Z_j)\right),
\label{eq:sum_Ajx0_Z}
\end{align}
with
\begin{equation*}
Z_j=\phase(a_j^*x)\phase(\overline{a_j^*x_0})-\phase(a_j^*y)\phase(\overline{a_j^*x_0}).
\end{equation*}
Were there no terms ``$|a_j^*x_0|^2$'' in Equation \eqref{eq:sum_Ajx0_Z}, we could apply Bennett's concentration inequality: the random variables $Z_j$ are bounded by $2$ in modulus, and, as we are going to see, their variance is small if $x$ and $y$ are close. Bennett's inequality would then guarantee that the term in Equation \eqref{eq:sum_Ajx0_Z} is small with high probability. Unfortunately, the $|a_j^*x_0|^2$ are not almost surely bounded, so we cannot directly apply Bennett's inequality.

To overcome this problem, we first condition over $Ax_0$. When conditioned over $Ax_0$, the random variables $|a_j^*x_0|^2 Z_j$ are almost surely bounded; we will prove that they still have a small variance. We still cannot directly apply Bennett's inequality, because the bounds depend on $j$, but we can adapt its proof, and get a concentration inequality for the following sum:
\begin{equation*}
\sum_{j=1}^m|a_j^*x_0|^2Z_j - |a_j^*x_0|^2\E(Z_j|Ax_0).
\end{equation*}
After that, we will also need to derive a concentration inequality for
\begin{equation*}
\sum_{j=1}^m|a_j^*x_0|^2\E(Z_j|Ax_0)-\E(|a_j^*x_0|^2Z_j),
\end{equation*}
but it will be easier.

The first step is to control the distribution of the $|a_j^*x_0|$. The idea is that there are a few indexes $j$ for which $|a_j^*x_0|$ is large, but these are sufficiently rare so that the sum $\sum_j|a_j^*x_0|^2Z_j$, when conditioned over $Ax_0$, essentially behaves as if all random variables were bounded by the same constant.

The proof of the following lemma is in Paragraph \ref{sss:dist_Ajx0}.

\begin{lem}\label{lem:dist_Ajx0}
For some constants $C_1,C_2>0$, the following event happens with probability at least $1-C_1e^{-C_2 \sqrt{m}}$: for any $s\in\{1,\dots,\lfloor m^{1/4}\rfloor\}$,
\begin{align*}
\Card\left\{j\in\{1,\dots,m\},|a_j^*x_0|\geq s\right\}&\leq \frac{m}{s^2} \max(m^{-1/2},e^{-s^2/2})
\end{align*}
and
\begin{align*}
\Card\left\{j\in\{1,\dots,m\},|a_j^*x_0|> m^{1/4}\right\}&=0.
\end{align*}
\end{lem}
Let us denote by $\mathcal{E}_0$ the event described in the previous lemma:
\begin{align}
\mathcal{E}_0=\Big(
\forall s\in\{1,\dots,\lfloor m^{1/4}\rfloor\}, &\Card\left\{j\in\{1,\dots,m\},|a_j^*x_0|\geq s\right\} \leq \frac{m}{s^2} \max(m^{-1/2},e^{-s^2/2});\nonumber\\
\mbox{and }&\Card\left\{j\in\{1,\dots,m\},|a_j^*x_0|> m^{1/4}\right\}=0
\Big).\label{eq:def_E0}
\end{align}
The second step is to get an upper bound on the variance of the $Z_j$, conditioned by $Ax_0$. The proof of the following lemma is in Paragraph \ref{sss:var_bound}.
\begin{lem}\label{lem:var_bound}
There exists a constant $C>0$ depending only on $\epsilon$ such that, for any fixed unit-normed $x,y$ such that
\begin{equation}\label{eq:var_cond_xy}
|\scal{x_0}{x}|\leq (1-\epsilon)||x_0||\,||x||\quad\mbox{and}\quad
|\scal{x_0}{y}|\leq (1-\epsilon)||x_0||\,||y||,
\end{equation}
we have, for any $j$,
\begin{equation*}
\Var(Z_j|Ax_0)\leq C\left(1+\frac{|a_j^*x_0|^2}{||x_0||^2}\right)
||x-y||^2\log\left(4||x-y||^{-1}\right).
\end{equation*}
\end{lem}
From the previous lemma, we deduce that, if $x\in\mathcal{M}_n^k,y\in\mathcal{M}_n^{k+1}$ are fixed and satisfy $||x-y||\leq 2^{-(k-1)}$, we have
\begin{gather}
\Var(\Re Z_j|Ax_0)\leq \Var(Z_j|Ax_0)\leq C'\left(1+|a_j^*x_0|^2\right) \gamma^{-2k},\label{eq:var_bound}\\
\Var(\Im Z_j|Ax_0)\leq \Var(Z_j|Ax_0)\leq C'\left(1+|a_j^*x_0|^2\right) \gamma^{-2k}
\nonumber.
\end{gather}
where $\gamma$ can be any real number in $]1;2[$, and $C'>0$ is a large enough constant (depending on $\gamma$).

To follow the proof of Bennett's inequality, we now have to upper bound, for suitable values of $\lambda>0$,
\begin{align*}
&\E\left(\exp\left(\lambda \sum_{j=1}^m\Re\left(|a_j^*x_0|^2Z_j-|a_j^*x_0|^2\E(Z_j|Ax_0)\right)\right)\Bigg| Ax_0\right)\\
&\hskip 7cm
=\prod_{j=1}^m \E(e^{\lambda \Re (|a_j^*x_0|^2 Z_j-|a_j^*x_0|^2 \E(Z_j|Ax_0))}\Big| Ax_0).
\end{align*}
We use here the fact that, even when conditioned on $Ax_0$, the $Z_j$ are independent random variables.

The upper bound relies on the following lemma, proven in Paragraph \ref{sss:esp_exp_bound}.
\begin{lem}\label{lem:esp_exp_bound}
Let $Z$ be any real random variable such that $|Z|\leq 2$ with probability $1$. If we set $\sigma^2=\Var(Z)$, then, for any $\lambda\in\R^+$,
\begin{equation*}
\E\left(e^{\lambda (Z-\E(Z))}\right) \leq 1+\frac{\sigma^2}{16}\left(e^{4\lambda}-1-4\lambda\right).
\end{equation*}
\end{lem}

From Equation \eqref{eq:var_bound} and the previous lemma, for any $\lambda\geq 0$,
\begin{align}
\E&\Big(e^{\lambda |a_j^*x_0|^2\Re(Z_j-\E(Z_j|Ax_0))}\Big|Ax_0\Big)\nonumber\\
&\leq 1+ \frac{C'(1+|a_j^*x_0|^2)\gamma^{-2k}}{16}\left(e^{4|a_j^*x_0|^2\lambda}-1-4|a_j^*x_0|^2\lambda\right).\nonumber
\end{align}
So we can upper bound
\begin{align}
&\E\left(\exp\left(\lambda \sum_{j=1}^m\Re\left(|a_j^*x_0|^2Z_j-|a_j^*x_0|^2\E(Z_j|Ax_0)\right)\right)\Bigg| Ax_0\right)
\nonumber\\
&\hskip 1cm
\leq \exp\left(\sum_{j=1}^m\log\left(1+ \frac{C'(1+|a_j^*x_0|^2)\gamma^{-2k}}{16}\left(e^{4|a_j^*x_0|^2\lambda}-1-4|a_j^*x_0|^2\lambda\right)\right) \right).
\label{eq:exp_bound}
\end{align}
On the event $\mathcal{E}_0$ defined in Equation \eqref{eq:def_E0}, we can simplify the sum inside the exponential. Specifically, if we define the function
\begin{align*}
P:s\in\R^+\to \frac{m}{\max(s,1)^2} \max(m^{-1/2},e^{-s^2/2}),
\end{align*}
we have that, on the event $\mathcal{E}_0$, for any non-decreasing function $f:\R^+\to\R^+$,
\begin{align*}
\sum_{j=1}^m f(|a_j^*x_0|)
&\leq \sum_{s=1}^{+\infty} f(s)\Big(\Card\{j,|a_j^*x_0|\geq s-1\}-\Card\{j,|a_j^*x_0|\geq s\}\Big)\\
&=\sum_{s=0}^{+\infty}(f(s+1)-f(s))\Card\{j,|a_j^*x_0|\geq s\}+mf(0)\\
&\leq \sum_{s=1}^{\lfloor m^{1/4}\rfloor}(f(s+1)-f(s))P(s)+P(0)f(0)\\
&\leq \sum_{s=1}^{\lfloor m^{1/4}\rfloor} f(s)\left(P(s-1)-P(s)\right)+f(\lfloor m^{1/4}\rfloor+1)P(\lfloor m^{1/4}\rfloor)\\
&= \sum_{s= 1}^{\lfloor m^{1/4}\rfloor}f(s)\int_{s}^{s+1} (-P'(t-1))dt+f(\lfloor m^{1/4}\rfloor+1)P(\lfloor m^{1/4}\rfloor)\\
&\leq \int_1^{ m^{1/4}+1}f(t)(-P'(t-1))dt+f(\lfloor m^{1/4}\rfloor+1)P(\lfloor m^{1/4}\rfloor).
\end{align*}
By a direct computation, we see that, if $C>0$ is properly chosen, we can bound:
\begin{align*}
-P'(s-1)&\leq C\frac{m}{s^2}e^{-s^2/4}\mbox{ if }s\leq \sqrt{\log m}+1,\\
&\leq C\frac{m^{1/2}}{s^3}\mbox{ if }\sqrt{\log m}+1<s\leq m^{1/4}+1;\\
P(\lfloor m^{1/4}\rfloor)&\leq C.
\end{align*}
So
\begin{align*}
\frac{1}{m}\sum_{j=1}^m f(|a_j^*x_0|)
\leq C\left(\int_1^{\sqrt{\log m}+1}\frac{f(t)}{t^2}e^{-t^2/4}dt+
m^{-1/2}\int_{\sqrt{\log m}+1}^{m^{1/4}+1}\frac{f(t)}{t^3}dt\right)+\frac{C}{m}f(m^{1/4}+1).
\end{align*}
We plug this inequality into Equation \eqref{eq:exp_bound}. For any $\lambda\geq 0$, we set
\begin{equation*}
f_\lambda(x)=\log\left(1+\frac{C'(1+x^2)\gamma^{-2k}}{16}(e^{4\lambda x^2}-1-4\lambda x^2)\right),
\end{equation*}
and, on the event $\mathcal{E}_0$, we have:
\begin{align}
&\E\left(\exp\left(\lambda \sum_{j=1}^m\Re\left(|a_j^*x_0|^2Z_j-|a_j^*x_0|^2\E(Z_j|Ax_0)\right)\right)\Bigg| Ax_0\right)
\nonumber\\
&\hskip 1cm\leq \exp\left(Cm\left(\int_1^{\sqrt{\log m}+1}\frac{f_{\lambda}(t)}{t^2}e^{-t^2/4}dt+
m^{-1/2}\int_{\sqrt{\log m}+1}^{m^{1/4}+1}\frac{f_{\lambda}(t)}{t^3}dt
+\frac{1}{m}f(m^{1/4}+1)\right)\right).\label{eq:bound_exp_int}
\end{align}
We upper bound the sum of the integrals, using standard analysis techniques. The detailed proof is in Paragraph \ref{sss:eval_lambda_small}.
\begin{lem}\label{lem:eval_lambda_small}
There exists a constant $\tilde C>0$ depending only on $\gamma$ and $\epsilon>0$ such that, for any $\lambda\in]0;\frac{1}{40}[$,
\begin{equation*}
\int_1^{\sqrt{\log m}+1}\frac{f_{\lambda}(t)}{t^2}e^{-t^2/4}dt+
m^{-1/2}\int_{\sqrt{\log m}+1}^{m^{1/4}+1}\frac{f_{\lambda}(t)}{t^3}dt+
\frac{1}{m}f_{\lambda}(m^{1/4}+1)
\leq \tilde C \gamma^{-2k}\lambda^2,
\end{equation*}
provided that
\begin{subequations}
\begin{gather}
\Big(\log(\max(1,\gamma^k/\lambda))+1\Big)\left(\frac{\gamma^k}{\lambda}\right)^{4/3}\leq m^{1/2};\label{eq:eval_cond1}\\
\frac{m^{1/2}\lambda \gamma^{-2k}}{1+\log m} \geq 1.\label{eq:eval_cond2}
\end{gather}
\end{subequations}
\end{lem}
We apply this result with
\begin{equation*}
\lambda = \frac{\eta \gamma^{2k}}{8C\tilde C (k+1)^2m^{1/4}},
\end{equation*}
where $\eta>0$ is the fixed constant chosen in the statement of Lemma \ref{lem:ecart_net}, $C$ is the constant of Equation \eqref{eq:bound_exp_int} and $\tilde C$ is the one of Lemma \ref{lem:eval_lambda_small}. We consider only the values of $k\in\N$ such that
\begin{equation}\label{eq:cond_k_1}
\gamma^{2k}< \frac{C\tilde C}{5 \eta}m^{1/4},
\end{equation}
which in particular ensures that
\begin{equation*}
\lambda< \frac{1}{40}.
\end{equation*}
With this definition, Conditions \eqref{eq:eval_cond1} and \eqref{eq:eval_cond2} are satisfied. Indeed, as $\gamma>1$,
\begin{gather*}
\frac{\gamma^k}{\lambda}= \frac{8C\tilde C(k+1)^2m^{1/4}}{\eta \gamma^k}=O(m^{1/4});\\
\Rightarrow\quad\quad
\left(\log(\max(1,\gamma^k/\lambda))+1\right)\left(\frac{\gamma^k}{\lambda}\right)^{4/3} = O(m^{1/3}\log m) \leq m^{1/2},
\end{gather*}
if $m$ is large enough. For the second condition, because of Equation \eqref{eq:cond_k_1},
\begin{align*}
\frac{m^{1/2}\lambda\gamma^{-2k}}{1+\log m}
&=\frac{m^{1/2}}{1+\log m} \frac{\eta}{8C\tilde C(k+1)^2m^{1/4}}\\
&\geq \frac{m^{1/4}}{1+\log m}\frac{\eta}{8CC'\left(1+\log(C\tilde Cm^{1/4}/(5\eta))/(2\log(\gamma))\right)^2}\\
&\geq c\frac{m^{1/4}}{(1+\log m)^3}\\
&\geq 1,
\end{align*}
if $m$ is large enough. (In the second inequality, $c>0$ is a positive constant.)

As the two conditions are satisfied, we can combine Lemma \ref{lem:eval_lambda_small} and Equation \eqref{eq:bound_exp_int}. We get that, on the event $\mathcal{E}_0$,
\begin{align*}
\E\left(\exp\left(\lambda \sum_{j=1}^m\Re\left(|a_j^*x_0|^2Z_j-|a_j^*x_0|^2\E(Z_j|Ax_0)\right)\right)\Bigg| Ax_0\right)
&\leq
\exp\left(C\tilde C m \gamma^{-2k}\lambda^2\right)\\
&=\exp\left( \frac{\eta^2 \gamma^{2k} m^{1/2}}{64 C\tilde C(k+1)^4} \right).
\end{align*}
So, by Markov's inequality, on the event $\mathcal{E}_0$, if $m\geq n^2$,
\begin{align*}
P&\left(\sum_{j=1}^m\Re \left(|a_j^*x_0|^2Z_j-|a_j^*x_0|^2\E(Z_j|Ax_0)\right)\geq \frac{\eta m}{4(k+1)^2\sqrt{n}} \Bigg| Ax_0 \right)\\
&\leq P\left(\sum_{j=1}^m\Re \left(|a_j^*x_0|^2Z_j-|a_j^*x_0|^2\E(Z_j|Ax_0)\right)\geq \frac{\eta m^{3/4}}{4(k+1)^2} \Bigg| Ax_0 \right)\\
&\leq \exp\left( \frac{\eta^2 \gamma^{2k} m^{1/2}}{64C\tilde C(k+1)^4 } \right)\exp\left(-\frac{\lambda\eta m^{3/4}}{4(k+1)^2 }\right)\\
&= \exp\left(-\frac{\eta^2\gamma^{2k} m^{1/2}}{64C\tilde C(k+1)^4}\right).
\end{align*}
We integrate over $Ax_0$, and obtain
\begin{align*}
P&\left(\mathcal{E}_0\cap\left\{\sum_{j=1}^m\Re \left(|a_j^*x_0|^2Z_j-|a_j^*x_0|^2\E(Z_j|Ax_0)\right)\geq \frac{\eta m}{4(k+1)^2\sqrt{n}} \right\} \right)\\
&\leq \exp\left(-\frac{\eta^2\gamma^{2k} m^{1/2}}{16C\tilde C(k+1)^4}\right).
\end{align*}
We can apply the same reasoning to $-\Re(Z_j),\Im(Z_j)$ and $-\Im(Z_j)$. This yields:
\begin{align}
P&\left(\mathcal{E}_0\cap\left\{\left|\sum_{j=1}^m\left(|a_j^*x_0|^2Z_j-|a_j^*x_0|^2\E(Z_j|Ax_0)\right)\right|\geq \frac{\eta m}{2(k+1)^2\sqrt{n}} \right\} \right)\nonumber\\
&\leq 4\exp\left(-\frac{\eta^2\gamma^{2k} m^{1/2}}{16C\tilde C(k+1)^4}\right).
\label{eq:concentration_part1}
\end{align}
Now that we have a bound for $\sum_{j=1}^m\left(|a_j^*x_0|^2Z_j-|a_j^*x_0|^2\E(Z_j|Ax_0)\right)$, we remember that we also have to bound
\begin{equation*}
\sum_{j=1}^m\left(|a_j^*x_0|^2\E(Z_j|Ax_0)-\E(|a_j^*x_0|^2Z_j)\right).
\end{equation*}
We remark that, for all $j$, $\E(Z_j|Ax_0)=\E(Z_j|a_j^*x_0)$, so that the random variables $|a_j^*x_0|^2\E(Z_j|Ax_0)$, for $j=1,\dots,m$, are independent and identically distributed.

We begin with the following lemma, proven in Paragraph \ref{sss:maj_esp}.
\begin{lem}\label{lem:maj_esp}
There exist a constant $C>0$ depending only on $\epsilon$ such that, for any fixed unit-normed $x,y$ such that
\begin{equation*}
|\scal{x_0}{x}|\leq (1-\epsilon)||x_0||\,||x||\quad\mbox{and}\quad
|\scal{x_0}{y}|\leq (1-\epsilon)||x_0||\,||y||,
\end{equation*}
and any $j=1,\dots,m$,
\begin{equation*}
|\E(Z_j|a_j^*x_0)| \leq C \min\left(1,||x-y||\left(1+\frac{|a_j^*x_0|}{||x_0||}\right)\right).
\end{equation*}
\end{lem}
To simplify the expressions, we still assume that $||x_0||=1$. If $||x-y||\leq 2^{-(k-1)}$, the previous lemma guarantees that, for any $j$,
\begin{equation}\label{eq:maj_esp_gamma}
|\E(Z_j|a_j^* x_0)|\leq 2C \min\left(1,\gamma^{-k}(1+|a_j^*x_0|)\right),
\end{equation}
where $\gamma$ is still our real number in $]1;2[$.

This inequality allows us to upper bound $\E\left(e^{\lambda\left(|a_j^*x_0|^2\E(Z_j|Ax_0)-\E(|a_j^*x_0|^2Z_j)\right)}\right)$, for $\lambda$ small enough. The next lemma is proved in Paragraph \ref{sss:esp_exp_bound2}.
\begin{lem}\label{lem:esp_exp_bound2}
There exist constants $c,C'>0$, that depend only on $\gamma$ and $\epsilon$, such that, for any $\lambda\in[-c;c]$,
\begin{align*}
\log\left(\E\left(e^{\lambda\Re\left(|a_j^*x_0|^2\E(Z_j|Ax_0)-\E(|a_j^*x_0|^2Z_j)\right)}\right) \right)
&\leq C' \lambda^2 \gamma^{-2k},\\
\mbox{and }
\log\left(\E\left(e^{\lambda\Im\left(|a_j^*x_0|^2\E(Z_j|Ax_0)-\E(|a_j^*x_0|^2Z_j)\right)}\right) \right)
&\leq C' \lambda^2 \gamma^{-2k}.
\end{align*}
\end{lem}
So by Markov's inequality, taking
\begin{equation*}
\lambda=\frac{\eta\gamma^{2k}}{8C'(k+1)^2m^{1/4}},
\end{equation*}
for $k$ such that
\begin{equation}\label{eq:cond_k_2}
\gamma^{2k}\leq \frac{8cC'}{\eta}m^{1/4},
\end{equation}
we have, when $m\geq n^2$,
\begin{align*}
P&\left(\Re\left(\sum_{j=1}^m\left(|a_j^*x_0|^2\E(Z_j|Ax_0)-\E(|a_j^*x_0|^2Z_j)\right)\right)\geq \frac{\eta m}{4(k+1)^2\sqrt{n}} \right)\\
&\leq P\left(\Re\left(\sum_{j=1}^m\left(|a_j^*x_0|^2\E(Z_j|Ax_0)-\E(|a_j^*x_0|^2Z_j)\right)\right)\geq \frac{\eta m^{3/4}}{4(k+1)^2} \right)\\
&\leq \E\left(\exp\left(\lambda\Re\left(\sum_{j=1}^m\left(|a_j^*x_0|^2\E(Z_j|Ax_0)-\E(|a_j^*x_0|^2Z_j)\right)\right)\right)\right)
\exp\left(-\frac{\lambda\eta m^{3/4}}{4(k+1)^2}\right)\\
&=\exp\left(m \log\left(\E\left(e^{\lambda\Re\left(|a_j^*x_0|^2\E(Z_j|Ax_0)-\E(|a_j^*x_0|^2Z_j)\right)}\right) \right)- \frac{\lambda\eta m^{3/4}}{4(k+1)^2} \right)\\
&\leq\exp\left(m C'\lambda^2\gamma^{-2k}- \frac{\lambda\eta m^{3/4}}{4(k+1)^2} \right)\\
&=\exp\left(-\frac{m^{1/2}\eta^2\gamma^{2k}}{64 C'(k+1)^4}\right).
\end{align*}
The same inequality holds if we replace $\Re$ by $-\Re,\Im$ or $-\Im$, so we obtain:
\begin{align*}
P&\left(\left|\sum_{j=1}^m\left(|a_j^*x_0|^2\E(Z_j|Ax_0)-\E(|a_j^*x_0|^2Z_j)\right)\right|\geq \frac{\eta m}{2(k+1)^2\sqrt{n}} \right)
\leq 4\exp\left(-\frac{m^{1/2}\eta^2\gamma^{2k}}{64 C'(k+1)^4}\right).
\end{align*}
We are close to the end. The previous equation, combined with Equation \eqref{eq:concentration_part1} yields, by triangular inequality, that for any fixed $x\in\mathcal{M}_n^k,y\in\mathcal{M}_n^{k+1}$ such that $||x-y||\leq 2^{-(k-1)}$,
\begin{align*}
P&\left(\mathcal{E}_0\cap\left\{\left|\sum_{j=1}^m\left(|a_j^*x_0|^2Z_j-\E(|a_j^*x_0|^2Z_j)\right)\right|
\geq \frac{\eta m}{(k+1)^2\sqrt{n}} \right\} \right)\\
&\leq 8\exp\left(-\mathcal{C}\frac{\gamma^{2k}}{(k+1)^4}m^{1/2}\right),
\end{align*}
where $\mathcal{C}$ is a constant that depends only on $\eta,\epsilon$ and $\gamma$. We recall that $Z_j$ depends on $x$ and $y$, although it does not appear in the notation.

From Equation \eqref{eq:Card_Mnk},
\begin{equation*}
\Card\mathcal{M}_n^k\leq 2^{2n(k+3)}\quad\mbox{and}\quad
\Card\mathcal{M}_n^{k+1}\leq 2^{2n(k+4)}.
\end{equation*}
The number of possible pairs $(x,y)\in\mathcal{M}_n^k\times\mathcal{M}_n^{k+1}$ is then bounded by
\begin{equation*}
2^{2n(2k+7)}\leq e^{10n(k+1)},
\end{equation*}
and by union bound,
\begin{align*}
  P&\left(\mathcal{E}_0\cap\left\{\exists x,y\in\mathcal{M}_n^k\times\mathcal{M}_n^{k+1}, \left|\sum_{j=1}^m\left(|a_j^*x_0|^2Z_j-\E(|a_j^*x_0|^2Z_j)\right)\right|
\geq \frac{\eta m}{(k+1)^2\sqrt{n}} \right\} \right)\\
&\leq 8\exp\left(-\mathcal{C}\frac{\gamma^{2k}}{(k+1)^4}m^{1/2}+10 n(k+1)\right).
\end{align*}
From Lemma \ref{lem:dist_Ajx0}, the probability of $\mathcal{E}_0$ is at least $1-C_1 e^{-C_2m^{1/2}}$ for some constants $C_1,C_2>0$, so
\begin{align*}
P&\left(\forall x,y\in\mathcal{M}_n^k\times\mathcal{M}_n^{k+1}, \left|\sum_{j=1}^m\left(|a_j^*x_0|^2Z_j-\E(|a_j^*x_0|^2Z_j)\right)\right|
< \frac{\eta m}{(k+1)^2\sqrt{n}}  \right)\\
&\geq 1- 8\exp\left(-\mathcal{C}\frac{\gamma^{2k}}{(k+1)^4}m^{1/2}+10 n(k+1)\right)-C_1\exp(-C_2m^{1/2}).
\end{align*}
There exists a constant $\mathcal{C'}$ depending only on $\gamma$ such that $\gamma^{2k}\geq \mathcal{C}'(k+1)^5$ for any $k\in\N$. If we assume that $m\geq Mn^2$ for some $M>0$, we have
\begin{align*}
P&\left(\forall x,y\in\mathcal{M}_n^k\times\mathcal{M}_n^{k+1}, \left|\sum_{j=1}^m\left(|a_j^*x_0|^2Z_j-\E(|a_j^*x_0|^2Z_j)\right)\right|
< \frac{\eta m}{(k+1)^2\sqrt{n}}  \right)\\
&\geq 1- 8\exp\left(-m^{1/2}(k+1) (\mathcal{C}\mathcal{C}'-10 M^{-1/2})\right)-C_1\exp(-C_2m^{1/2})\\
&\geq 1- 8\exp\left(- (\mathcal{C}\mathcal{C}'-10 M^{-1/2})m^{1/2}\right)-C_1\exp(-C_2m^{1/2}).
\end{align*}
When $M>0$ is large enough, this can be lower bounded by $1-C_1\exp(-C_2m^{1/2})$, where the constants $C_1,C_2>0$ depend on $\eta,\epsilon$ and $\gamma$ but not on $k,m$ or $n$. 

We recall Equations \eqref{eq:cond_k_1} and \eqref{eq:cond_k_2}: the reasoning holds only for the values of $k$ such that
\begin{equation*}
\gamma^{2k}< \alpha m^{1/4},
\end{equation*}
where, again, $\alpha>0$ is a constant that depends only on $\eta,\epsilon$ and $\gamma$. This means that, if we have chosen $\gamma\in]1;2[$ sufficiently close to $1$, it holds for any $k$ satisfying
\begin{equation*}
k < \mathcal{A} \ln m - c,
\end{equation*}
where $c\in\R$ is a constant that does not depend on $n$ or $m$.

\end{proof}

\subsubsection{Proof of Lemma \ref{lem:F}\label{sss:F}}

\begin{lem*}[Lemma \ref{lem:F}]
There exist $\delta>0$ such that, for any $x\in\mathcal{E}_n$,
\begin{equation*}
|F(x)|\geq (1+\delta)m\frac{||x_0||}{||x||}|\scal{x_0}{x}|.
\end{equation*}
\end{lem*}
\begin{proof}
We write
\begin{equation*}
x=\alpha x_0 + \beta x',
\end{equation*}
with $\alpha,\beta\in\C$ and $x'\in \C^n$ such that $\scal{x_0}{x'}=0$ and $||x'||=1$.

\begin{align*}
F(x)&=\E(\scal{Ax_0}{b\odot\phase(Ax)})\\
&=\sum_{j=1}^m\E\left(\overline{(Ax_0)_j}|(Ax_0)_j|\phase((Ax)_j)\right)\\
&=m \E\left(\overline{(Ax_0)_1}|(Ax_0)_1|\phase((Ax)_1)\right)\\
&=m \E\left(\overline{(Ax_0)_1}|(Ax_0)_1|\phase(\alpha (Ax_0)_1 + \beta (Ax')_1)\right)\\
&=m ||x_0||^2 \phase(\alpha) \E\left(\frac{\overline{(Ax_0)_1}}{||x_0||}\frac{|(Ax_0)_1|}{||x_0||}\phase\left(\frac{(Ax_0)_1}{||x_0||} + \frac{\beta}{\alpha||x_0||} (Ax')_1\right)\right)\\
&=m||x_0||^2\phase(\alpha)\E\left(\overline{Z_1}|Z_1|\phase\left(Z_1+ \frac{|\beta|}{|\alpha|\,||x_0||} Z_2 \right)\right).
\end{align*}
where $Z_1=\frac{(Ax_0)_1}{||x_0||}$ and $Z_2=\phase(\beta/\alpha)(Ax')_1$ are independent complex Gaussian variables with variance $1$.

The expectation cannot be analytically computed, but it can be lower bounded by a simple function. The following lemma is proven in Paragraph \ref{sss:min_f}.
\begin{lem}\label{lem:min_f}
For any $t\in\R^+$, we set
\begin{equation*}
f(t)=\E\left(\overline{Z_1}|Z_1|\phase\left(Z_1+t Z_2\right)\right).
\end{equation*}
The function $f$ is real-valued. For any $\gamma>0$, there exist $\delta>0$ such that
\begin{equation*}
\forall t\in[\gamma;+\infty[,\quad\quad
f(t)\geq \frac{1+\delta}{\sqrt{1+t^2}}.
\end{equation*}
\end{lem}

As $x$ belongs to $\mathcal{E}_n$, we have:
\begin{align*}
\frac{|\beta|}{|\alpha|\,||x_0||}
&=\frac{\sqrt{||x||^2-|\alpha|^2||x_0||^2}}{|\alpha|\,||x_0||}\\
&=\sqrt{\frac{1}{|\alpha|^2||x_0||^2}-1}\\
&=\sqrt{\frac{||x_0||^2}{|\scal{x_0}{x}|^2}-1}\\
&\geq \sqrt{\frac{1}{(1-\epsilon)^2}-1}.
\end{align*}
Consequently, we can apply the lemma with $\gamma = \sqrt{\frac{1}{(1-\epsilon)^2}-1}$. It implies that, for some $\delta>0$ that depends only on $\epsilon$,
\begin{align*}
|F(x)|&\geq m||x_0||^2(1+\delta)\frac{1}{\sqrt{1+\left(\frac{|\beta|}{|\alpha|\,||x_0||}\right)^2}}\\
&= m||x_0||^2(1+\delta)|\alpha|\,||x_0||\\
&=(1+\delta)m \frac{||x_0||}{||x||} |\scal{x_0}{x}|.
\end{align*}
\end{proof}

\subsubsection{Proof of Lemma \ref{lem:dist_Ajx0}\label{sss:dist_Ajx0}}

\begin{lem*}[Lemma \ref{lem:dist_Ajx0}]
For some constants $C_1,C_2>0$, the following event happens with probability at least $1-C_1e^{-C_2 \sqrt{m}}$: for any $s\in\{1,\dots,\lfloor m^{1/4}\rfloor\}$,
\begin{align*}
\Card\left\{j\in\{1,\dots,m\},|a_j^*x_0|\geq s\right\}&\leq \frac{m}{s^2} \max(m^{-1/2},e^{-s^2/2})
\end{align*}
and
\begin{align*}
\Card\left\{j\in\{1,\dots,m\},|a_j^*x_0|> m^{1/4}\right\}&=0.
\end{align*}
\end{lem*}

\begin{proof}[Proof of Lemma \ref{lem:dist_Ajx0}]

We recall that $A_1x_0,\dots,A_mx_0$ are independent complex Gaussian random variables with variance $||x_0||^2=1$. In particular, for any $s\in\N$,
\begin{align*}
P(|a_j^*x_0|\geq s)&=e^{-s^2};\\
\E(1_{|a_j^*x_0|\geq s})&=e^{-s^2};\\
\Var(1_{|a_j^*x_0|\geq s})&\leq e^{-s^2}.
\end{align*}

We first consider the values of $s$ belonging to $\{1,\dots,\lfloor\sqrt{\log m}\rfloor\}$. For any of these $s$, by Bennett's inequality, if we denote by $h$ the function $h:x\in\R^+\to (1+x)\log(1+x)-x$,
\begin{align*}
&P\left(\Card\left\{j\in\{1,\dots,m\},|a_j^*x_0|\geq s\right\}\geq \frac{m}{s^2}e^{-s^2/2} \right)\\
=&P\left(\sum_{j=1}^m \left(1_{|a_j^*x_0|\geq s}-\E\left(1_{|a_j^*x_0|\geq s}\right)\right)
\geq m\left(\frac{e^{-s^2/2}}{s^2}-e^{-s^2}\right)  \right)\\
\leq&\exp\left(-m e^{-s^2}h\left(\frac{e^{s^2/2}}{s^2}-1\right)\right)\\
=&\exp\left(-m\frac{e^{-s^2/2}}{s^2}\left(\frac{s^2}{2}-2\log(s)-1+s^2 e^{-s^2/2}\right)\right)\\
\leq&\exp\left(-c_1m e^{-s^2/2}\right),
\end{align*}
for some absolute constant $c_1>0$. As $s\leq \sqrt{\log m}$, this yields:
\begin{align*}
P\left(\Card\left\{j\in\{1,\dots,m\},|a_j^*x_0|\geq s\right\}\geq \frac{m}{s^2}e^{-s^2/2} \right)\leq
\exp(-c_1m^{1/2}).
\end{align*}
Second, we consider the values of $s$ in $\{\left\lfloor\sqrt{\log m}\right\rfloor+1,\dots,\lfloor m^{1/4}+1\rfloor\}$.
\begin{align*}
&P\left(\Card\left\{j\in\{1,\dots,m\},|a_j^*x_0|\geq s\right\}\geq \frac{m^{1/2}}{s^2} \right)\\
=&P\left(\sum_{j=1}^m \left(1_{|a_j^*x_0|\geq s}-\E\left(1_{|a_j^*x_0|\geq s}\right)\right)
\geq m\left(\frac{m^{-1/2}}{s^2}-e^{-s^2}\right)  \right)\\
\leq&\exp\left(-me^{-s^2}h\left(\frac{m^{-1/2}}{s^2}e^{s^2}-1\right)\right)\\
=&\exp\left(- m^{1/2}\left(1-\frac{\log m}{2s^2}-\frac{2\log s}{s^2}-\frac{1}{s^2}+m^{1/2}e^{-s^2}
\right)\right)\\
\overset{(a)}{\leq}&\exp\left(-m^{1/2}\left(1-\frac{1}{2}-\frac{\log (\log m)}{\log m}-\frac{1}{\log m}
\right)\right)\\
\leq&\exp\left(-\frac{m^{1/2}}{4}\right).
\end{align*}
as soon as $m$ is large enough. For (a), we have used the inequality $s\geq\sqrt{\log m}$.

To conclude, we observe that, if 
\begin{equation*}
\Card\left\{j\in\{1,\dots,m\},|a_j^*x_0|\geq s\right\}\leq \frac{m^{1/2}}{s^2}
\end{equation*}
for $s=\lfloor m^{1/4}+1\rfloor>m^{1/4}$, we must have
\begin{equation*}
\Card\left\{j\in\{1,\dots,m\},|a_j^*x_0| > m^{1/4} \right\}=0.
\end{equation*}
So we see that the desired event holds, for $m$ large enough, with probability at least
\begin{align*}
1- \sqrt{\log m}e^{-c_1\sqrt{m}}-m^{1/4}e^{-\sqrt{m}/4},
\end{align*}
which can be bounded by $1-C_1e^{-C_2\sqrt{m}}$ for $C_1,C_2>0$ well-chosen.
\end{proof}

\subsubsection{Proof of Lemma \ref{lem:var_bound}\label{sss:var_bound}}

\begin{lem*}[Lemma \ref{lem:var_bound}]
There exists a constant $C>0$ depending only on $\epsilon$ such that, for any fixed unit-normed $x,y$ such that
\begin{equation}\label{eq:var_cond_xy}
|\scal{x_0}{x}|\leq (1-\epsilon)||x_0||\,||x||\quad\mbox{and}\quad
|\scal{x_0}{y}|\leq (1-\epsilon)||x_0||\,||y||,
\end{equation}
we have, for any $j$,
\begin{equation*}
\Var(Z_j|Ax_0)\leq C\left(1+\frac{|a_j^*x_0|^2}{||x_0||^2}\right)
||x-y||^2\log\left(4||x-y||^{-1}\right).
\end{equation*}
\end{lem*}

\begin{proof}[Proof of Lemma \ref{lem:var_bound}]
By the definition of $Z_j$, it suffices to prove
\begin{equation}\label{eq:var_reformulation}
\Var(\phase(a_j^*x)-\phase(a_j^*y)|Ax_0)\leq C\left(1+\frac{|a_j^*x_0|^2}{||x_0||^2}\right) ||x-y||^2\log\left(4||x-y||^{-1}\right).
\end{equation}
We write
\begin{equation*}
x = \alpha_xx_0+x'\mbox{ and }y=\alpha_y x_0+\beta x'+y'',
\end{equation*}
where $\alpha_x,\alpha_y,\beta$ are complex numbers and $x',y''\in\C^n$ satisfy $\scal{x'}{x_0}=\scal{y''}{x_0}=\scal{x'}{y''}=0$. Because of Equation \eqref{eq:var_cond_xy}, and because $x,y$ are unit-normed,
\begin{subequations}
\begin{gather}
||x'||\geq \sqrt{\epsilon(2-\epsilon)}\geq \sqrt{\epsilon};\label{eq:var_x_prime}\\
|\beta-1|=\frac{|\scal{y-x}{x'}|}{||x'||^2}\leq \frac{1}{\sqrt{\epsilon}}||y-x||;\label{eq:var_beta}\\
||\alpha_xx_0-\beta\alpha_yx_0||=\frac{|\scal{x-y}{x_0}|}{||x_0||}\leq ||x-y||;\label{eq:var_alpha}\\
||y''||=\frac{|\scal{y-x}{y''}|}{||y''||}\leq ||x-y||.\label{eq:var_y_seconde}
\end{gather}
\end{subequations}
As $|Z_j|$ is bounded (by $2$), the desired inequality is true for $||x-y||\geq \sqrt{\epsilon}/2$, provided that $C$ is large enough, so we can assume $||x-y||<\sqrt{\epsilon}/2$, which in particular guarantees that $|\beta|>1/2$.

As
\begin{align*}
\Var(\phase(a_j^*x)-\phase(a_j^*y)|Ax_0)&\leq
\E\left(\left|\phase(a_j^*x)-\phase(a_j^*y)\right|^2\Big|Ax_0\right)\\
&=2\left(1-\Re\left(\E\left(\phase(\overline{a_j^*x})\phase(a_j^*y)|Ax_0\right)\right)\right),
\end{align*}
we only need, in order to prove Equation \eqref{eq:var_reformulation}, to show that, for some constant $C>0$,
\begin{align}\label{eq:var_reformulation2}
1-\Re\left(\E\left(\phase(\overline{a_j^*x})\phase(a_j^*y)|Ax_0\right)\right)
\leq C\left(1+\frac{|a_j^*x_0|^2}{||x_0||^2}\right)||x-y||^2\log(4||x-y||^{-1}).
\end{align}

We have
\begin{align*}
\phase(a_j^*x)&=\phase\left(\frac{a_j^*x'}{||x'||}+\frac{\alpha_x}{||x'||}a_j^*x_0\right);\\
\phase(a_j^*y)&=\phase\left(\frac{a_j^*x'}{||x'||}+\frac{\alpha_y}{\beta||x'||}a_j^*x_0+\frac{1}{\beta||x'||}a_j^*y'' \right)\phase(\beta),
\end{align*}
and $\frac{a_j^*x'}{||x'||}$ is a complex Gaussian random variable with variance $1$, independent from $Ax_0$ and $a_j^*y''$. So
\begin{align*}
&\quad 1-\Re\left(\E\left(\phase(\overline{a_j^*x})\phase(a_j^*y)|Ax_0,a_j^*y''\right)\right)\\
&=1-\frac{1}{\pi}\Re\left(\phase(\beta)\int_\C \phase\left(\overline{z+\frac{\alpha_x}{||x'||}a_j^*x_0}\right)\phase\left(z+\frac{\alpha_y}{\beta||x'||}a_j^*x_0+\frac{1}{\beta||x'||}a_j^*y'' \right)e^{-|z|^2}d^2z\right).
\end{align*}
We upper bound this quantity with the following proposition, proven in Paragraph \ref{sss:controle_G}.
\begin{prop}\label{prop:controle_G}
Let us define the function
\begin{equation*}
\begin{array}{rccc}
G:&\C^2&\to&\C\\
&(a,b)&\to&1-\frac{1}{\pi}\Re\int_\C\phase(\overline{z+a})\phase(z+b)e^{-|z|^2}d^2z.
\end{array}
\end{equation*}
For some constant $c_1>0$, the following inequalities are true:
\begin{align*}
\forall a,b\in\C,\quad\quad
|\Re G(a,b)|&\leq c_1 |a-b|^2\max\left(1,\log\left(|a-b|^{-1}\right)\right),\\
|\Im G(a,b)|&\leq c_1 |a-b|.
\end{align*}
\end{prop}
So
\begin{align*}
&\quad 1-\Re\left(\E\left(\phase(\overline{a_j^*x})\phase(a_j^*y)|Ax_0,a_j^*y''\right)\right)\\
&=1-\Re(\phase(\beta))+\Re(\phase(\beta))\Re G\left(\frac{\alpha_x}{||x'||}a_j^*x_0,\frac{\alpha_y}{\beta||x'||}a_j^*x_0+\frac{1}{\beta||x'||}a_j^*y''\right)\\
&\hskip 2cm-\Im(\phase(\beta))\Im G\left(\frac{\alpha_x}{||x'||}a_j^*x_0,\frac{\alpha_y}{\beta||x'||}a_j^*x_0+\frac{1}{\beta||x'||}a_j^*y''\right)\\
&\leq c_1\left|\frac{\alpha_y-\beta\alpha_x}{\beta||x'||}a_j^*x_0+\frac{1}{\beta||x'||}a_j^*y'' \right|^2\max\left(1,\log\left(\left|\frac{\alpha_y-\beta\alpha_x}{\beta||x'||}a_j^*x_0+\frac{1}{\beta||x'||}a_j^*y''\right|^{-1}\right)\right)\\
&\hskip 2cm+|1-\Re(\phase\beta)| + |\Im(\phase(\beta))|\left|\frac{\alpha_y-\beta\alpha_x}{\beta||x'||}a_j^*x_0+\frac{1}{\beta||x'||}a_j^*y'' \right|\\
&\leq c_1\left(\left|\frac{\alpha_y-\beta\alpha_x}{\beta||x'||}a_j^*x_0\right|+\left|\frac{1}{\beta||x'||}a_j^*y'' \right|\right)^2\max\left(1,\log\left(\left|\frac{\alpha_y-\beta\alpha_x}{\beta||x'||}a_j^*x_0\right|+\left|\frac{1}{\beta||x'||}a_j^*y''\right|\right)^{-1}\right)\\
&\hskip 2cm+2|1-\beta|^2 + |\beta-1|\left(\left|\frac{\alpha_y-\beta\alpha_x}{\beta||x'||}a_j^*x_0\right|+\left|\frac{1}{\beta||x'||}a_j^*y'' \right|\right)\\
&\leq 2c_1\left|\frac{\alpha_y-\beta\alpha_x}{\beta||x'||}a_j^*x_0\right|^2\max\left(1,\log\left(\left|\frac{\alpha_y-\beta\alpha_x}{\beta||x'||}a_j^*x_0\right|^{-1}\right)\right)\\
&\hskip 2cm + 2c_1\left|\frac{1}{\beta||x'||}a_j^*y'' \right|^2\max\left(1,\log\left(\left|\frac{1}{\beta||x'||}a_j^*y''\right|^{-1}\right)\right)\\
&\hskip 2cm+2|1-\beta|^2 + |\beta-1|\left(\left|\frac{\alpha_y-\beta\alpha_x}{\beta||x'||}a_j^*x_0\right|+\left|\frac{1}{\beta||x'||}a_j^*y'' \right|\right)\\
&\overset{(*)}{\leq} 2c_1\left|\frac{\alpha_y-\beta\alpha_x}{\beta||x'||}\right|^2\max(||x_0||,|a_j^*x_0|)^2 \max\left(1,\log\left(\left|\frac{\alpha_y-\beta\alpha_x}{\beta||x'||}.\max(||x_0||,|a_j^*x_0|)\right|^{-1}\right)\right)\\
&\hskip 2cm + 2c_1\left|\frac{1}{\beta||x'||}a_j^*y'' \right|^2\max\left(1,\log\left(\left|\frac{1}{\beta||x'||}a_j^*y''\right|^{-1}\right)\right)\\
&\hskip 2cm+2|1-\beta|^2 + |\beta-1|\left(\frac{\left|\alpha_y-\beta\alpha_x\right|\,||x_0||}{\beta||x'||}\frac{|a_j^*x_0|}{||x_0||}+\left|\frac{1}{\beta||x'||}a_j^*y'' \right|\right)\\
&\leq 2c_1\left|\frac{\alpha_y-\beta\alpha_x}{\beta||x'||}\right|^2\max(||x_0||,|a_j^*x_0|)^2 \max\left(1,\log\left(\left|\frac{\alpha_y-\beta\alpha_x}{\beta||x'||}||x_0||\right|^{-1}\right)\right)\\
&\hskip 2cm + 2c_1\left|\frac{1}{\beta||x'||}a_j^*y'' \right|^2\max\left(1,\log\left(\left|\frac{1}{\beta||x'||}a_j^*y''\right|^{-1}\right)\right)\\
&\hskip 2cm+2\frac{||x-y||^2}{\epsilon} + \frac{||x-y||}{\sqrt{\epsilon}}\left(2\frac{||x-y||}{\sqrt{\epsilon}}\frac{|a_j^*x_0|}{||x_0||}+\frac{2}{\sqrt{\epsilon}}\left|a_j^*y'' \right|\right)\\
&\leq c_2 ||x-y||^2\left(1+\frac{|a_j^*x_0|}{||x_0||}\right)^2\max(1,\log||x-y||^{-1})+ c_2|a_j^*y''|^2\max(1,\log|a_j^*y''|^{-1})\\
&\hskip 2cm + c_2||x-y||\,|a_j^*y''|.
\end{align*}
For $(*)$, we have used the fact that $t\to t^2\max(1,\log(1/t))$ is non-decreasing. For the last two lines, we have used this same fact and Equations \eqref{eq:var_x_prime}, \eqref{eq:var_beta} and \eqref{eq:var_alpha}.

The random variable $a_j^*y''$ is complex and Gaussian, has variance $||y''||^2$ and is independent from $Ax_0$, so, taking the expectation over $a_j^*y''$ then using Equation \eqref{eq:var_y_seconde}, we get:
\begin{align*}
&\quad 1-\Re\left(\E\left(\phase(\overline{a_j^*x})\phase(a_j^*y)|Ax_0\right)\right)\\
&\leq c_2 ||x-y||^2\left(1+\frac{|a_j^*x_0|}{||x_0||}\right)^2\max(1,\log||x-y||^{-1})+ c_3||y''||^2\max(1,\log||y''||^{-1})\\
&\hskip 2cm + c_3||y''||\,||x-y||\\
&\leq c_4 ||x-y||^2\left(1+\frac{|a_j^*x_0|}{||x_0||}\right)^2\max(1,\log||x-y||^{-1}).
\end{align*}
As $||x-y||\leq 2$ (because $x$ and $y$ are unit-normed), this implies Equation \eqref{eq:var_reformulation2} and concludes.
\end{proof}

\subsubsection{Proof of Lemma \ref{lem:esp_exp_bound}\label{sss:esp_exp_bound}}

\begin{lem*}[Lemma \ref{lem:esp_exp_bound}]
Let $Z$ be any real random variable such that $|Z|\leq 2$ with probability $1$. If we set $\sigma^2=\Var(Z)$, then, for any $\lambda\in\R^+$,
\begin{equation*}
\E\left(e^{\lambda (Z-\E(Z))}\right) \leq 1+\frac{\sigma^2}{16}\left(e^{4\lambda}-1-4\lambda\right).
\end{equation*}
\end{lem*}

\begin{proof}[Proof of Lemma \ref{lem:esp_exp_bound}]
Let us define $Z'=Z-\E(Z)$. We have $|Z'|\leq 4$ with probability $1$, $\E(Z')=0$ and $\E(Z'^2)=\sigma^2$. Then,
\begin{align*}
\E\left(e^{\lambda Z'}\right)
&=\E\left(1+\lambda Z' +\sum_{k\geq 2}\frac{\lambda^kZ'^k}{k!}\right)\\
&\leq 1 + \sum_{k\geq 2}\E\left(\frac{\lambda^k Z'^2 4^{k-2}}{k!}\right)\\
&=1+\frac{\sigma^2}{16}\left(e^{4\lambda}-1-4\lambda\right).
\end{align*}
\end{proof}

\subsubsection{Proof of Lemma \ref{lem:eval_lambda_small}\label{sss:eval_lambda_small}}

\begin{lem*}[Lemma \ref{lem:eval_lambda_small}]
There exists a constant $\tilde C>0$ depending only on $\gamma$ and $\epsilon>0$ such that, for any $\lambda\in]0;\frac{1}{40}[$,
\begin{equation*}
\int_1^{\sqrt{\log m}+1}\frac{f_{\lambda}(t)}{t^2}e^{-t^2/4}dt+
m^{-1/2}\int_{\sqrt{\log m}+1}^{m^{1/4}+1}\frac{f_{\lambda}(t)}{t^3}dt+
\frac{1}{m}f_{\lambda}(m^{1/4}+1)
\leq \tilde C \gamma^{-2k}\lambda^2,
\end{equation*}
provided that
\begin{subequations}
\begin{gather}
\Big(\log(\max(1,\gamma^k/\lambda))+1\Big)\left(\frac{\gamma^k}{\lambda}\right)^{4/3}\leq m^{1/2};\tag{\ref{eq:eval_cond1}}\\
\frac{m^{1/2}\lambda \gamma^{-2k}}{1+\log m} \geq 1.\tag{\ref{eq:eval_cond2}}
\end{gather}
\end{subequations}
\end{lem*}

\begin{proof}[Proof of Lemma \ref{lem:eval_lambda_small}]
As we only consider the function $f_\lambda$ on $]1;+\infty[$, we can upper bound it by the slightly simpler expression
\begin{equation*}
\tilde f_\lambda(x)=\log\left(1+\frac{C'\gamma^{-2k}}{8}x^2(e^{4\lambda x^2}-1-4\lambda x^2)\right).
\end{equation*}
Let $X_0$ be the (unique) positive number such that
\begin{align}
\frac{C'\gamma^{-2k}}{8}X_0^2(e^{4\lambda X_0^2}-1-4\lambda X_0^2)&=1.\nonumber\\
\iff\hskip 1cm \lambda X_0^2(e^{4\lambda X_0^2}-1-4\lambda X_0^2)&=\frac{8}{C'}\lambda\gamma^{2k}.\label{eq:eval_def_x0}
\end{align}
The function $\tilde f_\lambda$ satisfies the following inequalities:
\begin{gather*}
\forall x\in\R^+,\quad\quad
\tilde f_\lambda(x)\leq \frac{C'\gamma^{-2k}}{8}x^2(e^{4\lambda x^2}-1-4\lambda x^2);
\end{gather*}
\begin{align*}
\forall x\geq X_0,\quad\quad
\tilde f_\lambda(x)&\leq \log\left(\frac{C'\gamma^{-2k}}{8}x^2(e^{4\lambda x^2}-1-4\lambda x^2)\right)+\log 2\\
&\leq \log\left(\frac{C'\gamma^{-2k}}{8}x^2e^{4\lambda x^2}\right)+\log 2\\
&\leq \log\left(\frac{C'}{4}\right)+2\log x + 4\lambda x^2.
\end{align*}
In particular, if $X_0\leq 2m^{1/4}$,
\begin{align*}
\frac{1}{m}f_{\lambda}(m^{1/4}+1)
&\leq \frac{1}{m}f_{\lambda}(2m^{1/4})\\
&\leq \frac{1}{m}\left(\log\left(\frac{C'}{4}\right)+2\log(2m^{1/4})+16\lambda m^{1/2}\right)\\
&\leq D\left(\frac{\log m+\lambda m^{1/2}}{m}\right)\\
&\overset{\eqref{eq:eval_cond2}}{\leq} \frac{2D\lambda}{m^{1/2}}\\
&\overset{\eqref{eq:eval_cond2}}{\leq} 2D\lambda^2 \gamma^{-2k},
\end{align*}
and if $X_0>2m^{1/4}$, from the definition of $X_0$, we see that
\begin{align*}
\frac{1}{m}f_{\lambda}(m^{1/4}+1)
&\leq \frac{1}{m}\log(1+\frac{C'\gamma^{-2k}}{8}X_0^2(e^{4\lambda X_0^2}-1-4\lambda X_0^2))\\
&\leq \frac{\log 2}{m}\\
&\leq (\log 2)\frac{\lambda^2\gamma^{-2k}}{(m^{1/2}\lambda\gamma^{-2k})^2}\\
&\overset{\eqref{eq:eval_cond2}}{\leq} (\log 2)\lambda^2 \gamma^{-2k}.
\end{align*}
So $\frac{1}{m}f_{\lambda}(m^{1/4}+1)$ is bounded by $\tilde C\gamma^{-2k}\lambda^2$ and we only have to show the same bound for the integral terms.

Using the inequalities we have established over $\tilde f_{\lambda}$,
\begin{align}
\int_1^{\sqrt{\log m}+1}&\frac{f_{\lambda}(t)}{t^2}e^{-t^2/4}dt+
m^{-1/2}\int_{\sqrt{\log m}+1}^{m^{1/4}+1}\frac{f_{\lambda}(t)}{t^3}dt
\nonumber\\
&\leq \frac{C'\gamma^{-2k}}{8}\int_0^{+\infty} (e^{4\lambda t^2}-1-4\lambda t^2)e^{-t^2/4} dt\label{eq:eval_term1}\\
&\quad +m^{-1/2}\frac{C'\gamma^{-2k}}{8}\int_{\sqrt{\log m}+1}^{\max(X_0,\sqrt{\log m}+1)}\frac{1}{t}(e^{4\lambda t^2}-1-4\lambda t^2)dt\label{eq:eval_term2}\\
&\quad +m^{-1/2}\int_{\min(m^{1/4}+1,\max(X_0,\sqrt{\log m}+1))}^{m^{1/4}+1}\frac{1}{t^3}\left(\log\left(\frac{C'}{4}\right)+ 2\log t + 4\lambda t^2\right)dt\label{eq:eval_term3}.
\end{align}
We separately study each of the three right-side terms.

For Term \eqref{eq:eval_term1}, we can do an exact computation, taking into account the fact that $\lambda\leq 1/40$:
\begin{align*}
\mbox{\eqref{eq:eval_term1}}
&=\frac{C'\sqrt{\pi}}{8}\gamma^{-2k} \left(\frac{1}{\sqrt{1-16\lambda}}-1-8\lambda\right)\\
&\leq \frac{C''\sqrt{\pi}}{8}\gamma^{-2k}\lambda^2.
\end{align*}
For Term \eqref{eq:eval_term2}, if $X_0<\sqrt{\log m}+1$, then it is zero. Otherwise, $X_0\geq \sqrt{\log m}+1$ and
\begin{align}
\mbox{\eqref{eq:eval_term2}}
&\leq m^{-1/2}\frac{C'\gamma^{-2k}}{8}\int_0^{X_0}\frac{1}{t}(e^{4\lambda t^2}-1-4\lambda t^2)dt\nonumber\\
&=m^{-1/2}\frac{C'\gamma^{-2k}}{8}\int_0^{2\sqrt{\lambda}X_0}\frac{1}{t}(e^{t^2}-1-t^2)dt.\label{eq:eval_term2_upper_bound}
\end{align}
When $\frac{8}{C'}\lambda\gamma^{2k}\leq 1$, we check from the definition of $X_0$ (Equation \eqref{eq:eval_def_x0}) that $\lambda X_0^2\leq 1$, so $2\sqrt{\lambda}X_0\leq 2$ and
\begin{align*}
\mbox{\eqref{eq:eval_term2}}
&\leq m^{-1/2}\frac{C''\gamma^{-2k}}{8}\int_0^{2\sqrt{\lambda}X_0}t^3 dt.\\
&= m^{-1/2} \frac{C''\gamma^{-2k}}{2} \lambda^2X_0^4.
\end{align*}
From Equation \eqref{eq:eval_def_x0} again, we see that, as $\lambda X_0^2\leq 1$,
\begin{align}
\frac{8}{C'}\lambda\gamma^{2k}&\geq C'''(\lambda X_0^2)^3;\nonumber\\
\Rightarrow\hskip 1cm
\left(\frac{8}{C'C'''}\right)^{2/3}\frac{\gamma^{4k/3}}{\lambda^{4/3}}&\geq  X_0^4.
\label{eq:eval_x0_small_equiv}
\end{align}
From Condition \eqref{eq:eval_cond1}, we know that $\left(\frac{\gamma^k}{\lambda}\right)^{4/3}\leq m^{1/2}$, so
\begin{align*}
\mbox{\eqref{eq:eval_term2}}
&\leq m^{-1/2}C''''\gamma^{-2k}\lambda^2\left(\frac{\gamma^{k}}{\lambda}\right)^{4/3} \leq C'''' \gamma^{-2k}\lambda^2.
\end{align*}
On the other hand, when $\frac{8}{C'}\lambda\gamma^{2k}> 1$, $2\sqrt{\lambda}X_0$ is bounded away from zero. We evaluate the integral in Equation \eqref{eq:eval_term2_upper_bound} by parts:
\begin{align*}
\mbox{\eqref{eq:eval_term2}}
&\leq m^{-1/2}\frac{C'\gamma^{-2k}}{8}\int_0^{2\sqrt{\lambda}X_0}\frac{1}{t}(e^{t^2}-1-t^2)dt\\
&\leq m^{-1/2} C''\gamma^{-2k}\frac{1}{\lambda X_0^2} e^{4\lambda X_0^2}.
\end{align*}
From Equation \eqref{eq:eval_def_x0}, we can compute that, when $2\sqrt{\lambda}X_0$ is bounded away from zero,
\begin{equation*}
\frac{1}{X_0^2} e^{4\lambda X_0^2} \leq C'''\frac{\lambda^2\gamma^{2k}}{\left(1+\log\left(8\lambda\gamma^{2k}/C'\right)\right)^2},
\end{equation*}
which yields, together with Condition \eqref{eq:eval_cond2}:
\begin{align*}
\mbox{\eqref{eq:eval_term2}}
&\leq \frac{m^{-1/2}C''C''' \lambda}{\left(1+\log\left(8\lambda\gamma^{2k}/C'\right)\right)^2}
\leq m^{-1/2}C''C'''\lambda
\leq C''C''' \lambda^2\gamma^{-2k}.
\end{align*}
Finally, we consider the last term. When $X_0\geq m^{1/4}+1$, it is zero, so we only have to consider the case where $X_0< m^{1/4}+1$.
\begin{align}
\mbox{\eqref{eq:eval_term3}}
&\leq m^{-1/2}C''\int_{\max(X_0,1)}^{m^{1/4}+1}\frac{1+\log t+\lambda t^2}{t^3}dt\nonumber\\
&= m^{-1/2}C''\left[-\frac{3+2\log t}{4t^2} +\lambda\log t \right]_{\max(X_0,1)}^{m^{1/4}+1}\nonumber\\
&\leq m^{-1/2}C''\left( \frac{1+\log(\max(1,X_0))}{X_0^2} + \lambda \log\left(\frac{m^{1/4}+1}{\max(X_0,1)}\right)\right)\nonumber\\
&\leq m^{-1/2}C''\left( \frac{1+\log(\max(1,X_0))}{X_0^2} + \lambda \log\left(m^{1/4}+1\right)\right).\label{eq:eval_term3_simple}
\end{align}
The second part of Equation \eqref{eq:eval_term3_simple} can be upper bounded as desired, thanks to Condition \eqref{eq:eval_cond2}:
\begin{align}
m^{-1/2}C'' \lambda \log(m^{1/4}+1)
&\leq m^{-1/2}C''\lambda(1+\log m)\nonumber\\
&\leq C'' \lambda^2\gamma^{-2k}.\label{eq:eval_term3_1}
\end{align}
For the first part, let us distinguish the cases $\frac{8}{C'}\lambda\gamma^{2k}\leq 1$ and $\frac{8}{C'}\lambda\gamma^{2k}>1$.

In the case where $\frac{8}{C'}\lambda\gamma^{2k}\leq 1$, we see (in a similar way as in Equation \eqref{eq:eval_x0_small_equiv}) that
\begin{equation*}
c_1 \frac{\gamma^{k/3}}{\lambda^{1/3}}\leq X_0\leq c_2 \frac{\gamma^{k/3}}{\lambda^{1/3}},
\end{equation*}
so
\begin{align}
m^{-1/2}\left(\frac{1+\log(\max(1,X_0))}{X_0^2}\right)
&\leq m^{-1/2}C'''\left(1+\log(\max(1,\gamma^k/\lambda))\right)\frac{\lambda^{2/3}}{\gamma^{2k/3}}\nonumber\\
&= m^{-1/2}C'''\left(1+\log(\max(1,\gamma^k/\lambda))\right)\lambda^2\gamma^{-2k} \left(\frac{\gamma^k}{\lambda}\right)^{4/3}\nonumber\\
&\leq C''' \lambda^2 \gamma^{-2k}.\label{eq:eval_term3_2}
\end{align}
For the last equality, we have used Condition \eqref{eq:eval_cond1}.

In the case where $\frac{8}{C'}\lambda\gamma^{2k}>1$, as we have already seen, $\sqrt{\lambda}X_0$ is bounded away from $0$, so, for some constant $C'''>0$,
\begin{equation*}
X_0\geq C''' \lambda^{-1/2},
\end{equation*}
which implies
\begin{align}
m^{-1/2}\left(\frac{1+\log(\max(1,X_0))}{X_0^2}\right)
&\leq C'''' m^{-1/2}\lambda \left(1+\log(\max(1,\lambda^{-1/2}))\right)\nonumber\\
&\leq m^{-1/2}C'''' \lambda \left(1+\log\left(\max\left(1,\frac{1}{\lambda\gamma^{-2k}}\right)\right)\right).\label{eq:eval_term3_final}
\end{align}
From Condition \eqref{eq:eval_cond2}, we know that
\begin{gather*}
\lambda \gamma^{-2k}\geq m^{-1/2}(1+\log m);\\
\Rightarrow\quad\quad
\frac{1+\log\left(\max\left(1,\frac{1}{\lambda\gamma^{-2k}}\right)\right)}{\lambda\gamma^{-2k}}\leq m^{1/2} \frac{1+\log\left(\frac{m^{1/2}}{1+\log m}\right)}{1+\log m}\leq m^{1/2}.
\end{gather*}
We plug this into Equation \eqref{eq:eval_term3_final} and get
\begin{equation}\label{eq:eval_term3_3}
m^{-1/2}\left(\frac{1+\log(\max(1,X_0))}{X_0^2}\right)
\leq C'''' \lambda^2\gamma^{-2k}.
\end{equation}
Finally, we combine Equations \eqref{eq:eval_term3_1}, \eqref{eq:eval_term3_2} and \eqref{eq:eval_term3_3}. With Equation \eqref{eq:eval_term3_simple}, they show that
\begin{equation*}
\mbox{\eqref{eq:eval_term3}}\leq \mathcal{C}\lambda^2\gamma^{-2k},
\end{equation*}
for some constant $\mathcal{C}>0$.
\end{proof}

\subsubsection{Proof of Lemma \ref{lem:maj_esp}\label{sss:maj_esp}}

\begin{lem*}[Lemma \ref{lem:maj_esp}]
There exist a constant $C>0$ depending only on $\epsilon$ such that, for any fixed unit-normed $x,y$ such that
\begin{equation*}
|\scal{x_0}{x}|\leq (1-\epsilon)||x_0||\,||x||\quad\mbox{and}\quad
|\scal{x_0}{y}|\leq (1-\epsilon)||x_0||\,||y||,
\end{equation*}
and any $j=1,\dots,m$,
\begin{equation*}
|\E(Z_j|a_j^*x_0)| \leq C \min\left(1,||x-y||\left(1+\frac{|a_j^*x_0|}{||x_0||}\right)\right).
\end{equation*}
\end{lem*}

\begin{proof}[Proof of Lemma \ref{lem:maj_esp}]
As $Z_j=\phase(a_j^*x)\phase(\overline{a_j^*x_0})-\phase(a_j^*y)\phase(\overline{a_j^*x_0})$,
\begin{equation*}
|\E(Z_j|a_j^*x_0)| = |\E(\phase(a_j^* x)|a_j^*x_0)-\E(\phase(a_j^*y)|a_j^*x_0)|.
\end{equation*}
As in the proof of Lemma \ref{lem:var_bound}, we write
\begin{equation*}
x=\alpha_x x_0+x'\mbox{ and }y=\alpha_yx_0+\beta x'+y'',
\end{equation*}
where $\alpha_x,\alpha_y,\beta$ are complex numbers and $x',y''\in\C^n$ satisfy $\scal{x'}{x_0}=\scal{y''}{x_0}=\scal{x'}{y''}=0$. We recall Equations \eqref{eq:var_x_prime} to \eqref{eq:var_y_seconde}:
\begin{gather*}
||x'||\geq \sqrt{\epsilon(2-\epsilon)}\geq \sqrt{\epsilon};\tag{\ref{eq:var_x_prime}}\\
|\beta-1|=\frac{|\scal{y-x}{x'}|}{||x'||^2}\leq \frac{1}{\sqrt{\epsilon}}||y-x||;\tag{\ref{eq:var_beta}}
\\
||\alpha_xx_0-\alpha_yx_0||=\frac{|\scal{x-y}{x_0}|}{||x_0||}\leq ||x-y||;\tag{\ref{eq:var_alpha}}
\\
||y''||=\frac{|\scal{y-x}{y''}|}{||y''||}\leq ||x-y||.\tag{\ref{eq:var_y_seconde}}
\end{gather*}
The variable $Z_j$ is bounded in modulus by $2$, so the desired inequality holds for $||x-y||\geq\sqrt{\epsilon}/2$ if we choose $C\geq 4/\sqrt{\epsilon}$. In what follows, we assume that $||x-y||<\sqrt{\epsilon}/2$, which notably guarantees that $|\beta|>1/2$.

The random variables $a_j^*x_0,a_j^*x'$ and $a_j^*y''$ are independent complex Gaussians, with respective variances $||x_0||^2,||x'||^2,||y''||^2$. Thus,
\begin{align}
\E(\phase(a_j^*x)|a_j^*x_0)
&=\E\left(\phase\left(\frac{\alpha_x}{||x'||}a_j^*x_0 + \frac{a_j^* x'}{||x'||}\right)\Bigg| a_j^*x_0\right)\nonumber\\
&=\frac{1}{\pi}\int_\C \phase\left(\frac{\alpha_x}{||x'||}a_j^*x_0 + z \right)e^{-{|z|^2}}d^2z,\label{eq:int_phase_x}
\end{align}
and similarly,
\begin{align}
\E(\phase(a_j^*y)|a_j^*x_0,a_j^*y'')
&=\frac{\phase(\beta)}{\pi}\int_\C \phase\left(\frac{\alpha_y}{\beta ||x'||}a_j^*x_0 +\frac{a_j^* y''}{\beta ||x'||} + z \right)e^{-{|z|^2}}d^2z.
\label{eq:int_phase_y}
\end{align}
The function
\begin{equation*}
a\in\C\to
\frac{1}{\pi}\int_\C\phase(a+z)e^{-{|z|^2}}d^2z
=\frac{1}{\pi}\int_\C\phase(z)e^{-{|z-a|^2}}d^2z
\end{equation*}
is Lipschitz (as can be seen by derivation under the integral sign). If we denote by $D>0$ the Lipschitz constant, Equations \eqref{eq:int_phase_x} and \eqref{eq:int_phase_y} imply that
\begin{align*}
&\left|\E(\phase(a_j^*x)|a_j^*x_0)-\overline{\phase(\beta)}\E(\phase(a_j^*y)|a_j^*x_0,a_j^*y'')\right|\\
&\quad\quad \leq D\left|\left|\frac{\alpha_x}{||x'||}a_j^*x_0-\left(\frac{\alpha_y}{\beta ||x'||}a_j^*x_0 +\frac{a_j^* y''}{\beta ||x'||}\right)\right|\right|\\
&\quad\quad\leq D \left(||x-y|| \frac{|a_j^*x_0|}{||x_0||}
\left(\frac{1}{\sqrt{\epsilon}}+\frac{2}{\epsilon}\right)
+\frac{2}{\sqrt{\epsilon}} |a_j^* y''|\right)\\
&\quad\quad\leq D ||x-y|| \left( \frac{|a_j^*x_0|}{||x_0||}
\left(\frac{1}{\sqrt{\epsilon}}+\frac{2}{\epsilon}\right)
+\frac{2}{\sqrt{\epsilon}} \frac{|a_j^* y''|}{||y''||}\right).
\end{align*}
For the last two inequalities, we have used Equations \eqref{eq:var_x_prime} to \eqref{eq:var_y_seconde}. We finally take the expectation over $a_j^* y''$; by triangular inequality,
\begin{align*}
&\left|\E(\phase(a_j^*x)|a_j^*x_0)-\overline{\phase(\beta)}\E(\phase(a_j^*y)|a_j^*x_0)\right|\\
&\quad\quad
\leq D ||x-y|| \left( \frac{|a_j^*x_0|}{||x_0||}
\left(\frac{1}{\sqrt{\epsilon}}+\frac{2}{\epsilon}\right)
+\sqrt{\frac{\pi}{\epsilon}} \right)\\
&\quad\quad\leq C ||x-y||\left(1+\frac{|a_j^*x_0|}{||x_0||}\right),
\end{align*}
when $C>0$ is large enough. Additionally,
\begin{align*}
&\left|\E(\phase(a_j^*y)|a_j^*x_0)-\overline{\phase(\beta)}\E(\phase(a_j^*y)|a_j^*x_0)\right|\\
&\quad\quad \leq |1-\beta|\\
&\quad\quad \leq 2\frac{|1-\beta|}{|\beta|}\\
&\quad\quad \leq \frac{4}{\sqrt{\epsilon}}||y-x||.
\end{align*}
So by triangular inequality,
\begin{align*}
&\left|\E(\phase(a_j^*x)|a_j^*x_0)-\E(\phase(a_j^*y)|a_j^*x_0)\right|\\
&\quad\quad\leq C' ||x-y||\left(1+\frac{|a_j^*x_0|}{||x_0||}\right),
\end{align*}

We also have
\begin{equation*}
|\E(Z_j|a_j^*x_0)|\leq C,
\end{equation*}
for any constant $C\geq 2$, so
\begin{equation*}
|\E(Z_j|a_j^*x_0)|\leq C\min\left(1,||x-y||\left(1+\frac{|a_j^*x_0|}{||x_0||}\right)\right)
\end{equation*}
when $C>0$ is large enough.
\end{proof}

\subsubsection{Proof of Lemma \ref{lem:esp_exp_bound2}\label{sss:esp_exp_bound2}}

\begin{lem*}[Lemma \ref{lem:esp_exp_bound2}]
There exist constants $c,C'>0$, that depend only on $\gamma$ and $\epsilon$, such that, for any $\lambda\in[-c;c]$,
\begin{align*}
\log\left(\E\left(e^{\lambda\Re\left(|a_j^*x_0|^2\E(Z_j|Ax_0)-\E(|a_j^*x_0|^2Z_j)\right)}\right) \right)
&\leq C' \lambda^2 \gamma^{-2k},\\
\mbox{and }
\log\left(\E\left(e^{\lambda\Im\left(|a_j^*x_0|^2\E(Z_j|Ax_0)-\E(|a_j^*x_0|^2Z_j)\right)}\right) \right)
&\leq C' \lambda^2 \gamma^{-2k}.
\end{align*}
\end{lem*}

\begin{proof}[Proof of Lemma \ref{lem:esp_exp_bound2}]
We only prove the first inequality; the proof of the second one is identical. We assume that $\lambda$ is positive; the same reasoning holds with minor modifications when $\lambda$ is negative.

To simplify the notations, we set
\begin{equation*}
\mathcal{Z}_j=|a_j^*x_0|^2\E(Z_j|a_j^*x_0)
-\E(|a_j^*x_0|^2Z_j).
\end{equation*}
We recall from Equation \eqref{eq:maj_esp_gamma} that
\begin{equation*}
|a_j^*x_0|^2 |\E(Z_j|a_j^*x_0)|\leq 2C |a_j^*x_0|^2 \min(1,\gamma^{-k}(1+|a_j^*x_0|)).
\end{equation*}
As a consequence, because $a_j^*x_0$ is a complex Gaussian random variable with variance $||x_0||^2=1$,
\begin{align*}
|\E(|a_j^*x_0|^2Z_j)|
&=|\E\left(|a_j^*x_0|^2 \E(Z_j|a_j^*x_0)\right)|\\
&\leq 2C \E\left( |a_j^*x_0|^2 \min(1,\gamma^{-k}(1+|a_j^*x_0|))\right)\\
&\leq 2C\gamma^{-k}\E(|a_j^*x_0|^2(1+|a_j^*x_0|))\\
&=2C\left(1+\frac{3}{4}\sqrt{\pi}\right)\gamma^{-k}.
\end{align*}
Combining this with Equation \eqref{eq:maj_esp_gamma}, we see that there exists a constant $C''>0$ such that
\begin{align*}
|\mathcal{Z}_j|=
\Big||a_j^*x_0|^2\E(Z_j|a_j^*x_0)
-\E(|a_j^*x_0|^2Z_j)\Big| \leq C''(1+|a_j^*x_0|)^2\min(1,\gamma^{-k}(1+|a_j^*x_0|)).
\end{align*}
Let us note that, because $\E(\mathcal{Z}_j)=0$,
\begin{align*}
\log(\E(e^{\lambda\Re(\mathcal{Z}_j)}))
&\leq \E(e^{\lambda\Re(\mathcal{Z}_j)})-1\\
&=\E(e^{\lambda\Re(\mathcal{Z}_j)}-\lambda\Re(\mathcal{Z}_j)-1).
\end{align*}
The function $f:x\to e^{\lambda x}-\lambda x-1$ is non-decreasing over $\R^+$, and satisfies $f(x)\leq f(|x|)$ for any $x\in\R$. Hence,
\begin{align}
&\log(\E(e^{\lambda\Re(\mathcal{Z}_j)}))\nonumber\\
\leq \E&\left(
e^{\lambda C''(1+|a_j^*x_0|)^2\min(1,\gamma^{-k}(1+|a_j^*x_0|))}
-\lambda C''(1+|a_j^*x_0|)^2\min(1,\gamma^{-k}(1+|a_j^*x_0|))
-1\right)\nonumber\\
&=\frac{1}{\pi}\int_\C \left(e^{\lambda C''(1+|z|)^2\min(1,\gamma^{-k}(1+|z|))}-\lambda C''(1+|z|)^2\min(1,\gamma^{-k}(1+|z|))-1\right)e^{-|z|^2}d^2z\nonumber\\
&=2\int_0^{+\infty} \left(e^{\lambda C''(1+r)^2\min(1,\gamma^{-k}(1+r))}-\lambda C''(1+r)^2\min(1,\gamma^{-k}(1+r))-1\right)re^{-r^2}dr\nonumber\\
&=2\int_1^{+\infty} \left(e^{\lambda C''r^2\min(1,\gamma^{-k}r)}-\lambda C''r^2\min(1,\gamma^{-k}r)-1\right)(r-1)e^{-(r-1)^2}dr\nonumber\\
&\leq C'''\int_0^{+\infty}\left(e^{\lambda C''r^2\min(1,\gamma^{-k}r)}-\lambda C''r^2\min(1,\gamma^{-k}r)-1\right)e^{-r^2/2}dr\nonumber\\
&= C'''\int_0^{\gamma^k}\left(e^{\lambda C''r^3\gamma^{-k}}-\lambda C''r^3\gamma^{-k}-1\right)e^{-r^2/2}dr\label{eq:esp_term1}\\
&\hskip 2cm
+ C'''\int_{\gamma^k}^{+\infty}\left(e^{\lambda C''r^2}-\lambda C''r^2-1\right)e^{-r^2/2}dr.\label{eq:esp_term2}
\end{align}
We need to show that both components \eqref{eq:esp_term1} and \eqref{eq:esp_term2} are upper bounded by $C'\lambda^2\gamma^{-2k}$ for some constant $C'>0$ sufficiently large, provided that $|\lambda|\leq c$ for some constant $c>0$.

For Term \eqref{eq:esp_term1}, we use the fact that, when $r\leq C''^{-1/3}\gamma^{k/3}\lambda^{-1/3}$,
\begin{gather*}
\lambda C''r^3\gamma^{-k}\leq 1;\\
\Rightarrow\quad
 e^{\lambda C''r^3\gamma^{-k}}-\lambda C''r^3\gamma^{-k}-1
\leq (\lambda C''r^3\gamma^{-k})^2.
\end{gather*}
It yields:
\begin{align}
\mbox{\eqref{eq:esp_term1}}
&\leq C'''\int_0^{\min(\gamma^k,C''^{-1/3}\gamma^{k/3}\lambda^{-1/3})} (\lambda C''r^3\gamma^{-k})^2e^{-r^2/2}dr\nonumber\\
&\hskip 2cm+
C'''\int_{\min(\gamma^k,C''^{-1/3}\gamma^{k/3}\lambda^{-1/3})}^{\gamma^k}e^{\lambda C''r^3\gamma^{-k}}e^{-r^2/2}dr\nonumber\\
&\leq C''' C''^2\lambda^2\gamma^{-2k}\int_0^{+\infty}r^6e^{-r^2/2}dr
+
C'''\int_{\min(\gamma^k,C''^{-1/3}\gamma^{k/3}\lambda^{-1/3})}^{\gamma^k}e^{\lambda C''r^3\gamma^{-k}}e^{-r^2/2}dr.\label{eq:esp_term1_tmp}
\end{align}
For the second term of this sum, if we assume that
\begin{equation*}
\lambda<\frac{1}{4C''},
\end{equation*}
we have
\begin{align*}
\int_{\min(\gamma^k,C''^{-1/3}\gamma^{k/3}\lambda^{-1/3})}^{\gamma^k}e^{\lambda C''r^3\gamma^{-k}}e^{-r^2/2}dr&= \int_{\min(\gamma^k,C''^{-1/3}\gamma^{k/3}\lambda^{-1/3})}^{\gamma^k}e^{r^2\left(\lambda C''r\gamma^{-k}-\frac{1}{2}\right)}dr\\
&\leq \int_{\min(\gamma^k,C''^{-1/3}\gamma^{k/3}\lambda^{-1/3})}^{\gamma^k}e^{r^2\left(\lambda C''-\frac{1}{2}\right)}dr\\
&\leq \int_{\min(\gamma^k,C''^{-1/3}\gamma^{k/3}\lambda^{-1/3})}^{\gamma^k}e^{-r^2/4}dr\\
&\leq \int_{C''^{-1/3}\gamma^{k/3}\lambda^{-1/3}}^{+\infty}e^{-r^2/4}dr\\
&\leq C''' \frac{e^{-\left(C''^{-1/3}\gamma^{k/3}\lambda^{-1/3}\right)^2/4}}{C''^{-1/3}\gamma^{k/3}\lambda^{-1/3}}\\
&\leq C'''' \lambda^2\gamma^{-2k}.
\end{align*}
For the last inequality, we have used the fact that there exists a constant $D>0$ such that $e^{-x}\leq D x^{-5/2}$, for all $x>0$.

Plugging this into Equation \eqref{eq:esp_term1_tmp}, we get
\begin{equation}\label{eq:esp_term1_final}
\mbox{\eqref{eq:esp_term1}}\leq C'\lambda^2\gamma^{-2k}.
\end{equation}
For Term \eqref{eq:esp_term2}, still under the assumption $\lambda<1/(4C'')$,
\begin{align}
\mbox{\eqref{eq:esp_term2}}&\leq
C''' \int_{\gamma^k}^{\max(\gamma^k,(\lambda C'')^{-1/2})}(\lambda C''r^2)^2 e^{-r^2/2}dr
+C''' \int_{\max(\gamma^k,(\lambda C'')^{-1/2})}^{+\infty}e^{\lambda C''r^2}e^{-r^2/2}dr\nonumber\\
&\leq C'''C''^2 \lambda^2\int_{\gamma^k}^{+\infty}r^4e^{-r^2/2}dr
+C'''\int_{\max(\gamma^k,(\lambda C'')^{-1/2})}^{+\infty}e^{-r^2/4}dr\nonumber\\
&\leq C''''\left(\lambda^2 \gamma^{3k}e^{-\gamma^{2k}/2}+\min\left(\frac{e^{-\gamma^{2k}/4}}{\gamma^k},\sqrt{\lambda C''}e^{-1/(4\lambda C'')} \right)\right)\nonumber\\
&\overset{(*)}{\leq} \tilde C(\lambda^2 \gamma^{-2k}+\min(\gamma^{-4k},\lambda^4))\nonumber\\
&\leq 2\tilde C\lambda^2\gamma^{-2k}.\label{eq:esp_term2_final}
\end{align}
For Inequality $(*)$, we have used the existence of a constant $D$ such that, for all $k$, $\gamma^{3k}e^{-\gamma^{2k}/2}\leq D\gamma^{-2k}$ and, for all $\lambda$ staying in a bounded interval, $\sqrt{\lambda}e^{-1/(4\lambda C'')}\leq D\lambda^4$.

Equations \eqref{eq:esp_term1_final} and \eqref{eq:esp_term2_final}, combined with Equation \eqref{eq:esp_term2}, show that, when $\lambda<1/(4C'')$,
\begin{equation*}
\log(\E(e^{\lambda\Re(\mathcal{Z}_j)}))
\leq C' \lambda^2\gamma^{-2k},
\end{equation*}
for some constant $C'>0$ that depends only upon $\gamma$.
\end{proof}

\subsubsection{Proof of Lemma \ref{lem:min_f}\label{sss:min_f}}

\begin{lem*}[Lemma \ref{lem:min_f}]
For any $t\in\R^+$, we set
\begin{equation*}
f(t)=\E\left(\overline{Z_1}|Z_1|\phase\left(Z_1+t Z_2\right)\right).
\end{equation*}
The function $f$ is real-valued. For any $\gamma>0$, there exist $\delta>0$ such that
\begin{equation*}
\forall t\in[\gamma;+\infty[,\quad\quad
f(t)\geq \frac{1+\delta}{\sqrt{1+t^2}}.
\end{equation*}
\end{lem*}

\begin{proof}[Proof of Lemma \ref{lem:min_f}]
As $(\overline{Z}_1,\overline{Z}_2)$ has the same distribution as $(Z_1,Z_2)$,
\begin{equation*}
\forall t\in\R^+,\quad\quad
f(t) = \E(Z_1|Z_1|\phase(\overline{Z}_1+t\overline{Z}_2))=\overline{f(t)},
\end{equation*}
so $f(t)$ is a real number, for any $t\geq 0$.

Let us now show the second part of the result. We have
\begin{align*}
f(t)
&=\frac{1}{\pi^2}\int_{\C^2}\overline{z_1}|z_1|\phase(z_1+tz_2)e^{-|z_1|^2}e^{-|z_2|^2}d^2z_1d^2z_2\\
&=\frac{1}{\pi^2}\int_{\C^2}\overline{y_1}|y_1|\phase(y_2)e^{-|y_1|^2}e^{-|y_2-y_1/t|^2}d^2y_1d^2y_2\\
&=\frac{1}{\pi^2}\int_{\C^2}\overline{y_1}|y_1|\phase(y_2)e^{-|y_1|^2}e^{-|y_2|^2}
\left(\sum_{k\geq 0}\frac{1}{k!}\left(y_2\overline{y_1}/t+y_1\overline{y_2}/t-|y_1|^2/t^2\right)^k\right)
d^2y_1d^2y_2\\
&=\frac{1}{\pi^2}\sum_k
\sum_{k_1+k_2\leq k}\frac{(-1)^{k-(k_1+k_2)}}{k_1!k_2!(k-k_1-k_2)!}
\frac{1}{t^{2k-(k_1+k_2)}}\times\\
&\hskip 4cm
\int_{\C^2}y_1^{k-k_1} \overline{y_1}^{k-k_2+1}|y_1|y_2^{k_1}\overline{y_2}^{k_2}\phase(y_2)e^{-|y_1|^2}e^{-|y_2|^2}
d^2y_1d^2y_2\\
&\overset{(*)}{=}\frac{1}{\pi^2}\sum_k
\sum_{2k_1+1\leq k}\frac{(-1)^{k-2k_1-1}}{k_1!(k_1+1)!(k-2k_1-1)!}\frac{1}{t^{2k-2k_1-1}}
\int_{\C^2}|y_1|^{2(k-k_1)+1}|y_2|^{2k_1+1}e^{-|y_1|^2}e^{-|y_2|^2}
d^2y_1d^2y_2\\
&\overset{(**)}{=}\sum_{l}\frac{1}{t^{2l+1}}
\left(\frac{1}{\pi}\int_{\C^2}|y|^{2l+3}e^{-|y|^2}d^2y\right)\sum_{k_1\leq l}
\frac{(-1)^{k_1+l}}{k_1!(k_1+1)!(l-k_1)!}
\left(\frac{1}{\pi}\int_{\C^2}|y|^{2k_1+1}e^{-|y|^2}d^2y\right)\\
&\overset{(***)}{=}\sum_{l}\frac{\pi}{t^{2l+1}}
(l+1)(l+2)\binom{2(l+2)}{l+2}
\sum_{k_1\leq l}
(-1)^{k_1+l}\binom{2(k_1+1)}{k_1+1}
\binom{l}{k_1}
2^{-2(l+k_1+3)}
.
\end{align*}
Equality $(*)$ is true because the integral is zero if $k_2\ne k_1+1$, as can be seen with a change of variable $y_1\to uy_1$ for $u$ a complex number of modulus $1$. Equality $(**)$ is obtained by setting $l=k-k_1-1$. Equality $(***)$ is a consequence of the following inequality, valid for all odd $K$:
\begin{equation*}
\frac{1}{\pi}\int_\C |y|^Ke^{-|y|^2}d^2y=\sqrt{\pi}2^{-K}\frac{K!}{\left(\frac{K-1}{2}\right)!}.
\end{equation*}

This reasoning is valid only for $t$ large enough; for small values of $t$, the series may not converge. We see that, in order for all the involved series to be absolutely convergent, it is enough that the following one is absolutely convergent:
\begin{align*}
&\sum_k
\sum_{k_1+k_2\leq k}\frac{1}{k_1!k_2!(k-k_1-k_2)!}
\frac{1}{t^{2k-(k_1+k_2)}}\times\\
&\hskip 3cm
\int_{\C^2}\left|y_1^{k-k_1} \overline{y_1}^{k-k_2+1}|y_1|y_2^{k_1}\overline{y_2}^{k_2}\phase(y_2)e^{-|y_1|^2}e^{-|y_2|^2}\right|
d^2y_1d^2y_2.
\end{align*}
When $t\geq 2$, for example, this series can be upper bounded by
\begin{align*}
&\sum_k
\sum_{k_1+k_2\leq k}\frac{1}{k_1!k_2!(k-k_1-k_2)!}
\frac{1}{t^{2k-(k_1+k_2)}}
\int_{\C^2}|y_1|^{2k-(k_1+k_2)+2}|y_2|^{k_1+k_2}e^{-|y_1|^2}e^{-|y_2|^2}
d^2y_1d^2y_2\\
&=\int_\C\left(\sum_k \frac{1}{k!}|y_1|^2
\left(\frac{|y_1||y_2|}{t}+\frac{|y_1||y_2|}{t}+\frac{|y_1|^2}{t^2}\right)^k
e^{-|y_1|^2}e^{-|y_2|^2}\right)
d^2y_1d^2y_2\\
&=\int_\C |y_1|^2\exp\left(-|y_1|^2-|y_2|^2+2\frac{|y_1||y_2|}{t}+\frac{|y_1|^2}{t^2}
\right)d^2y_1d^2y_2\\
&\leq \int_\C |y_1|^2\exp\left(-\left(1-\frac{1}{t}-\frac{1}{t^2}\right)|y_1|^2-\left(1-\frac{1}{t}\right)|y_2|^2
\right)d^2y_1d^2y_2\\
&\leq \int_\C |y_1|^2\exp\left(-\frac{1}{4}|y_1|^2-\frac{1}{2}|y_2|^2
\right)d^2y_1d^2y_2<+\infty.
\end{align*}
So the series converge.

For any $l\in\N,k_1\in\{0,\dots,l\}$, we set
\begin{gather*}
c_{l,k_1}=\binom{2(k_1+1)}{k_1+1}
\binom{l}{k_1}
2^{-2(l+k_1+3)};\\
C_l=(l+1)(l+2)\binom{2(l+2)}{l+2}\sum_{k_1\leq l}(-1)^{k_1+l}c_{l,k_1}.
\end{gather*}
The series $\sum_{k_1\leq l}(-1)^{k_1+l}c_{l,k_1}$ is alternating, and we can check that
\begin{align*}
\max_{k_1\leq l}|c_{l,k_1}|=|c_{l,[l/2]}|.
\end{align*}
This allows us to see that
\begin{align*}
\left|\sum_{k_1\leq l}(-1)^{k_1+l}c_{l,k_1}\right|
&\leq \max_{k_1\leq l}|c_{l,k_1}|\\
&=c_{l,[l/2]}\\
&\leq \frac{1}{8\pi}\frac{1}{l2^l},
\end{align*}
We do not derive the second inequality in full detail: the principle is to compute the upper limit of the sequence $(c_{l,[l/2]}l2^l)_{l\in\N}$ with Sterling's formula, then to study the variations of this sequence, to show that it is bounded by its upper limit. Hence, using this inequality and the fact that, for any $s$, $\binom{2s}{s}\leq 2^{2s}/\sqrt{\pi s}$, we see that
\begin{equation*}
|C_l|\leq \frac{l+1}{l}.\sqrt{\frac{l+2}{\pi}}\frac{2^{l+1}}{\pi}.
\end{equation*}
So for any $l\geq 3$,
\begin{align*}
|C_l|&\leq \frac{l2^{l+1}}{\pi^{3/2}}.
\end{align*}
We explicitly compute $C_0,C_1,C_2$:
\begin{equation*}
C_0=\frac{3}{8};\quad\quad
C_1=-\frac{15}{64};\quad\quad
C_2=\frac{105}{512}.
\end{equation*}
Hence, combining the previous results, for any $t\geq 2$,
\begin{align*}
f(t)&=\pi \sum_{l\geq 0}\frac{C_l}{t^{2l+1}}\\
&\geq \pi \sum_{l=0}^2\frac{C_l}{t^{2l+1}}-
-\frac{1}{\sqrt{\pi}}\sum_{l=3}^{+\infty}\frac{l2^{l+1}}{t^{2l+1}}\\
&=\pi\left(\frac{3}{8}\frac{1}{t}-\frac{15}{64}\frac{1}{t^3}+\frac{105}{512}\frac{1}{t^5}\right)
-\frac{16}{\sqrt{\pi}t^5}\frac{3t^2-4}{t^4-4}.
\end{align*}
From here, we can easily verify with a computer that, for any $t>2.5$,
\begin{equation}\label{eq:t_large}
f(t)>\frac{1.05}{\sqrt{1+t^2}}.
\end{equation}

Let us now show that $f(t)>(1+t^2)^{-1/2}$ for any $t\in]0;2.5]$. If we set
\begin{equation*}
Y_1 = \frac{-tZ_1+Z_2}{\sqrt{1+t^2}}\quad\mbox{and}\quad
Y_2 = \frac{Z_1+tZ_2}{\sqrt{1+t^2}},
\end{equation*}
we see that $Y_1$ and $Y_2$ are independent Gaussian random variables, with variance $1$, and that
\begin{equation*}
f(t)=\frac{1}{1+t^2}\E\left((\overline{Y_2-tY_1})|Y_2-tY_1|\phase(Y_2)
\right).
\end{equation*}
We set
\begin{equation*}
g(t)=\E\left((\overline{Y_2-tY_1})|Y_2-tY_1|\phase(Y_2)\right).
\end{equation*}
A straight computation yields
\begin{gather}
g'(t)=\E\left(
\left(-\frac{3}{2}\overline{Y}_1|Y_2-tY_1|-\frac{1}{2}Y_1\frac{(\overline{Y_2-tY_1})^2}{|Y_2-tY_1|}\right)\phase(Y_2)
\right);
\label{eq:g_prime}\\
g''(t)=\E\left(
\left(\frac{3}{2}|Y_1|^2\phase(\overline{Y_2-tY_1})
+\frac{3}{4}\overline{Y}_1^2\phase(Y_2-tY_1)\right.\right.\nonumber\\
\left.\left.\hskip 6cm-\frac{1}{4} Y_1^2\phase(\overline{Y_2-tY_1})^3
\right)
\phase(Y_2)\right).\label{eq:g_seconde}
\end{gather}
For any $u,t>0$, we see by triangular inequality that
\begin{align*}
&\left|\phase(Y_2-tY_1)-\phase(Y_2-uY_1)\right|\\
&\quad\quad\leq 
\left|\frac{Y_2-tY_1}{|Y_2-tY_1|}-\frac{Y_2-uY_1}{|Y_2-tY_1|}\right|
+\left|\frac{Y_2-uY_1}{|Y_2-tY_1|}-\frac{Y_2-uY_1}{|Y_2-uY_1|}\right|\\
&\quad\quad\leq 2|t-u|\frac{|Y_1|}{|Y_2-tY_1|},
\end{align*}
which also implies
\begin{align*}
\left|\phase(Y_2-tY_1)^3-\phase(Y_2-uY_1)^3\right|\leq 6|t-u|\frac{|Y_1|}{|Y_2-tY_1|}.
\end{align*}
Plugging this into Equation \eqref{eq:g_seconde}:
\begin{align*}
|g''(t)-g''(u)|
&\leq 6|t-u| \E\left(\frac{|Y_1|^3}{|Y_2-tY_1|}\right)\\
&\leq 6|t-u| \E\left(\frac{|Y_1|^3}{|Y_2|}\right)\\
&=\frac{9}{2}\pi |t-u|.
\end{align*}
We deduce from here that, for any $u,t$ such that $0\leq u\leq t$,
\begin{align*}
g(t)&=g(u)+(t-u)g'(u)+\frac{(t-u)^2}{2}g''(u)+\int_u^t(t-s)(g''(s)-g''(u))ds\\
&\geq g(u)+(t-u)g'(u)+\frac{(t-u)^2}{2}g''(u)-\frac{9}{2}\pi\int_u^t(t-s)(s-u)ds\\
&= g(u)+(t-u)g'(u)+\frac{(t-u)^2}{2}g''(u)-\frac{9}{2}\pi \frac{(t-u)^3}{6}.
\end{align*}
In $u=0$, Equations \eqref{eq:g_prime} and \eqref{eq:g_seconde} allow us to compute $g'(0)$ and $g''(0)$: we have $g'(0)=0$ and $g''(0)=\frac{3}{2}$. Thus, from the last equation, for any $t\geq 0$,
\begin{equation*}
g(t)\geq 1 + \frac{3}{4}t^2-\frac{3}{4}\pi t^3,
\end{equation*}
which allows us to verify (with a computer) that, for any $t\in]0;0.1]$,
\begin{equation*}
f(t)=\frac{g(t)}{1+t^2}\geq \frac{1+\frac{3}{4}t^2-\frac{3}{4}\pi t^3}{1+t^2}>\frac{1}{\sqrt{1+t^2}}.
\end{equation*}
We can apply the same reasoning to values of $u$ that are different from $0$. Equations \eqref{eq:g_prime} and \eqref{eq:g_seconde} do not appear to have a simple analytic expression when $u\ne 0$. They can however be computed with a computer. We do so for $u=0.1,0.2,0.3,0.4,\dots,2.4$, and successively show that the previous inequality also holds on the intervals $[0.1;0.2],[0.2,0.3],\dots,[2.7,2.5]$. 

We have thus proven that $f(t)>(1+t^2)^{-1/2}$ for any $t\in]0;2.5]$. By compacity (as $f$ is continuous), it means that there exists $\delta>0$ such that
\begin{equation*}
\forall t\in[\gamma;2.5],\quad\quad
f(t)\geq \frac{1+\delta}{\sqrt{1+t^2}}.
\end{equation*}
Together with Equation \eqref{eq:t_large}, this implies the lemma.

\end{proof}

\subsubsection{Proof of Proposition \ref{prop:controle_G}\label{sss:controle_G}}

\begin{prop*}[Proposition \ref{prop:controle_G}]
Let us define the function
\begin{equation*}
\begin{array}{rccc}
G:&\C^2&\to&\C\\
&(a,b)&\to&1-\frac{1}{\pi}\Re\int_\C\phase(\overline{z+a})\phase(z+b)e^{-|z|^2}d^2z.
\end{array}
\end{equation*}
For some constant $c_1>0$, the following inequalities are true:
\begin{align*}
\forall a,b\in\C,\quad\quad
|\Re G(a,b)|&\leq c_1 |a-b|^2\max\left(1,\log\left(|a-b|^{-1}\right)\right),\\
|\Im G(a,b)|&\leq c_1 |a-b|.
\end{align*}
\end{prop*}

\begin{proof}[Proof of Proposition \ref{prop:controle_G}]
\begin{align*}
|\Re G(a,b)|&=\frac{1}{\pi}\left|\Re\int_\C\left(1-\phase(\overline{z+a})\phase(z+b)\right)e^{-|z|^2}d^2z \right|\\
&=\frac{1}{\pi}\left|\Re\int_\C\left(1-\phase\left(1+\frac{b-a}{z+a}\right)\right)e^{-|z|^2}d^2z \right|\\
&\leq \frac{1}{\pi}\left|\Re\int_{|z+a|>2|b-a|}  \left(1-\phase\left(1+\frac{b-a}{z+a}\right)\right)e^{-|z|^2}d^2z \right|\\
&\quad\quad+\frac{2}{\pi}\left|\int_{|z+a|\leq 2|b-a|} e^{-|z|^2}d^2z \right|\\
&\overset{(a)}{\leq} \frac{c_2}{\pi}\left|\int_{|z+a|>2|b-a|}  \left|\frac{b-a}{z+a}\right|^2 e^{-|z|^2}d^2z \right|
+\frac{2}{\pi}\left|\int_{|z+a|\leq 2|b-a|} 1d^2z \right|\\
&\leq \frac{c_2}{\pi}\left|\int_{1\geq |z+a|>2|b-a|}  \left|\frac{b-a}{z+a}\right|^2 e^{-|z|^2}d^2z \right| 
+ \frac{c_2}{\pi}\left|\int_{|z+a|>1}  \left|\frac{b-a}{z+a}\right|^2 e^{-|z|^2}d^2z \right|
+8|b-a|^2\\
& \leq \frac{c_2}{\pi}\left|\int_{1\geq |z+a|>2|b-a|}  \left|\frac{b-a}{z+a}\right|^2d^2z \right| 
+ \frac{c_2}{\pi}|b-a|^2\left|\int_{\C}  e^{-|z|^2}d^2z \right|
+8|b-a|^2\\
& \leq  c_1 |b-a|^2\max\left(1,\log\left(|b-a|^{-1}\right)\right).
\end{align*}
Inequality (a) comes from the fact that $z\in\C\to \Re(1-\phase(1+z))\in\R$ is a $\mathcal{C}^\infty$ function on $\{z\in\C,|z|<1/2\}$, and its derivative in $0$ is $0$ (because the function reaches a local minimum at this point). So by compacity, there exists a constant $c_2>0$ such that, for any $z$ verifying $|z|<1/2$,
\begin{equation*}
\left|\Re(1-\phase(1+z))\right|\leq c_2 |z|^2.
\end{equation*}

The proof of the second inequality is identical, except that we bound $\left|\Im\left(1-\phase\left(1+\frac{b-a}{z-a}\right)\right)\right|$ by $c_2\left|\frac{b-a}{z+a}\right|$ on the set $\{z,|z+a|>2|b-a|\}$.
\end{proof}

\subsection{Proof of Lemma \ref{lem:inside_net}\label{ss:inside_net}}

\begin{lem*}[Lemma \ref{lem:inside_net}]
For any $c>0$, there exist $C_1,C_2,C_3>0$ such that, with probability at least
\begin{equation*}
1-C_1\exp(-C_2 m^{1/8}),
\end{equation*}
the following property holds for any unit-normed $x,y\in\C^n$, when $m\geq 2n^2$:
\begin{equation*}
|\scal{Ax_0}{b\odot\phase(Ax)}-\scal{Ax_0}{b\odot\phase(Ay)}|
\leq C_3||x_0||^2 nm^{1/4}
\quad\mbox{if }||x-y||\leq cm^{-7/2}.
\end{equation*}
\end{lem*}

\begin{proof}[Proof of Lemma \ref{lem:inside_net}]
We write
\begin{align*}
&|\scal{Ax_0}{b\odot\phase(Ax)}-\scal{Ax_0}{b\odot\phase(Ay)}|\\
&\quad =\left| \sum_{i=1}^m\overline{(Ax_0)_i}|(Ax_0)_i|\left(\phase((Ax)_i)-\phase((Ay)_i)\right)\right|\\
&\quad\leq \sum_{i=1}^m |(Ax_0)_i|^2\left|\phase((Ax)_i)-\phase((Ay)_i)\right|\\
&\quad\leq 2\sum_{i=1}^m |(Ax_0)_i|^2\min\left(1,\frac{\left|(Ax)_i-(Ay)_i\right|}{|(Ax)_i|}\right)\\
&\quad\leq 2\,\sum_{\mathclap{|(Ax)_i|\leq 1/m^2}} |(Ax_0)_i|^2
+2\,\sum_{\mathclap{|(Ax)_i|> 1/m^2}} |(Ax_0)_i|^2\frac{\left|(Ax)_i-(Ay)_i\right|}{|(Ax)_i|}\\
&\quad\leq 2\,\sum_{\mathclap{|(Ax)_i|\leq 1/m^2}} |(Ax_0)_i|^2
+2 m^2 |||A|||^3\, ||x-y|| \,||x_0||^2.
\end{align*}

From Proposition \ref{prop:davidson}, if $m\geq 2n^2\geq 2n$, $|||A|||\leq 3\sqrt{m}$ with probability at least
\begin{equation*}
1-2\exp(-m).
\end{equation*}
On this event, we can deduce from the previous inequality that, for any $x,y$ such that $||x-y||\leq cm^{-7/2}$,
\begin{align*}
&|\scal{Ax_0}{b\odot\phase(Ax)}-\scal{Ax_0}{b\odot\phase(Ay)}|\\
&\quad \leq 2\,\sum_{\mathclap{|(Ax)_i|\leq 1/m^2}} |(Ax_0)_i|^2
+ 54 m^{7/2}||x-y||\,||x_0||^2\\
&\quad \leq 2\,\sum_{\mathclap{|(Ax)_i|\leq 1/m^2}} |(Ax_0)_i|^2
+ 54 c ||x_0||^2.
\end{align*}

To upper bound the first term of the right-hand side, we use two auxiliary lemmas, proven in Paragraphs \ref{sss:Card_Ix} and \ref{sss:A_I}.

\begin{lem}\label{lem:Card_Ix}
For any unit-normed $x\in\C^n$, we define $I_x=\left\{i\in\{1,\dots,m\},|(Ax)_i|\leq \frac{1}{m^2}\right\}$. There exist $C_1,C_2>0$ such that, when $m\geq n^2$, the event
\begin{equation*}
\Big( \forall x, \Card I_x< nm^{1/8}\Big)
\end{equation*}
has probability at least
\begin{equation*}
1-C_1\exp(-C_2m^{1/8}).
\end{equation*}
\end{lem}

\begin{lem}\label{lem:A_I}
There exist $C>0$ such that, with probability at least
\begin{equation*}
1-\exp(-nm^{1/4}),
\end{equation*}
for any $I\subset\{1,\dots,m\}$ such that $\Card I \leq nm^{1/8}$,
\begin{equation*}
\sum_{i\in I}|(Ax_0)_i|^2 \leq C ||x_0||^2 nm^{1/4}.
\end{equation*}
\end{lem}

We combine these lemmas with the last inequality. This proves that, with probability at least
\begin{equation*}
1-C_1\exp(-C_2m^{1/8}),
\end{equation*}
(for some constants $C_1,C_2>0$ possibly different from the ones introduced in Lemma \ref{lem:Card_Ix}),
\begin{align*}
|\scal{Ax_0}{b\odot\phase(Ax)}-\scal{Ax_0}{b\odot\phase(Ay)}|
&\leq ||x_0||^2 \left(2Cnm^{1/4} +54c \right),\\
&\leq C_3||x_0||^2 nm^{1/4},
\end{align*}
for all $x,y$ verifying $||x-y||\leq cm^{-7/2}$.

\end{proof}

\subsubsection{Proof of Lemma \ref{lem:Card_Ix}\label{sss:Card_Ix}}

\begin{lem*}[Lemma \ref{lem:Card_Ix}]
For any unit-normed $x\in\C^n$, we define $I_x=\left\{i\in\{1,\dots,m\},|(Ax)_i|\leq \frac{1}{m^2}\right\}$. There exist $C_1,C_2>0$ such that, when $m\geq n^2$, the event
\begin{equation*}
\Big( \forall x, \Card I_x< nm^{1/8}\Big)
\end{equation*}
has probability at least
\begin{equation*}
1-C_1\exp(-C_2m^{1/8}).
\end{equation*}
\end{lem*}

\begin{proof}[Proof of Lemma \ref{lem:Card_Ix}]
Let $\mathcal{M}\geq 1$ be temporarily fixed.

For any $n,m$, let $\mathcal{N}_{n,m}$ be a $\frac{1}{\mathcal{M}m^2}$-net of the unit sphere of $\C^n$. From \citep[Lemma 5.2]{vershynin}, there is one of cardinality at most
\begin{equation*}
\left(1+4\mathcal{M} m^2\right)^{2n}\leq (5\mathcal{M}m^2)^{2n}.
\end{equation*}
We define two events:
\begin{gather*}
\mathcal{E}_1=\left\{\forall x\in\mathcal{N}_{n,m},\Card\left\{i,|(Ax)_i|\leq \frac{2}{m^2}\right\}<nm^{1/8} \right\};\\
\mathcal{E}_2=\{\forall i\in\{1,\dots,m\},||a_i^*||\leq \mathcal{M} \}.
\end{gather*}
(We recall that $a_i^*$ is the $i$-th line of $A$.)

On the intersection of these two elements, we have $\Card I_x<nm^{1/8}$ for any unit-normed $x\in\C^n$. Indeed, for any such $x$, there exists $x'\in\mathcal{N}_{n,m}$ such that $||x-x'||\leq 1/(\mathcal{M}m^2)$. For any $i\in I_x$,
\begin{align*}
|(Ax')_i| &\leq |(Ax)_i| + |a_i^*(x-x')|\\
&\leq \frac{1}{m^2} + ||a_i^*||\,||x-x'||\\
&\leq \frac{2}{m^2}.
\end{align*}
As a consequence, $I_x\subset\left\{i,|(Ax')_i|\leq \frac{2}{m^2}\right\}$, whose cardinality is strictly less than $nm^{1/8}$ because we are on event $\mathcal{E}_1$.

Let us find lower bounds on the probabilities of $\mathcal{E}_1$ and $\mathcal{E}_2$.

For any $x\in\mathcal{N}_{n,m}$, for any $i=1,\dots,m$,
\begin{equation*}
P\left(|(Ax)_i|\leq\frac{2}{m^2}\right)= 1-e^{-\frac{4}{m^4}}\leq \frac{4}{m^4},
\end{equation*}
because $(Ax)_i$ is a complex Gaussian random variable with variance $1$. So by Hoeffding's inequality, for $x$ fixed,
\begin{align*}
P\left(\Card\left\{i,|(Ax)_i|\leq \frac{2}{m^2}\right\}\geq nm^{1/8} \right)
&=P\left(\sum_{i=1}^m 1_{|(Ax)_i|\leq 2/m^2}\geq nm^{1/8}\right)\\
&\leq P\left(\sum_{i=1}^m 1_{|(Ax)_i|\leq 2/m^2}\geq m\E\left(1_{|(Ax)_1|\leq 2/m^2}\right) + \left(nm^{1/8}-\frac{4}{m^3}\right)\right)\\
&\leq \exp\left(-\frac{4}{m^3}h\left(\frac{m^{3+1/8}n}{4}-1\right)\right),
\end{align*}
where $h$ is the function $t\to (1+t)\log(1+t)-t$.

We simplify:
\begin{align*}
P\left(\Card\left\{i,|(Ax)_i|\leq \frac{2}{m^2}\right\}\geq 2n \right)
&\leq \exp\left(-nm^{1/8} \log(m^{3+1/8}n/4)+nm^{1/8}-\frac{4}{m^3}\right)\\
&\leq \exp\left(-nm^{1/8} \left(\log(m^{3+1/8}n)-3\right) \right).
\end{align*}
Finally, as the cardinality of $\mathcal{N}_{n,m}$ is at most $(5\mathcal{M}m^2)^{2n}$,
\begin{align}
P(\mathcal{E}_1)
&\geq 1 - (5\mathcal{M}m^2)^{2n}e^{-nm^{1/8} \left(\log(m^{3+1/8}n)-3\right)}\nonumber\\
&=1 - \exp\left(-nm^{1/8} \left(\log(m^{3+1/8}n)-3\right) + 2n\log(5\mathcal{M}m^2))\right).\label{eq:PE1}
\end{align}

Let us now consider $\mathcal{E}_2$. For any $i$, $a_i^*$ is a random vector with $n$ independent random complex Gaussian coordinates, of variance $1$. Gaussian measure concentration results (see for example \citep[Proposition 2.2]{barvinok}) imply that, for any $\delta>0$,
\begin{align*}
P\left(||a_i^*||> \sqrt{n+\delta} \right)
&\leq \left(1+\frac{\delta}{n}\right)^ne^{-\delta}.
\end{align*}
For $\delta = \mathcal{M}^2-n$, we get
\begin{align*}
P\left(||a_i^*||> \mathcal{M} \right)
&\leq \left(\frac{\mathcal{M}^2}{n}\right)^n e^{-\left(\mathcal{M}^2-n\right)}\\
&\leq 3 \mathcal{M}^{2n}e^{-\mathcal{M}^2}.
\end{align*}
As a consequence,
\begin{equation}\label{eq:PE2}
P(\mathcal{E}_2)\geq 1- 3m\mathcal{M}^{2n}e^{-\mathcal{M}^2}.
\end{equation}
We can take, for example, $\mathcal{M}=\sqrt{m}$. We evaluate Equations \eqref{eq:PE1} and \eqref{eq:PE2} for this value of $\mathcal{M}$ and get, when $m\geq n^2$,
\begin{equation*}
P(\mathcal{E}_1\cap\mathcal{E}_2)
\geq 1 - C_1e^{-C_2 m^{1/8}}.
\end{equation*}

\end{proof}

\subsubsection{Proof of Lemma \ref{lem:A_I}\label{sss:A_I}}

\begin{lem*}[\ref{lem:A_I}]
There exist $C>0$ such that, with probability at least
\begin{equation*}
1-\exp(-nm^{1/4}),
\end{equation*}
for any $I\subset\{1,\dots,m\}$ such that $\Card I \leq nm^{1/8}$,
\begin{equation*}
\sum_{i\in I}|(Ax_0)_i|^2 \leq C ||x_0||^2 n m^{1/4}.
\end{equation*}
\end{lem*}

\begin{proof}[Proof of Lemma \ref{lem:A_I}]
By homogeneity, we can assume $||x_0||=1$.

The random variables $(Ax_0)_1,\dots,(Ax_0)_m$ are independent and (complex) Gaussian with variance $1$. Hence, by Bernstein's inequality for subexponential variables, there exist a constant $c>0$ such that, for any $t>0$, and for any fixed $I\subset\{1,\dots,m\}$,
\begin{equation*}
P\left(\sum_{i\in I}|(Ax_0)_i|^2 \geq \Card I +t  \right)
\leq \exp\left(-c\min\left(t,\frac{t^2}{\Card I}\right)\right).
\end{equation*}
In particular, if $\Card I = nm^{1/8}$,
\begin{equation*}
P\left(\sum_{i\in I}|(Ax_0)_i|^2 \geq nm^{1/8} + \frac{2}{c}nm^{1/4}  \right)
\leq \exp\left(-2nm^{1/4}\min\left(1,\frac{2}{c}m^{1/8}\right)\right).
\end{equation*}
So as soon as $m$ is large enough,
\begin{equation*}
P\left(\sum_{i\in I}|(Ax_0)_i|^2 \geq \frac{3}{c}nm^{1/4}  \right)
\leq \exp(-2n m^{1/8}).
\end{equation*}
There are less than $m^{nm^{1/8}}=e^{nm^{1/8}\log m}$ subsets of $\{1,\dots,m\}$ with cardinality $nm^{1/8}$, so
\begin{equation*}
P\left(\exists I\mbox{ s.t. }\Card I\leq nm^{1/8},\sum_{i\in I}|(Ax_0)_i|^2 \geq \frac{3}{c}n m^{1/4}  \right)
\leq \exp(-n m^{1/4}).
\end{equation*}
\end{proof}

\bibliographystyle{plainnat}
\bibliography{../bib_articles.bib,../bib_proceedings.bib,../bib_livres.bib,../bib_misc.bib}

\begin{thebibliography}{37}
\providecommand{\natexlab}[1]{#1}
\providecommand{\url}[1]{\texttt{#1}}
\expandafter\ifx\csname urlstyle\endcsname\relax
  \providecommand{\doi}[1]{doi: #1}\else
  \providecommand{\doi}{doi: \begingroup \urlstyle{rm}\Url}\fi

\bibitem[Balan et~al.(2006)Balan, Casazza, and Edidin]{balan}
R.~Balan, P.~Casazza, and D.~Edidin.
\newblock On signal reconstruction without noisy phase.
\newblock \emph{Applied and Computational Harmonic Analysis}, 20:\penalty0
  345--356, 2006.

\bibitem[Bandeira et~al.(2016)Bandeira, Boumal, and
  Voroninski]{bandeira_low_rank}
A.~S. Bandeira, N.~Boumal, and V.~Voroninski.
\newblock On the low-rank approach for semidefinite programs arising in
  synchronization and community detection.
\newblock In \emph{Proceedings of the Conference on Computational Learning
  Theory}, 2016.

\bibitem[Barvinok(2005)]{barvinok}
A.~Barvinok.
\newblock Math 710: measure concentration.
\newblock Lecture notes, 2005.
\newblock http://www.math.lsa.umich.edu/~barvinok/total710.pdf.

\bibitem[Bauschke et~al.(2002)Bauschke, Combettes, and Luke]{bauschke}
H.~H. Bauschke, P.~L. Combettes, and D.~R. Luke.
\newblock Phase retrieval, error reduction algorithm, and fienup variants: a
  view from convex optimization.
\newblock \emph{Journal of the Optical Society of America}, 19:\penalty0
  1334--1345, 2002.

\bibitem[Bhojanapalli et~al.(2016)Bhojanapalli, Neyshabur, and
  Srebo]{bhojanapalli}
S.~Bhojanapalli, B.~Neyshabur, and N.~Srebo.
\newblock Global optimality of local search for low rank matrix recovery.
\newblock In \emph{Advances in Neural Information Processing Systems 29}, 2016.

\bibitem[Cand\`es and Li(2014)]{candes_li}
E.~J. Cand\`es and X.~Li.
\newblock Solving quadratic equations via phaselift when there are about as
  many equations as unknowns.
\newblock \emph{Foundations of Computational Mathematics}, 14\penalty0
  (5):\penalty0 1017--1026, 2014.

\bibitem[Cand\`es et~al.(2013)Cand\`es, Strohmer, and Voroninski]{candes2}
E.~J. Cand\`es, T.~Strohmer, and V.~Voroninski.
\newblock Phaselift: exact and stable signal recovery from magnitude
  measurements via convex programming.
\newblock \emph{Communications in Pure and Applied Mathematics}, 66\penalty0
  (8):\penalty0 1241--1274, 2013.

\bibitem[Cand\`es et~al.(2015)Cand\`es, Li, and Soltanolkotabi]{candes_li2}
E.~J. Cand\`es, X.~Li, and M.~Soltanolkotabi.
\newblock Phase retrieval from coded diffraction patterns.
\newblock \emph{Applied and Computational Harmonic Analysis}, 39\penalty0
  (2):\penalty0 277--299, 2015.

\bibitem[Candès et~al.(2015)Candès, Li, and Soltanolkotabi]{candes_wirtinger}
E.~J. Candès, X.~Li, and M.~Soltanolkotabi.
\newblock Phase retrieval via wirtinger flow: theory and algorithms.
\newblock \emph{IEEE Transactions of Information Theory}, 61\penalty0
  (4):\penalty0 1985--2007, 2015.

\bibitem[Chai et~al.(2011)Chai, Moscoso, and Papanicolaou]{chai}
A.~Chai, M.~Moscoso, and G.~Papanicolaou.
\newblock Array imaging using intensity-only measurements.
\newblock \emph{Inverse Problems}, 27\penalty0 (1), 2011.

\bibitem[Chen et~al.(2016)Chen, Fannjiang, and Liu]{chen_fannjiang}
P.~Chen, A.~Fannjiang, and G.-R. Liu.
\newblock Phase retrieval with one or two diffraction patterns by alternating
  projection with null initialization.
\newblock \emph{preprint}, 2016.
\newblock http://arxiv.org/abs/1510.07379.

\bibitem[Chen and Candès(2015)]{candes_wirtinger2}
Y.~Chen and E.~J. Candès.
\newblock Solving random quadratic systems of equations is nearly as easy as
  solving linear systems.
\newblock \emph{To appear in Communications on Pure and Applied Mathematics},
  2015.

\bibitem[Conca et~al.(2015)Conca, Edidin, Hering, and Vinzant]{conca}
A.~Conca, D.~Edidin, M.~Hering, and C.~Vinzant.
\newblock Algebraic characterization of injectivity in phase retrieval.
\newblock \emph{Applied and Computational Harmonic Analysis}, 32\penalty0
  (2):\penalty0 346--356, 2015.

\bibitem[Dasgupta and Gupta(2003)]{dasgupta}
S.~Dasgupta and A.~Gupta.
\newblock An elementary proof of a theorem of {J}ohnson and {L}indenstrauss.
\newblock \emph{Random Structures and Algorithms}, 22\penalty0 (1):\penalty0
  60--65, 2003.

\bibitem[Davidson and Szarek(2001)]{davidson}
K.~R. Davidson and S.~J. Szarek.
\newblock Local operator theory, random matrices and {B}anach spaces.
\newblock In W.~B. Johnson and J.~Lindenstrauss, editors, \emph{Handbook of the
  geometry of Banach spaces, volume 1}, pages 317--366. Elsevier, 2001.

\bibitem[Drusvyatskiy et~al.(2015)Drusvyatskiy, Ioffe, and Lewis]{drusvyatskiy}
D.~Drusvyatskiy, A.~D. Ioffe, and A.~S. Lewis.
\newblock Transversality and alternating projections for nonconvex sets.
\newblock \emph{Foundations of Computational Mathematics}, 15\penalty0
  (6):\penalty0 1637--1651, 2015.

\bibitem[Fickus et~al.(2014)Fickus, Mixon, Nelson, and Wang]{fickus}
M.~Fickus, D.~G. Mixon, A.~A. Nelson, and Y.~Wang.
\newblock Phase retrieval from very few measurements.
\newblock \emph{Linear Algebra and its Applications}, 449:\penalty0 475--499,
  2014.

\bibitem[Fienup(1982)]{fienup}
J.~R. Fienup.
\newblock Phase retrieval algorithms: a comparison.
\newblock \emph{Applied Optics}, 21\penalty0 (15):\penalty0 2758--2769, 1982.

\bibitem[Gao and Xu(2016)]{gao_xu}
B.~Gao and Z.~Xu.
\newblock Gauss-{N}ewton method for phase retrieval.
\newblock \emph{preprint}, 2016.
\newblock http://arxiv.org/abs/1606.08135.

\bibitem[Ge et~al.(2016)Ge, Lee, and Ma]{ge_lee_ma}
R.~Ge, J.~D. Lee, and T.~Ma.
\newblock Matrix completion has no spurious local minimum.
\newblock \emph{preprint}, 2016.
\newblock https://arxiv.org/abs/1605.07272.

\bibitem[Gerchberg and Saxton(1972)]{gerchberg}
R.~Gerchberg and W.~Saxton.
\newblock A practical algorithm for the determination of phase from image and
  diffraction plane pictures.
\newblock \emph{Optik}, 35:\penalty0 237--246, 1972.

\bibitem[Gross et~al.(2015)Gross, Krahmer, and Kueng]{gross}
D.~Gross, F.~Krahmer, and R.~Kueng.
\newblock A partial derandomization of {P}hase{L}ift using spherical designs.
\newblock \emph{Journal of Fourier Analysis and Applications}, 21\penalty0
  (2):\penalty0 229--266, 2015.

\bibitem[Kolte and Özgür(2016)]{kolte}
R.~Kolte and A.~Özgür.
\newblock Phase retrieval via incremental truncated {W}irtinger flow.
\newblock \emph{preprint}, 2016.
\newblock arxiv.org/abs/1606.03196.

\bibitem[Lewis et~al.(2009)Lewis, Luke, and Malick]{lewis}
A.~S. Lewis, D.~R. Luke, and J.~Malick.
\newblock Local linear convergence for alternating and averaged nonconvex
  projections.
\newblock \emph{Foundations of Computational Mathematics}, 9\penalty0
  (4):\penalty0 485--513, 2009.

\bibitem[Netrapalli et~al.(2013)Netrapalli, Jain, and Sanghavi]{netrapalli}
P.~Netrapalli, P.~Jain, and S.~Sanghavi.
\newblock Phase retrieval using alternating minimization.
\newblock In \emph{Advances in Neural Information Processing Systems 26}, pages
  1796--2804, 2013.

\bibitem[Noll and Rondepierre(2016)]{noll}
D.~Noll and A.~Rondepierre.
\newblock On local convergence of the method of alternating projections.
\newblock \emph{Foundations of Computational Mathematics}, 16\penalty0
  (2):\penalty0 425--455, 2016.

\bibitem[Schechtman et~al.(2015)Schechtman, Eldar, Cohen, Chapman, Miao, and
  Segev]{schechtman}
Y.~Schechtman, Y.~C. Eldar, O.~Cohen, H.~N. Chapman, J.~Miao, and M.~Segev.
\newblock Phase retrieval with application to optical imaging: a contemporary
  overview.
\newblock \emph{IEEE Signal processing magazine}, 32\penalty0 (3):\penalty0
  87--109, 2015.

\bibitem[Soltanolkotabi(2014)]{soltanolkotabi}
M.~Soltanolkotabi.
\newblock \emph{Algorithms and theory for clustering and nonconvex quadratic
  programming}.
\newblock PhD thesis, Stanford University, 2014.

\bibitem[Sun and Smith(2012)]{sun_smith}
D.~L. Sun and J.~O. Smith.
\newblock Estimating a signal from a magnitude spectrogram via convex
  optimization.
\newblock In \emph{Audio Engineering Society 133rd Convention}, 2012.

\bibitem[Sun et~al.(2016)Sun, Qu, and Wright]{sun_qu_wright}
J.~Sun, Q.~Qu, and J.~Wright.
\newblock A geometric analysis of phase retrieval.
\newblock \emph{preprint}, 2016.
\newblock http://arxiv.org/abs/1602.06664.

\bibitem[Sun and Luo(2015)]{sun_luo}
R.~Sun and Z.-Q. Luo.
\newblock Guaranteed matrix completion via non-convex factorization.
\newblock \emph{to appear in IEEE Transactions on Information Theory}, 2015.

\bibitem[Tu et~al.(2016)Tu, Boczar, Simchowitz, Soltanolkotabi, and Recht]{tu}
Stephen Tu, Ross Boczar, Max Simchowitz, Mahdi Soltanolkotabi, and Ben Recht.
\newblock Low-rank solutions of linear matrix equations via procrustes flow.
\newblock In \emph{Proceedings of the 33nd International Conference on Machine
  Learning}, pages 964--973, 2016.

\bibitem[Vershynin(2012)]{vershynin}
R.~Vershynin.
\newblock Introduction to the non-asymptotic analysis of random matrices.
\newblock In Y.~Eldar and G.~Kutyniok, editors, \emph{Compressed sensing,
  theory and applications}, pages 210--268. Cambridge University Press, 2012.

\bibitem[Waldspurger et~al.(2015)Waldspurger, d'Aspremont, and Mallat]{maxcut}
I.~Waldspurger, A.~d'Aspremont, and S.~Mallat.
\newblock Phase recovery, maxcut and complex semidefinite programming.
\newblock \emph{Mathematical Programming}, 149\penalty0 (1-2):\penalty0 47--81,
  2015.

\bibitem[Wang et~al.(2016)Wang, Giannakis, and Eldar]{wang}
G.~Wang, G.~B. Giannakis, and Y.~C. Eldar.
\newblock Solving random systems of quadratic equations via truncated
  generalized gradient flow.
\newblock In \emph{Advances in Neural Information Processing Systems 29}, 2016.

\bibitem[White et~al.(2015)White, Sanghavi, and Ward]{white}
C.~D. White, S.~Sanghavi, and R.~Ward.
\newblock The local convexity of solving systems of quadratic equations.
\newblock \emph{preprint}, 2015.
\newblock http://arxiv.org/abs/1506.07868.

\bibitem[Zhang and Liang(2016)]{zhang}
H.~Zhang and Y.~Liang.
\newblock Reshaped {W}irtinger flow for solving quadratic systems of equations.
\newblock In \emph{Advances in Neural Information Processing Systems 29}, 2016.

\end{thebibliography}

\end{document}